\documentclass[a4paper,twoside,11pt]{amsart}
\usepackage[english]{babel}
\usepackage[T1]{fontenc}
\usepackage{amsmath}
\usepackage{amsthm}
\usepackage{amssymb}
\usepackage{mathrsfs}
\usepackage{hyperref}
\usepackage{xcolor}
\usepackage{mathtools}
\usepackage{tikz-cd}
\usepackage{fancybox}
\usepackage{amscd}
\usepackage{float}
\usepackage{stmaryrd}
\usepackage{xfrac}

\usepackage{enumerate}

\usepackage[matrix,arrow,cmtip]{xy} 
\SelectTips{cm}{}
\newdir{(}{{}*!/-5pt/@^{(}}
\entrymodifiers={+!!<0pt,\fontdimen22\textfont2>}

\makeatletter
\newtheorem*{rep@theorem}{\rep@title}
\newcommand{\newreptheorem}[2]{%
	\newenvironment{rep#1}[1]{%
		\def\rep@title{#2 \ref{##1}}%
		\begin{rep@theorem}}%
		{\end{rep@theorem}}}
\makeatother

\theoremstyle{plain}
\newtheorem{theorem}{Theorem}[section]
\newreptheorem{theorem}{Theorem}
\newtheorem*{theorem*}{Theorem}
\newtheorem{lemma}[theorem]{Lemma}
\newtheorem{proposition}[theorem]{Proposition}
\newtheorem{corollary}[theorem]{Corollary}
\theoremstyle{definition}
\newtheorem{definition}[theorem]{Definition}
\newtheorem{example}[theorem]{Example}
\newtheorem{notation}[theorem]{Notation}
\theoremstyle{remark}
\newtheorem{remark}[theorem]{Remark}

\newtheorem{assumption}[theorem]{Assumption}

\numberwithin{equation}{section}

\newcommand{\NN}{\mathbb{N}} 
\newcommand{\ZZ}{\mathbb{Z}} 
\newcommand{\MM}{\mathbb{M}} 
\newcommand{\QQ}{\mathbb{Q}} 
\newcommand{\CC}{\mathbb{C}} 
\newcommand{\EE}{\mathbb{E}} 

\newcommand{\id}{\mathrm{id}} 
\newcommand{\im}{\mathrm{im}}
\newcommand{\et}{\mathrm{\acute{e}t}}
\newcommand{\LSym}{\mathrm{LSym}}
\newcommand{\Aut}{\mathrm{Aut}}

\newcommand{\fix}{\mathrm{fix}}
\newcommand{\tv}{\mathrm{tv}}

\newcommand{\QAlg}{\mathrm{QAlg}}
\DeclareMathOperator{\colim}{{colim}}

\newcommand{\mv}{\mathrm{mv}}
\newcommand{\Res}{\mathrm{Res}}
\newcommand{\cInd}{\mathrm{cInd}}

\newcommand{\op}{\mathrm{op}}

\newcommand{\heart}{\heartsuit}
\newcommand{\vir}{\mathrm{vir}}
\newcommand{\ext}{{\mathrm{ext}}}
\newcommand{\intr}{{\mathrm{intr}}}
\newcommand{\Ob}{\mathrm{Ob}}
\newcommand{\ad}{\mathrm{ad}}
\newcommand{\dR}{\mathrm{dR}}
\newcommand{\cd}{\mathrm{cd}}
\renewcommand{\ss}{\mathrm{ss}}

\newcommand\sE{\mathcal{E}} 
\newcommand\sN{\mathcal{N}}
\newcommand\sO{\mathcal{O}} 

\newcommand{\sF}{\mathcal{F}}
\newcommand{\sA}{\mathcal{A}}
\newcommand{\sB}{\mathcal{B}}
\newcommand{\sR}{\mathcal{R}}
\newcommand{\sM}{\mathcal{M}}

\newcommand\fg{\mathfrak{g}}
\newcommand\fh{\mathfrak{h}}

\newcommand{\Sch}{\mathsf{Sch}}

\newcommand\Spc{\mathsf{Spc}}
\newcommand{\Set}{\mathsf{Set}}

\newcommand{\Stk}{\mathsf{Stk}}
\newcommand\Mod{\mathsf{Mod}}
\newcommand\Cat{\mathsf{Cat}}

\DeclareMathOperator{\Map}{Map}
\DeclareMathOperator{\Arr}{Arr} 

\newcommand{\Rep}{\mathsf{Rep}}
\newcommand{\Alg}{\mathsf{Alg}}
\newcommand{\PPP}{\mathsf{P}}
\newcommand{\CCC}{\mathsf{C}}
\newcommand{\DDD}{\mathsf{D}}
\newcommand{\EEE}{\mathsf{E}}
\newcommand{\HHH}{\mathsf{H}}
\newcommand{\modd}{\mathsf{mod}}
\newcommand{\Aff}{\mathsf{Aff}}
\newcommand{\Poly}{\mathsf{Poly}}
\newcommand{\Fun}{\mathrm{Fun}}

\newcommand{\clStk}{{\cl\Stk}}

\DeclareMathOperator{\Grp}{Grp}
\DeclareMathOperator{\Epi}{Epi}
\DeclareMathOperator{\Grpd}{Grpd}
\newcommand{\Ch}{\mathrm{Ch}}

\renewcommand{\AA}{\mathbb{A}} 
\newcommand{\PP}{\mathbb{P}} 
\newcommand{\VV}{\mathbb{V}} 

\DeclareMathOperator{\Spec}{Spec} 
\DeclareMathOperator{\Proj}{Proj} 
\DeclareMathOperator{\Bl}{Bl}
\DeclareMathOperator{\QCoh}{QCoh} 

\newcommand{\extd}{\mathrm{ext}} 
\newcommand{\cl}{\mathrm{cl}} 

\newcommand{\maxlocus}{{\mathrm{max}}}

\newcommand{\stD}{\mathscr{D}} 
\newcommand{\lL}{\mathscr{L}} 

\newcommand{\Gm}{\mathbb{G}_m}

\DeclareMathOperator{\GL}{GL}

\newcommand{\lr}{\longrightarrow}

\def\git{/\!\!/}

\newcommand{\oO}{\mathcal{O}}

\newcommand{\nN}{\mathcal{N}}
\newcommand{\mM}{\mathcal{M}}
\newcommand{\fF}{\mathcal{F}}

\newcommand{\wW}{\mathcal{W}}

\renewcommand{\tilde}{\widetilde}
\renewcommand{\hat}{\widehat}
\newcommand{\BL}{{\mathbb{L}}}

\makeatletter
\def\@tocline#1#2#3#4#5#6#7{\relax
	\ifnum #1>\c@tocdepth 
	\else
	\par \addpenalty\@secpenalty\addvspace{#2}%
	\begingroup \hyphenpenalty\@M
	\@ifempty{#4}{%
		\@tempdima\csname r@tocindent\number#1\endcsname\relax
	}{%
		\@tempdima#4\relax
	}%
	\parindent\z@ \leftskip#3\relax \advance\leftskip\@tempdima\relax
	\rightskip\@pnumwidth plus4em \parfillskip-\@pnumwidth
	#5\leavevmode\hskip-\@tempdima
	\ifcase #1
	\or\or \hskip 1em \or \hskip 2em \else \hskip 3em \fi%
	#6\nobreak\relax
	\hfill\hbox to\@pnumwidth{\@tocpagenum{#7}}\par
	\nobreak
	\endgroup
	\fi}
\makeatother

\title[Stabilizer reduction and sheaf-theoretic invariants]{Stabilizer reduction for derived stacks and applications to sheaf-theoretic invariants}
\author{Jeroen Hekking} 
\address{Department of Mathematics, Universität Regensburg, Regensburg, 93053, Germany}
\email{jeroen.hekking@ur.de}

\author{David Rydh}
\address{Department of Mathematics, KTH Royal Institute of Technology, Stockholm, 114 28, Sweden}
\email{dary@math.kth.se}

\author{Michail Savvas}
\address{Department of Mathematics, The University of Texas at Austin, Austin, TX 78712, USA}
\email{msavvas@utexas.edu}
\date{\today}

\begin{document}

	\thanks{The first author was supported by the G\"oran Gustafsson foundation.
		The second author was supported by the G\"oran Gustafsson Foundation for
		Research in Natural Sciences and Medicine.}
	
	\maketitle
	
	\begin{abstract}
		We construct a canonical stabilizer reduction $\widetilde{X}$ for any derived $1$-algebraic stack $X$ over $\CC$ as a sequence of derived Kirwan blow-ups, under mild natural conditions that include the existence of a good moduli space for the classical truncation $X_\cl$. 
		
		Our construction has several desired features: it naturally generalizes Kirwan's classical partial desingularization algorithm to the context of derived algebraic geometry, preserves quasi-smoothness, and is a derived enhancement of the intrinsic stabilizer reduction constructed by Kiem, Li and the third author. Moreover, if $X$ is $(-1)$-shifted symplectic, we show that the semi-perfect and almost perfect obstruction theory of $\widetilde{X}_\cl$ and the associated virtual fundamental cycle and virtual structure sheaf, constructed by the same authors, are naturally induced by $\widetilde{X}$ and its derived tangent complex. As corollaries, we define virtual classes for moduli stacks of semistable sheaves on surfaces, give a fully derived perspective on generalized Donaldson--Thomas invariants of Calabi--Yau threefolds and define new generalized Vafa--Witten invariants for surfaces via Kirwan blow-ups.
	\end{abstract}
	
\setcounter{tocdepth}{1}
\tableofcontents
\setcounter{tocdepth}{2}

	\section*{Introduction}
	
	\subsection*{Some brief history} It has long been understood that in order to effectively study geometric objects and their families, one needs to enlarge the category of algebraic varieties and consider the more general notion of algebraic stacks. While the theory of algebraic stacks now enjoys a vast literature and features prominently within algebraic geometry and beyond, the presence of automorphisms of objects, also called stabilizers, which is typically a desirable feature, can sometimes be problematic regarding the use of standard techniques that apply to algebraic varieties, e.g., (co)homological methods like integration.
	
	A very useful way to associate an algebraic space to an Artin stack is to consider its moduli space, whenever this makes sense. However, depending on the scope of one's considerations, this might not be adequate, given that it is not necessarily a well-behaved operation. For example, a smooth stack need not have a smooth moduli space.

	Stabilizer reduction refers to the process of resolving the stackiness of an algebraic stack to produce a canonical Deligne--Mumford stack, whose stabilizers are all finite. In classical algebraic geometry, Kirwan \cite{Kirwan} was the first to carry this out for smooth quotient stacks $X = [Q^{\mathrm{ss}}/G]$ obtained by Geometric Invariant Theory \cite{MFK}, namely by an action of a reductive group $G$ on a smooth projective variety $Q$ together with a linearization. 
	
	More precisely (assuming that $Q^{s} \neq \emptyset$), Kirwan produced a sequence of GIT quotient stacks 
	$$X_0 = X = [Q^{\mathrm{ss}} / G],\ X_1 = [Q_1^{\mathrm{ss}} / G],\ \dots\ ,\ \widetilde{X} \coloneqq X_n = [\widetilde{Q}^{\mathrm{ss}} / G]$$ 
	by iteratively blowing up $X_i$ along the (smooth) locus $X_i^{\maxlocus} \subseteq X_i$ of points with maximal dimensional stabilizer and then deleting unstable points by a careful study of stability on the blow-up. The final result $\widetilde{X} = [\widetilde{Q}^{\mathrm{ss}} / G]$ is a GIT quotient Deligne--Mumford stack satisfying $\widetilde{Q}^{\mathrm{ss}} = \widetilde{Q}^{s}$. At the level of GIT quotients, each morphism $X_{i+1} \to X_i$ fits into a natural commutative square
	\begin{align}
		\xymatrix{
			X_{i+1} = [Q_{i+1}^{\mathrm{ss}} / G] \ar[r] \ar[d] & X_i = [Q_i^{\mathrm{ss}} / G ] \ar[d] \\
			Y_{i+1} = Q_{i+1} \git G \ar[r] & Y_i = Q_i \git G
		}
	\end{align}
	where the lower horizontal arrow is a birational blow-up map. Thus, since $\widetilde{Y}$ only has finite quotient singularities and admits a projective, birational map $\widetilde{Y} \to Y$, it is a partial desingularization of $Y$. Kirwan's construction has had a wealth of applications, see for example \cite{KirwanHom, KLDesing, KiemHecke}.
	\medskip
	
	Subsequently, Kirwan's construction has been independently generalized by Edidin--Rydh \cite{EdidinRydh}, and Kiem--Li and the third author \cite{KLS, Sav}, to apply to (possibly singular) Artin stacks $X$ with a good moduli space $q \colon X \to Y$ in the sense of Alper \cite{AlperGood}. Such stacks are a natural generalization of GIT quotient stacks $X = [Q^{\mathrm{ss}} / G]$ with good moduli space morphism $q \colon [Q^{\mathrm{ss}} / G] \to Q \git G$. In fact, by deep results of Alper--Hall--Rydh \cite{AHR2, AlperLuna}, the GIT picture is the \'{e}tale local model for this general class of stacks. 
	
	The two constructions, while related, differ in the notion of blow-up being applied at each iterated step: Edidin--Rydh introduce a saturated blow-up by first blowing up $X$ along $X^{\maxlocus}$ and then carefully deleting a closed substack of unstable points, whereas Kiem--Li and the third author use a so-called Kirwan blow-up by first taking an intrinsic blow-up of $X$ along $X^{\maxlocus}$ and then deleting unstable points. The intrinsic blow-up was constructed using local embeddings into smooth quotient stacks and preserved Kuranishi models. It was thus expected to be the classical shadow of a derived blow-up. 
	
	\subsection*{Derived stabilizer reduction: statement of results} In this paper, we generalize Kirwan's algorithm to the setting of derived algebraic stacks $X$ whose classical truncation $X_\cl$ admits a good moduli space $q \colon X_\cl \to Y$. By the hidden smoothness principle, derived stacks should behave in analogy with smooth classical stacks, and hence it is natural to expect that a derived stabilizer reduction $\widetilde{X}$ should be the result of repeatedly blowing up $X$ along $X^{\maxlocus}$ and then deleting unstable points. This is exactly what we do.
	
	Our first result proves the existence of $X^{\maxlocus}$ in the derived setting, under mild finiteness assumptions which we leave implicit here and in the rest of the introduction. Uniqueness refers to uniqueness up to a contractible space of homotopies.
	
	\begin{theorem*}[Theorem \ref{thm existence of X max}]
		Let $X$ be a derived Artin stack whose classical truncation admits a good moduli space $X_\cl \to Y$. Let $d$ be the maximal dimension of stabilizers of points of $X$. 
		
		Then there exists a unique, closed immersion $X^{\maxlocus} \to X$ of derived algebraic stacks such that for any \'{e}tale morphism $[U / G] \to X$, with $U$ an affine derived scheme and $G$ a reductive group of dimension $d$, whose classical truncation fits in a Cartesian diagram
		\begin{align} \label{intro: cart diag}
			\vcenter{\xymatrix{
					[U_\cl/G] \ar[r] \ar[d] & X_\cl \ar[d] \\
					U_\cl \git G \ar[r] & Y,
			}}
		\end{align}
		there exists a Cartesian diagram of derived stacks
		\begin{align*}
			\xymatrix{
				[U^{G^0}/G] \ar[r] \ar[d] & [U / G] \ar[d] \\
				X^{\maxlocus} \ar[r] & X,
			}
		\end{align*}   
		where $G^0 \subseteq G$ is the identity component of $G$. The classical truncation of $X^{\maxlocus}$ is naturally isomorphic to $(X_\cl)^{\maxlocus}$.
	\end{theorem*}
	
	For a classical stack, a precise treatment of $X^{\maxlocus}$ was given in the appendix of \cite{EdidinRydh}. Our definition gives a derived enhancement, directly inspired by the classical construction.  To prove Theorem \ref{thm existence of X max} and to compare the derived blow-up of $X$ along the locus $X^{\maxlocus}$ to the intrinsic blow-up, we use properties of derived fixed loci, which are defined via Weil restrictions.  A substantial part of the paper is devoted to a thorough treatment of elements of derived equivariant geometry for this reason.
	
	With $X^{\maxlocus}$ at hand, the recent theory of derived blow-ups in arbitrary derived centers---generalized by the first author in \cite{HekkingGraded} from the quasi-smooth case treated in \cite{KhanRydhDBlowup}, and further developed in \cite{Weil}---together with its natural extension to (equivariant) derived stacks, allows us to define the derived intrinsic blow-up $X^\intr \coloneqq\Bl_{X^{\maxlocus}} X$ of $X$ in the center  $X^{\maxlocus}$. Our second result shows that its classical truncation is the intrinsic blow-up $(X_\cl)^\intr$ of $X_\cl$ defined by Kiem--Li and the third author. We may then delete unstable points as in \cite{KLS} to obtain an open substack $\hat{X} \subseteq X^\intr$, which we define to be the derived Kirwan blow-up of $X$.
	
	We summarize these results in the following theorem.
	
	\begin{theorem*}[Theorem~\ref{intr bl is der bl thm}, Proposition~\ref{Prop:stacky intr blow-up is derived blow-up}, Theorem~\ref{der intr kir blow-up thm}]
		There exists a canonical derived Artin stack $\hat{X}$, called the derived Kirwan blow-up of $X$, together with a morphism $\pi \colon \hat{X} \to X$, such that: 
		\begin{enumerate}
			\item Its classical truncation $\hat{X}_\cl$ admits a good moduli space morphism $\hat{q} \colon \hat{X}_\cl \to \hat{Y}$.
			\item The maximum stabilizer dimension of points in $\hat{X}$ is strictly smaller than that of $X$.
			\item For any affine, \'{e}tale, stabilizer-preserving morphism $[U/G] \to X$ whose classical truncation fits in a Cartesian diagram~\eqref{intro: cart diag}, the base change $\hat{X} \times_{X} [U/G]$ is naturally isomorphic to the derived Kirwan blow-up of $[U/G]$. 
			\item $\pi |_{\pi^{-1}(X^{\mathrm{s}})}$ is an isomorphism over the open locus $X^{\mathrm{s}}$ of stable points.
		\end{enumerate}
		
		$\hat{X}$ is the semi-stable locus $(X^{\intr})^{\mathrm{ss}} \subseteq X^{\intr}$, an open substack of the Artin stack $X^{\intr} = \Bl_{X^{\maxlocus}} X$, called the derived intrinsic blow-up of $X$.
		
		The classical truncations of $X^{\intr}$ and $\hat{X}$ are the classical intrinsic and Kirwan blow-ups of the classical truncation $X_\cl$ respectively.
	\end{theorem*}
	
	The main ingredient in the proof is an explicit, non-trivial computation of the classical truncation of the equivariant blow-up of $\Spec A$ along $(\Spec A)^{G^0}$, using several properties of derived blow-ups and fixed loci. Since $(\Bl_{X^{\maxlocus}} X)_\cl$ is in general not the blow-up of $X_\cl$ along $X_\cl^{\maxlocus}$, this theorem explains the difference between the two constructions \cite{EdidinRydh} and \cite{KLS}. When $X$ is smooth, these do coincide and both blow-ups are the same.  
	
	As an immediate consequence of this theorem we get the desired derived stabilizer reduction $\widetilde{X} \to X$ of $X$ by repeatedly taking derived Kirwan blow-ups until the maximal stabilizer dimension of points drops to zero. $\widetilde{X}$ is a derived Deligne--Mumford stack, which by construction is a natural derived enhancement of the intrinsic stabilizer reduction $\widetilde{X}_\cl$ of the classical truncation $X_\cl$ constructed in \cite{Sav} (this in particular also applies to the case where $X = X_\cl$ is classical to begin with). 
	\medskip
	
	With $\tilde{X}$ at our disposal, we investigate its properties for quasi-smooth and $(-1)$-shifted symplectic stacks and obtain associated virtual fundamental classes. These constructions have natural applications towards enumerative invariants of sheaves on surfaces and threefolds.
	
	\subsection*{Quasi-smoothness and invariants of Donaldson-type} Using the properties of derived blow-ups and the definition of derived intrinsic and Kirwan blow-ups, we show that they preserve the property of being quasi-smooth (see Proposition~\ref{prop:intr blowup quasi-smooth}). We thus obtain the following theorem.
	
    \begin{theorem*}[Theorem~\ref{thm:stab red of quasi-smooth}]	The derived stabilizer reduction $\tilde{X}$ of a quasi-smooth derived Artin stack $X$ is quasi-smooth. In particular, the intrinsic stabilizer reduction $\tilde{X}_\cl$ admits a natural virtual fundamental class $[\tilde{X}_\cl]^\vir \in A_\ast(\tilde{X}_\cl)$ and virtual structure sheaf $[\oO_{\tilde{X}_\cl}^\vir] \in K_0(\tilde{X}_\cl)$.
    \end{theorem*}
    
    Quasi-smooth derived Artin stacks naturally appear as derived moduli stacks ${\mM}$ of semistable sheaves or perfect complexes on smooth, projective surfaces. There are several variations of how one can obtain such stacks, e.g., by considering reduced versions when the genus of the surface is positive, or fixing the determinant. For more details, the interested reader can consult \cite{Mochizuki} or \cite{JoyceWC}.
    
    Let $\mM$ be a quasi-smooth moduli stack of this type such that the good moduli space $M_\cl$ of $\mM_\cl$ is proper. Then $\tilde{\mM}_\cl$ is a proper Deligne--Mumford stack and integrals $\int_{[\tilde{\mM}_\cl]^\vir} \gamma$ for cohomology classes $\gamma$ on $\tilde{\mM}_\cl$ may be viewed as Donaldson-type invariants (see Definition~\ref{def:Donaldson invariants}).
    
	\subsection*{$(-1)$-shifted symplectic geometry and Donaldson--Thomas invariants} The main motivation for the introduction of the intrinsic stabilizer reduction $\widetilde{X}_\cl$ in \cite{KLS, Sav} was the construction of intersection-theoretic generalized Donaldson--Thomas invariants. 
	
	Namely, it was shown by local computation that if $X$ is a $(-1)$-shifted symplectic derived Artin stack, then the intrinsic stabilizer reduction $\widetilde{X}_\cl$ admits a canonically induced semi-perfect obstruction theory (cf.\ \cite{LiChang}) of virtual dimension zero, and later in \cite{KiemSavvas} that in fact this can be promoted to an almost perfect obstruction theory. In particular, there are an associated virtual fundamental cycle and virtual structure sheaf
	\begin{align} \label{loc 1.3}
		[\widetilde{X}_\cl]^\mathrm{vir} \in A_0(\widetilde{X}_\cl), \quad [\oO_{\widetilde{X}_\cl}^\vir] \in K_0(\widetilde{X}_\cl),
	\end{align}
	and one can define
	$$\mathrm{DTK}(X) = \int_{[\widetilde{X}_\cl]^\mathrm{vir}} 1 \in \mathbb{Q}$$
	as the numerical generalized Donaldson--Thomas invariant via Kirwan blow-ups (DTK invariant) associated to $X$, along with $K$-theoretic invariants. Since derived moduli stacks of semi-stable sheaves and perfect complexes on Calabi--Yau threefolds are $(-1)$-shifted symplectic by \cite{PTVV}, this gives rise to DTK invariants that act as virtual counts of these objects (after rigidifying scaling automorphisms).
	
	By suitable local computation, we further show that our derived stabilizer reduction procedure gives a derived enhancement of $\widetilde{X}_\cl$ which completely recovers the above obstruction theory.
	
	\begin{theorem*}[Theorem~\ref{thm trunc perf}] Let $X$ be a $(-1)$-shifted symplectic derived Artin stack and $\widetilde{X}$ its derived stabilizer reduction. Then:
		\begin{enumerate}
			\item The $[0,1]$-truncation $E^\bullet = \tau^{[0,1]} \mathbb{T}_{\widetilde{X}}|_{\widetilde{X}_\cl}$ (with cohomological indexing notation) is a perfect complex. 
			\item Its dual $E_\bullet = (E^\bullet)^\vee$ together with the derived structure of $\tilde{X}$ naturally recover the data of the semi-perfect obstruction theory of \cite{KLS, Sav} and almost perfect obstruction theory of \cite{KiemSavvas} and hence the virtual fundamental cycle $[\widetilde{X}_\cl]^\vir$ and virtual structure sheaf $[\oO_{\widetilde{X}_\cl}^\vir]$ in~\eqref{loc 1.3}.
		\end{enumerate}
	\end{theorem*}
	
	This gives a fully derived perspective on generalized Donaldson--Thomas invariants via Kirwan blow-ups. As a corollary, we may combine the DTK formalism with virtual torus localization (cf.\ \cite{GrabPand}) in order to define new, generalized Vafa--Witten invariants enumerating Gieseker semi-stable Higgs pairs $(E, \phi \colon E \to E \otimes K_S)$ on $S$ with fixed positive rank and Chern classes, fixed determinant $\det E$ and zero trace $\mathrm{tr}\ \phi = 0$ (see Definition~\ref{Def: gen VW inv}).

	\subsection*{Future directions} Given Kirwan's original algorithm and its generalization in \cite{EdidinRydh}, it is natural to wonder whether there is a theory of derived good moduli spaces for derived algebraic stacks, such that, in particular, our derived stabilizer reduction procedure gives a sequence of diagrams
	\begin{align*}
		\xymatrix{
			X_{i+1} = \hat{X}_i \ar[r] \ar[d] & X_i \ar[d] \\
			Y_{i+1} \ar[r] & Y_i,
		}
	\end{align*}
	where the vertical maps are derived good moduli space morphisms, and which are thus derived enhancements of the diagrams~\eqref{intro: cart diag}. In fact, using the derived \'{e}tale slice theorem and local-to-global methods, one can indeed construct derived good moduli spaces whose existence is preserved by the operation of derived Kirwan blow-up. The details will appear elsewhere to keep the length of this paper under control.
	
	Another interesting direction of inquiry would be to investigate the properties and degeneracy of the pullback of an $n$-shifted symplectic form to the derived stabilizer reduction and possibly develop a theory of $n$-shifted quasi-symplectic derived stacks, of which derived stabilizer reductions are a particular case.
	
	Other natural extensions of this work would be to generalize the stabilizer reduction algorithm to derived stacks over a base of positive or even mixed characteristic. Since our main  arguments ultimately depend on the classical situation and on the formal properties of Weil restrictions, we in fact expect a generalized stabilizer reduction algorithm to exist in any nonconnective algebraic geometry. Such geometries are extensions of derived algebraic geometry---of which derived analytic geometry in the sense of \cite{BassatNonArchimedean}, \cite{BassatAnalytification}, \cite{FDA} is an example---and are proposed by Ben-Bassat and the first author in \cite{GDB}, together with a generalization of derived blow-ups to this setting. 
	
	Finally, there are several purely enumerative potential ramifications for sheaf-theoretic invariants that remain to be explored.
	
	We intend to come back to these questions in the future.
	
	\subsection*{Layout of the paper} In \S\ref{Sec: background dalg} we review the requisite background and language of derived algebraic geometry that we use. In \S\ref{Sec:Derived Blowups}, we recall derived blow-ups and discuss and prove statements that are necessary in subsequent parts of the paper. \S\ref{Sec:Derived_equivariant_geometry} is devoted to a careful study of derived equivariant geometry with a focus on fixed loci of derived group actions. \S\ref{Sec:Xmax} gives the construction of the locus $X^{\maxlocus}$ of points of maximal stabilizers of a derived stack $X$. In \S\ref{Sec:intr bl is der bl} we establish the relationship between intrinsic blow-ups and equivariant derived blow-ups, and in \S\ref{Sec:der stab red} we construct the derived stabilizer reduction of a derived Artin stack. \S\ref{Sec:quasi-smooth} deals with the case of quasi-smooth stacks, while \S\ref{Sec:Shifted Symplectic} with the case of $(-1)$-shifted symplectic stacks. Finally, \S\ref{Sec: DT VW} discusses applications to generalized Donaldson--Thomas and Vafa--Witten invariants. 
	
	\subsection*{Acknowledgements} 
	The first author would like to thank Adeel Khan for valuable conversations about derived algebraic stacks. 
	
	The third author would like to thank Jarod Alper and Daniel Halpern-Leistner for helpful answers to questions related to the \'{e}tale slice theorem and good moduli spaces and Hyeonjun Park for discussions related to fixed loci.
	
	\subsection*{Notation and conventions}
	Throughout, we work over $\CC$. All rings and algebras are assumed to be commutative. We will use both simplicial/derived rings and commutative differential graded algebras interchangeably, assuming their usual equivalence under the Dold--Kan correspondence. The particular model we are working with will be clear from context.

	All derived stacks and all good moduli space morphisms are assumed to have affine diagonal. By a derived algebraic stack we mean a derived $1$-algebraic stack.

	The classical truncation of $X$ is denoted by $X_\cl$ and the topological space of points of $X$ by $\lvert X\rvert$. If $x \in \lvert X\rvert$, then $G_x$ denotes the automorphism group/stabilizer of $x \in X_\cl$. In practice, we will mostly work with stabilizers of closed points $x \in \lvert X\rvert$. These will be reductive for most stacks of interest in this paper, namely stacks whose classical truncation admits a good moduli space.
	
	For a morphism $\rho \colon X \to Y$ and a sheaf or complex $\sE$ on $Y$, we often use $\sE \vert_X$ to denote $\rho^* \sE$, suppressing the pullback from the notation. We also use the phrases ``closed embedding'' and ``closed immersion'' interchangeably, likewise for ``Artin stack'' and ``algebraic stack''.
	
	$G$, $H$ denote complex reductive groups throughout. Since we are working over $\CC$, being reductive is equivalent to being linearly reductive and hence all finite-dimensional and rational $G$-representations are completely reducible in this paper (see, for example, \cite[Theorem~5.2]{Borel}). Usually, $H$ will be a subgroup of $G$. $G^0$ is the identity component of $G$. $T$ is the torus $\Gm = \CC^\times$.
	
	If $X$ is a derived stack with $G$-action, $X^G$ is used to denote the derived fixed point locus of $G$ in $X$, see \S \ref{Subsec:Absolute_fixed_loci}. For $X$ algebraic, $X^\intr, \hat{X}$ are used to denote the derived intrinsic and Kirwan blow-ups of $X$ respectively, see \S \ref{Subsec:Derived_intrinsic_and_Kirwan_blow-ups}. $\widetilde{X}$ denotes the derived stabilizer reduction of a stack $X$, see \S \ref{Subsec:Derived_stabilizer_reduction_of_Artin_stacks}.
	
	Quasi-coherent modules and algebras are written as $\sM, \sN, \dots$ and $\sA,\sB,\dots$  to distinguish them from their affine counterparts $M,N,\dots$ and $A,B,\dots$.
	
	We will use homological indexing notation, unless otherwise stated.
	
	We will mostly ignore set-theoretic issues to avoid distractions. These can be resolved by any of the standard methods, such as assuming the existence of Grothendieck universes.

	\section{Recollections on derived algebraic geometry} \label{Sec: background dalg}
	We review the language of derived algebraic geometry which we will be using. Much of this material can be found in the standard literature on the topic, like \cite{LurieSpectral}, \cite{ToenHAGII}, \cite{GaitsgoryStudy}. For convenience, we collect some constructions and results here, but refer the reader to these sources for more details.

	\subsection{Background on $\infty$-categories}
	Recall that an $\infty$-category is a generalization of a $1$-category that allows us to talk about $n$-morphisms for all $n \in \NN_{\geq 1}$, where all morphisms above $n=1$ are invertible (in other words, these are $(\infty,1)$-categories). We thus have, between any two objects $x,y$ in a given $\infty$-category $\CCC$, a mapping space $\CCC(x,y)$ which has the set of equivalence classes of $1$-morphisms $x \to y$ as connected components. Here, a space is considered as a homotopy type, which can be modelled for example on simplicial sets. 
	
	We use similar shorthands for $\infty$-categories as is customary for $1$-categories, by expressing the existence of a homotopy as a property. In particular, when we say that a diagram is commutative, we mean that there is a homotopy that makes the diagram commutative. Likewise, by uniqueness we will mean uniqueness up to a contractible space of choices.
	
	Write $\Cat$ for the $\infty$-category of $\infty$-categories, and $\Spc$ for the $\infty$-category of spaces. Although we will explicitly distinguish $\infty$-categories from 1-categories throughout, we do adopt the common language of $\infty$-categories, where limits, colimits, Kan extensions, etc., are always taken in their $\infty$-categorical meaning. Recall, if one chooses a point-set model, then these become homotopy limits, colimits, Kan extensions, etc..
	
	\begin{definition}
		Let $\CCC$ be an $\infty$-category. Define the \textit{cocompletion of $\CCC$ under sifted colimits}, written $\PPP_\Sigma(\CCC)$, as the full subcategory of the $\infty$-category $\PPP(\CCC)$ of presheaves on $\CCC$ spanned by those presheaves that send finite coproducts in $\CCC$ to products in $\Spc$. 
	\end{definition}
	\begin{remark}
		The $\infty$-category $\PPP_\Sigma(\CCC)$ is the universal way of adding sifted colimits to $\CCC$. It is also called the \textit{nonabelian derived category} of $\CCC$. See \cite[\S 5.5.8]{LurieHTT} for details. 
	\end{remark}

	\subsection{Derived rings and modules}
	Let $\Mod$ be the symmetric monoidal $\infty$-category of modules over $\CC$ in the stably homotopic sense, i.e., of spectra endowed with a $\CC$-action. Endow $\Mod$ with the standard $t$-structure (with homological indexing convention). Recall that $\Mod$ is also the $\infty$-category associated to the unbounded derived category of $\CC$-vector spaces. 
	
	Write $\Mod^0$ for the full subcategory of $\Mod$ spanned by the discrete, finitely generated $\CC$-modules. Then $\Mod^0$ generates $\Mod_{\geq 0}$ under sifted colimits.
	
	\begin{definition}
		Let $\Alg$ be the $\infty$-category of connective $\EE_\infty$-algebras in $\Mod$.
	\end{definition}
	
	Let $\Poly$ be the category of finitely generated polynomial rings over $\CC$. Since we are working over a ring of characteristic 0, $\Alg$ is equivalent to the $\infty$-category $\PPP_\Sigma(\Poly)$, hence also to the $\infty$-category associated to the model category of simplicial $\CC$-algebras, endowed with the projective model structure. It is also equivalent to the $\infty$-category of cdgas over $\CC$ concentrated in homological degree $\geq 0$. 
	
	An object $R$ of $\Alg$ will be called an \textit{algebra}. We can associate to $R$ a 1-categorical model $R'$, either incarnated as a simplicial ring or as a cdga. Of course, any $R''$ which is weakly equivalent to $R'$ will then also be a 1-categorical model of $R$. To keep the indexing convention consistent, we will write our cdgas as chain complexes in degree $\geq 0$.
	
	The forgetful functor $\Alg \to \Mod_{\geq 0}$ has a left adjoint, written $\LSym$. For $R \in \Alg$, we write $\Alg_R$ for the slice-$\infty$-category $\Alg_{R/}$. Objects in $\Alg_R$ are called \textit{$R$-algebras.} Likewise, we write $\Mod_R$ for the $\infty$-category of $R$-modules in $\Mod$. The forgetful functor $\Alg_R \to (\Mod_R)_{\geq 0}$ again has a left adjoint, written $\LSym_R$. 
	
	Let $\Mod^0_R$ be the full subcategory of $\Mod_R$ spanned by finitely generated free $R$-modules, and let $\Poly_R \subset \Alg_R$ be the essential image of $\LSym_R$ on $\Mod^0_R$. Then $\Mod^0_R$ again generates $(\Mod_R)_{\geq 0}$ under sifted colimits, and the canonical map $\PPP_\Sigma(\Poly_R) \to \Alg_R$ is an equivalence. Moreover, $\Alg_R$ is the $\infty$-category of connective $\EE_\infty$-algebras in $\Mod_R$.
	
	The functor $\pi_0 \colon \Spc \to \Set$ induces adjunctions
	\begin{align*}
		\pi_0: \Mod_{\geq 0} \leftrightarrows \Mod^\heart : i && \pi_0: \Alg \leftrightarrows \Alg^\heart : i
	\end{align*}
	where $\Mod^\heart$ is the heart of $\Mod$ with respect to the given $t$-structure, i.e., the 1-category of $\CC$-vector spaces, and $\Alg^\heart$ is the 1-category of discrete $\CC$-algebras. 
	
	For $R \in \Alg$ and $M \in \Mod_R$, the homotopy groups $\pi_n(M)$ are canonically $\pi_0(R)$-modules, for all $n \in \ZZ$. In particular, $\pi_n(R)$ is a $\pi_0(R)$-module for all $n  \geq 0$.

	\subsection{Cell attachments}
	Let $R \in \Alg$. For $n \geq 0$, we adjoin a free variable in homological degree $n$ to $R$ by writing $R[u] = R[S^n]$, which has the universal property that the space of $R$-algebra maps $R[u] \to B$ is equivalent to the space of maps $(S^n,*) \to (B,0)$ of pointed spaces, i.e., to $\Mod_R(R[n],B)$.
	
	Let a sequence $\underline{\sigma} = (\sigma_1,\dots,\sigma_k)$ of elements $\sigma_i \in \pi_{n_i}(R)$ be given. Then we define the \textit{finite quotient} of $R$ by $\sigma$, written $R/(\underline{\sigma}) = R/(\sigma_1,\dots,\sigma_k)$, via the pushout
	\begin{center}
		\begin{tikzcd}
			{R[u_1,\dots,u_k]} \arrow[r, "z"] \arrow[d, "s"] & R \arrow[d] \\
			R \arrow[r] & R/(\sigma_1,\dots,\sigma_k)
		\end{tikzcd}
	\end{center}
	where the $u_i$ are free in homological degree $n_i$, the map $s$ is induced by lifts $S^{n_i} \to R$ of $\sigma_i$, and $z$ sends each $u_i$ to zero. When $k=1$, we also say that $R/(\underline{\sigma})$ is obtained from $R$ by \textit{attaching an $(n+1)$-cell}. A \textit{zero cell attachment} is the map $R \to R[t]$, where $t$ is in homological degree 0. Note that the input of attaching several $(n+1)$-cells is equivalent to a morphism $\sigma \colon M[n] \to R$ of $R$-modules, where $M$ is free of finite rank. Giving this datum, we write the quotient also as $R/(\sigma)$. 
	
	A finite quotient of the form $R \to R/(f_1,\dots,f_k)$, where each $f_i$ is in homological degree $0$, is called \textit{quasi-smooth}.
	
	\subsection{Standard forms} For the purpose of performing computations, it will be useful to work with certain $1$-categorical models.
	
	Recall that a morphism $A \to B$ of algebras (and the corresponding map on spectra) is \emph{locally of finite presentation} if $B$ is a compact object in $\Alg_A$, i.e., the $\Spc$-valued functor $\Alg_A(B,-)$ preserves filtered colimits. It is \emph{finitely presented} if $B$ can be obtained from $A$ by a finite number of cell attachments. Then the following are equivalent: (1) $B$ is locally of finite presentation over $A$; (2) the map $\Spec B \to \Spec A$ is, Zariski-locally on $\Spec B$, finitely presented; and (3) $B$ is a retract of a finitely presented $A$-algebra.
	
	\begin{definition}
		Let $A$ be an algebra. A cdga model $R$ for $A$ is \emph{in standard form} if $R_0$ is smooth with $\Omega_{R_0}$ free, and the underlying graded-commutative ring of $R$ is freely generated over $R_0$ by a finite number of generators. 
	\end{definition}

	\begin{remark}
		\label{Rem:standard form cell attachment}
		Observe, if a cdga $R$ is in standard form, then it can be obtained from $R_0$ by a finite number of cell-attachments, say $R(0) \to R(1) \to \dots \to R(n)$, in such a way that when we terminate the process at $k \leq n$, then $R(k)$ is also in standard form. In fact, the underlying graded-commutative ring of $R(k)$ is the subring of the underlying graded-commutative ring of $R$ generated by elements of homogeneous degree $\leq k$, and $R(k+1)=R(k)/(\sigma)$ where $\sigma\colon M_{k+1}[k]\to R(k)$ and $M_{k+1}$ is a free module of finite rank over $R(k)$. Moreover $R_{k+1}=R(k)_{k+1} \oplus (M_{k+1})_0$. We say that $R$ is \emph{generated} by the $M_{k+1}$. 
	\end{remark}
	
	If $R$ is a model of standard form for $A$, then the cotangent complex of $\Spec A$ is given by the K\"{a}hler differentials $\Omega_R$, see \cite[\S 2.3]{JoyceSch}.
	
	\begin{definition}
		Let $x \in \Spec A$ with $A \in \Alg$ be given. Then a model $R$ in standard form for $A$ is \emph{minimal at $x$} if the differentials of $\Omega_R|_x$ are all zero.
	\end{definition}
	
	By the following lemma, locally finitely presented derived schemes can Zariski-locally be described by models in standard form.

	\begin{lemma} \label{Lem:Alg are standard form}
		Let $A \in \Alg$ be locally of finite presentation and $x \in \Spec A$. Then, up to Zariski localizing around $x$, we can find a model $R$ in standard form for $A$ such that $R_0$ is a finitely generated polynomial ring, or such that it is minimal at $x$.
	\end{lemma}
	
	\begin{proof}
		This follows for example from \cite[Theorem~4.1]{JoyceSch}.
	\end{proof}
	
	\begin{remark}
		In general, the two conditions cannot be met simultaneously, see \cite[Remark~3.2]{JoyceSch}.
	\end{remark}

	\subsection{Graded derived rings}
	\label{Par:Graded_derived_rings}
	Derived blow-ups in arbitrary centers are defined  along similar lines as the classical construction---namely, through a derived version of the Rees algebra, of which one then takes the projective spectrum. To explain this, let us first review graded algebras in the derived setting as defined in \cite{HekkingGraded}.
	
	Let $\MM$ be a commutative monoid, considered as a discrete symmetric monoidal category in the obvious way. Let $\Mod^\MM$ be the functor $\infty$-category $\Fun(\MM,\Mod)$, endowed with symmetric monoidal structure via Day convolution and with $t$-structure inherited from $\Mod$. An object $M$ of $\Mod^\MM$ is called an \textit{$\MM$-graded module}, written as $\bigoplus_{d \in \MM} M_d$, where $M_d$ is $M$ evaluated at $d$. 
	
	For $d \in \MM$ and $N \in \Mod$, we write $N(d)$ for the $\MM$-graded module which is $N$ concentrated in homogeneous degree $-d$. Let then $(\Mod^\MM)^0$ be the full subcategory of $\Mod^\MM$ generated by finite coproducts of modules of the form $N(d)$, where $N \in \Mod^0$ and $d \in \MM$. Then $(\Mod^\MM)^0$ generates $(\Mod^\MM)_{\geq 0}$ under sifted colimits.
	
	Define the $\infty$-category of \textit{$\MM$-graded algebras}, written $\Alg^{\MM}$, as the $\infty$-category of connective $\EE_\infty$-algebras in $\Mod^{\MM}$. As in the ungraded case, the forgetful functor $\Alg^{\MM} \to (\Mod^{\MM})_{\geq 0}$ has a left adjoint, written $\LSym^\MM$. Let $\Poly^\MM$ be the essential image of $\LSym^\MM$ evaluated on $(\Mod^\MM)^0$. Then the canonical map $\PPP_\Sigma(\Poly^\MM) \to \Alg^\MM$ is an equivalence.
	
	A key result from \cite{HekkingGraded} is that, as expected, the $\infty$-category $\Alg^\ZZ$ is (contravariantly) equivalent to the $\infty$-category of affine derived schemes with $\Gm$-action. The general theory of $\infty$-actions is reviewed in \S \ref{Par:actions}, and we give an in-depth analysis of $G$ actions on derived stacks for reductive algebraic groups $G$ in \S \ref{Sec:Derived_equivariant_geometry}.
	
	For $B \in \Alg^\MM$, write $\Mod^\MM_B$ for the $\infty$-category of \textit{$\MM$-graded $B$-modules}, which by definition are $B$-modules in $\Mod^\MM$. As in the ungraded case, $\Mod^\MM_B$ is stable and symmetric monoidal, and $(\Mod^\MM_B)_{\geq 0}$ is generated under sifted colimits by the $\infty$-category of finitely generated, free, $\MM$-graded $B$-modules. Once again, $\Alg^\MM_B \coloneqq (\Alg^\MM)_{B/}$ is generated under sifted colimits by the finitely generated, $\MM$-graded polynomial $B$-algebras, which are defined similarly as in the ungraded case. 
	
	We are primarily interested in $\NN$-graded and $\ZZ$-graded rings. These exhibit familiar behavior: to a $\ZZ$-graded ring $B$ we can associate an $\NN$-graded ring $B_{\geq 0}$ by discarding the pieces in negative homogeneous degrees, and an $\NN$-graded ring $B'$ can be promoted to a $\ZZ$-graded ring by putting $0$ in negative homogeneous degrees. We can also take the homogeneous degree-zero part $B_0$ of a $\ZZ$-graded ring, which gives us morphisms $B_0 \to B_{\geq 0} \to B$ of $\ZZ$-graded rings (where we think of $B_0,B_{\geq 0}$ as $\ZZ$-graded by the procedure just described, which we suppress from notation). If $B$ is $\NN$-graded, we also have a morphism $B \to B_0$ such that the composition $B_0 \to B \to B_0$ is homotopic to the identity.
	
	Let $\MM$ be either $\NN$ or $\ZZ$, and let $B \in \Alg^\MM$. As in the ungraded case, we can adjoin a free variable to $B$ in homological degree $n \in \NN$ and homogeneous degree $d \in \MM$, which produces an $\MM$-graded $B$-algebra $B[u]$ with the universal property that the space of $\MM$-graded $B$-algebra maps $B[u] \to C$ is equivalent to the space of maps $(S^n,*) \to (C_d,0)$, i.e.,\ to $\Mod^{\MM}(B[n]_d,C)$, where $B[n]_d$ is the homogeneous degree-$d$ part of the module $B[n]$.
	
	The $\MM$-grading on $B$ induces an $\MM$-grading on each $\pi_n(B)$. We can thus define a \textit{finite quotient} of $B$ by a sequence $\underline{\sigma} = (\sigma_1,\dots,\sigma_k)$ of elements $\sigma_i \in \pi_{n_i}(B)_{d_i}$, with $n_i \in \NN$ and $d_i \in \MM$, using the graded free variables construction just described and performing cell attachments in the same way as in the ungraded case. The result is again written as $B /(\underline{\sigma})$.
	
	\subsection{Derived stacks}
	Let $\mathbf{P}$ be one of the following properties: \'{e}tale, a (Zariski) open immersion, smooth. A morphism $A \to B$ of derived rings is called $\mathbf{P}$ if $\pi_0A \to \pi_0B$ is $\mathbf{P}$, and the morphism $(\pi_nA) \otimes_{\pi_0 A} (\pi_0B) \to \pi_nB$ is an isomorphism, for any $n \geq 0$. Define the $\infty$-category of \textit{affine derived schemes} as $\Aff \coloneqq \Alg^\op$, and write $\Spec(-)\colon \Alg^\op \to \Aff$ for the obvious functor. Declare morphisms in $\Aff$ to be $\mathbf{P}$ if the corresponding map in $\Alg$ is $\mathbf{P}$. For $U = \Spec A$ in $\Aff$, the \textit{underlying classical scheme} is $U_\cl \coloneqq \Spec \pi_0A$. 
	
	We endow $\Aff$ with the \textit{\'{e}tale topology} by declaring a family $\{U_i \to U\}_i$ in $\Aff$ to be a cover if all $U_i \to U$ are \'{e}tale, and the induced maps on classical schemes $(U_i)_\cl \to U_\cl$ are jointly surjective. 
	
	A \textit{derived prestack} is a functor $X\colon \Aff^\op \to \Spc$. A \textit{derived stack} is a derived prestack which is a sheaf (in the $\infty$-categorical sense) for the \'{e}tale topology. We write $\Stk$ for the $\infty$-category of derived stacks. For $X \in \Stk$, we write $\Stk_X$ for the $\infty$-category $\Stk_{/X}$ of derived stacks over $X$. The Yoneda embedding induces an embedding $\Spec(-)\colon \Aff \to \Stk$.
	
	Recall that a space $K$ is called \emph{$n$-truncated} if $\pi_m(K,x)$ is trivial for all points $x \in K$ and all $m>n$. A  stack $X$ is called \emph{$n$-truncated} if $X(T)$ is $n$-truncated for all classical affine schemes $T$. A morphism $X \to Y$ of derived stacks is \textit{affine} if $X \times_Y (\Spec A)$ is affine, for all $\Spec A \to Y$. Then a derived stack $X$ has \textit{affine diagonal} if the canonical morphism $\Delta\colon X \to X \times X$ is affine. Observe that $X$ has affine diagonal if and only if the intersection $U \times_X V$ is affine, for any pair of affine derived schemes $U,V$ over $X$.
	
	\begin{assumption}
		From here on, assume that all derived stacks are 1-truncated and have affine diagonal, unless otherwise stated.
	\end{assumption}

	Let $\clStk$ be category of classical stacks. Composition with the inclusion $\Alg^\heart \to \Alg$ gives a functor $(-)_\cl\colon \Stk \to \clStk$ which sends $X \in \Stk$ to the \textit{underlying classical stack} $X_\cl$. This functor has a fully faithful left adjoint $\iota$, and both functors $\iota, (-)_\cl$ preserve affines. A morphism $X \to Y$ of derived stacks is \textit{quasi-compact}/\textit{a closed immersion} if the morphism $X_\cl \to Y_\cl$ is quasi-compact/a closed immersion.

	\subsection{Derived schemes}
	In this and the next paragraph, let $\mathbf{P}$ be one of the properties: \'{e}tale, smooth, an open immersion. Although we could add here any property which is \'{e}tale-local in an appropriate sense, we restrict ourselves to what we need.
	
	Let $X$ be a derived stack. A family $\{U_i \to X\}_i$ of morphisms in $\Stk$ is a \textit{cover} if the canonical map $\colim \check{C}(f) \to X$ is an equivalence, where $\check{C}(f)$ is the \v{C}ech nerve of $f\colon \bigsqcup U_i \to X$, and the colimit is taken in $\Stk$. A morphism $T \to X$ in $\Stk$ with affine source \textit{is $\mathbf{P}$} if $S \times_X T \to S$ is $\mathbf{P}$ for all affine derived scheme $S$ over $X$. The latter is well-defined, since we assume that all derived stacks have affine diagonal. 
	
	A derived stack $U$ is called a \textit{derived scheme} if there is cover $\{U_i \to U\}$ of open immersions $U_i \to U$ by affine derived schemes $U_i$. Observe, the adjunction $\iota: \clStk \leftrightarrows \Stk : (-)_\cl$ restricts to an adjunction between classical schemes (which we assume to also have affine diagonal) and derived schemes. The $\infty$-category $\Sch$ of derived schemes is closed under fiber products in $\Stk$.
	
	A morphism $U \to X$ of derived stacks with $U$ a derived scheme \textit{is $\mathbf{P}$} if $U \times_X V \to V$ is  $\mathbf{P}$, for any derived scheme $V$ over $X$.

	\subsection{Derived algebraic spaces and stacks}
	A \textit{derived algebraic space} is a $0$-truncated derived stack $Y$ for which there is an \'{e}tale cover $U \to Y$ by a derived scheme $U$. A morphism $X' \to X$ of derived stacks is \textit{representable} if for all  derived schemes $U$ over $X$ the fiber product $U \times_X X'$ is a derived algebraic space. Observe that any morphism $Y \to X$ from a derived algebraic space to a derived stack is representable. Again, this is because  $X$ has affine diagonal, and since for any affine morphism $X' \to V$ in $\Stk$ with $V$ a derived scheme it holds that $X'$ is a derived scheme as well.
	
	A morphism $Y' \to Y$ of algebraic spaces \textit{is $\mathbf{P}$} if there is an \'{e}tale cover $V \to Y'$ by a derived scheme $V$ such that the composition $V \to Y$ is $\mathbf{P}$. A representable morphism $X' \to X$ of derived stacks \textit{is $\mathbf{P}$} if for any derived scheme $U$ over $X$, the map $U \times_X X' \to U$ is $\mathbf{P}$. 
	
	A \textit{derived algebraic stack} is a derived stack $X$ such that there is a smooth cover $Y \to X$ with $Y$ a derived algebraic space. Observe, the adjunction $\iota: \clStk \leftrightarrows \Stk : (-)_\cl$ restricts to an adjunction between classical algebraic stacks and derived algebraic stacks, and also between classical algebraic spaces and derived algebraic spaces.
	
	Let $X' \to X$ be a morphism of algebraic stacks. Then $X' \to X$ is \textit{smooth} if there is a smooth cover $Y' \to X'$ by a derived algebraic space $Y'$ such that the composition $Y' \to X$ is smooth. Likewise, $X' \to X$ is \textit{\'{e}tale} if, smooth-locally on $X$, there is an \'{e}tale cover $Y' \to X'$ by a derived algebraic space $Y'$, such that the composition $Y' \to X$ is \'{e}tale. Finally,  $X' \to X$ is an \textit{open immersion} if it is representable, \'{e}tale and a monomophism. 
	
	\subsection{Quasi-coherent modules and algebras}
	 
	Let $f\colon\Spec B \to \Spec A$ be a morphism of affine derived schemes. Then we have an adjunction $f^*: \Mod_A \leftrightarrows \Mod_B : f_*$, where $f^*$ sends $M \in \Mod_A$ to $M \otimes_AB$. This gives us a functor $
	\Aff^\op \to \Cat$ that sends $\Spec R$ to $\Mod_R$.
	Taking the right Kan extension yields the functor
	\begin{align*}
		\Stk^\op \to \Cat \colon X \mapsto \QCoh(X)
	\end{align*}
	which sends $X$ to the $\infty$-category of \textit{quasi-coherent $\sO_X$-modules}. Concretely, for $X \in \Stk$, we can write $\QCoh(X)$ as the limit of the diagram $\Spec R \mapsto \Mod_R$ over the comma $\infty$-category $\Aff^\op_{/X}$ consisting of derived affine schemes over $X$. This description tells us that we can think of a quasi-cohorent $\sO_X$-module $\sF$ as a homotopy coherent diagram $\{\sF_A\}_{\Spec A \to X}$, where each $ \sF_A$ is in $\Mod_A$. We also write $\sF(A) \coloneqq \sF(\Spec A) \coloneqq \sF_A$.
	
	For $X \in \Stk$ algebraic, we can restrict the indexing category in the formula $\QCoh(X) \simeq \lim_{\Spec R \to X} \Mod_R$ to derived affine schemes which are smooth over $X$, and still get the same result. Then $\QCoh(X)$ is naturally a stable, presentable, symmetric monoidal $\infty$-category for which the tensor product commutes with colimits in each variable separately, and moreover $\QCoh(X)$ carries a $t$-structure such that $\sF \in \QCoh(X)$ is connective if and only if $\sF_A$ is, for all $\Spec A \to X$. For a morphism $f\colon X \to Y$ of derived algebraic stacks, by the adjoint functor theorem we have an adjunction $f^*:\QCoh(Y) \leftrightarrows \QCoh(X) :f_*$, where $f^*$ is symmetric monoidal and $f_*$ is right-lax symmetric monoidal.
	
	Following the same procedure, we have a functor $\Aff^\op \to \Cat$ which sends $\Spec R$ to $\Alg_R$. Right-Kan extending this yields a functor $\Stk^\op \to \Cat$, which we write as $X \mapsto \QAlg(X)$. Objects in $\QAlg(X)$ are called \textit{quasi-coherent $\sO_X$-algebras}. When restricted to derived algebraic stacks, this functor lands in the $\infty$-category of presentable, symmetric monoidal $\infty$-categories, and symmetric monoidal colimit preserving functors between them. Since we are in characteristic zero, we can also describe $\QAlg(X)$ as the $\infty$-category of $\EE_\infty$-algebras in $\QCoh(X)_{\geq 0}$. We furthermore have a left adjoint to the forgetful functor $\QAlg(X) \to \QCoh(X)_{\geq 0}$, written $\LSym_X$.
	
	\begin{example}
		For $X \in \Stk$, let $\sO_X \in \QAlg(X)$ be such that $\sO_X(A) =  A$, for all $\Spec A \to X$, which is a unit in $\QAlg(X)$. We also write $\sO_X$ for the underlying quasi-coherent module.
	\end{example} 
	
	Let $X \in \Stk$ be algebraic. For  $\sA \in \QAlg(X)$, the (relative) \textit{spectrum} of $\sA$ is the derived algebraic stack $\Spec \sA$ over $X$ such that for $f\colon T \to X$, the space $(\Spec \sA)(T)$ is the space of quasi-coherent $\sO_T$-algebra maps $f^*\sA \to \sO_T$. 
	
	Write $\Aff_X$ for the full subcategory of $\Stk_X$ spanned by derived stacks which are affine over $X$, i.e., for which the structure map is affine. For $\sA \in \QAlg(X)$, it holds that $\Spec \sA$ is affine over $X$. In fact, the functor $\sA \mapsto \Spec \sA$ yields an equivalence $\QAlg(X)^\op \simeq \Aff_X$.
	
	Let $X \in \Stk$ be algebraic, and let $\sE \in \QCoh(X)_{\geq (-1)}$ be given. Define $\VV(\sE)$ as the derived stack over $X$ with the universal property that for $f\colon T \to X$ it holds that $\VV(\sE)$ is the space of $\sO_T$-module maps $f^*\sE \to \sO_T$. If $\sE$ is connective, then $\VV(\sE) \simeq \Spec \LSym_X(\sE)$.

	\subsection{The cotangent complex and the normal bundle}
	Let $X$ be a derived algebraic stack. A quasi-coherent $\sO_X$-module $\sF$ is called \textit{almost connective} if for all $\Spec A \to X$ there is some $n \in \ZZ$ such that $\sF_A \in (\Mod_A)_{\geq n}$. 
	
	Let $T = \Spec A$ be an affine derived scheme, and $M$ a connective $A$-module. Then there is a derived version of the square-zero extension of $A$ by $M$, written $A \oplus M$, which is an $A$-algebra over $A$. We write $T[M]$ for $\Spec (A \oplus M)$.
	
	Let $f\colon X \to Y$ be a morphism of derived algebraic stacks. The \textit{cotangent complex}  of $f$ is the  $(-1)$-connective, quasi-coherent $\sO_X$-module $\BL_{X/Y}$ with the following universal property. For any $f\colon T \to X$ with $T$ affine, and any connective $\sO_T$-module $M$, the space of maps $f^*\BL_{X/Y} \to M$ of $\sO_T$-modules is equivalent to the space of diagram fillers 
	\begin{center}
		\begin{tikzcd}
			T \arrow[r, "f"] \arrow[d] & X \arrow[d] \\
			T[M] \arrow[r] \arrow[ur,dashed] & Y
		\end{tikzcd}
	\end{center}
	
	Let $X \xrightarrow{f} Y \to Z$ be a morphism of derived algebraic stacks. Then we have an exact sequence $
	f^* \BL_{Y/Z} \to \BL_{X/Z} \to \BL_{X/Y}$ in $\QCoh(X)$. Moreover, for $X' \to Y'$ obtained by pulling back a derived algebraic stack $Y'$ over $Y$ along $f$, it holds that $h^*\BL_{X/Y} \simeq \BL_{X'/Y'}$, where $h$ is the structure map $X' \to X$.
	
	For $f\colon X \to Y$ a morphism of derived algebraic stacks, we let $N_{X/Y}$ be the shifted cotangent complex $\BL_{X/Y}[-1]$, and define the \textit{normal bundle} of $f$ as $\mathbb{N}_{X/Y} \coloneqq \VV(N_{X/Y})$. Note that $\VV(N_{X/Y})$ will in general be 2-truncated, contrary to our standing assumption. This does not play a role in what follows.
	
	\section{Overview of derived blow-ups} \label{Sec:Derived Blowups}
	We continue with developing the background for our main results, but specify to the recent theory on derived blow-ups and derived Rees algebras as developed in \cite{KhanRydhDBlowup}, \cite{HekkingGraded}, \cite{Weil} and \cite{GDB}. We also establish a few statements about derived blow-ups and their local models that will be very useful in subsequent sections.

	\subsection{$\infty$-group actions}
	\label{Par:actions}
	Although the construction of the derived blow-up only uses $G$-actions on derived stacks for reductive algebraic groups $G$---which we analyse in-depth in \S \ref{Sec:Derived_equivariant_geometry}---we will use actions by $\infty$-groups in $\Stk$ in \S \ref{Sec:Xmax} for the construction of $X^\maxlocus$. In this paragraph, we therefore review the parts of this theory that we will be using from \cite{NikolausPrincipal},  \cite{LurieHTT}. We omit proofs since we only deviate  slightly from the literature in terminology and notation.
	
	Throughout, let $\HHH$ be an $\infty$-topos. Although we will only use the case $\HHH =\Stk$, presenting the theory in full generality highlights the formal character of the story. 
	
	\begin{definition}
		Let $X\colon\Delta^\op \to \HHH$ be a simplicial diagram in $\HHH$. Then $X$ satisfies the \textit{Segal condition} if, for all $n \geq 0$, the natural map
		\begin{align*}
			\varphi_n \colon X_n \to X_1 \times_{X_0} X_1 \times_{X_0} \cdots \times_{X_0} X_1
		\end{align*}
		is an equivalence, where the morphisms $\varphi_n$ come about from the spine inclusions $(0 \to 1 \to \dots \to n) \to \Delta[n]$. If this is the case, then we say that \textit{morphisms are invertible} in $X$ if the map $X_2 \to X_1 \times_{X_0}X_1$ induced by the horn inclusion $\Lambda^2[2] \to \Delta[2]$ is an equivalence.
		
		$X$ is called an \textit{$\infty$-groupoid} (in $\HHH$) if it satisfies the Segal condition such that morphisms are invertible. If moreover $X_0 \simeq *$, then $X$ is called an \textit{$\infty$-group} (in $\HHH$). Write $\Grp(\HHH)$ for the full subcategory of $\HHH^{\Delta^\op}$ spanned by $\infty$-groups in $\HHH$.
	\end{definition}
	
	For $G$ an $\infty$-group, we think of $G_1$ as the underlying object of $G$, and often conflate the two in our notation. We write $B(*,G)$ to specifically mean the full simplicial diagram $[n] \mapsto G_n \simeq G^{\times n}$, and call it the \emph{bar construction}. 
	
	\begin{example}
		An $\infty$-group(oid) in $\Set$ is a group(oid) in the classical sense. An $\infty$-group in $\Spc$ is a loop space, which one can think of as a topological group where the group structure is relaxed in a homotopy coherent way. If $\HHH$ is the category of derived stacks on an $\infty$-site $\CCC$, then $\Grp(\HHH)$ is equivalent to $\Grp(\Spc)$-valued stacks on $\CCC$. In particular, classical group stacks are $\infty$-groups in $\Stk$.
	\end{example}

	For an $\infty$-category $\CCC$, write $\Arr(\CCC)$ for the arrow category $\Fun(\Delta[1],\CCC)$, and let $\Epi(\CCC)$ be the full subcategory of $\Arr(\CCC)$ spanned by the effective epimorphisms. For an object $T$ in $\CCC$, write $\Epi_{T/}(\CCC)$ for the full subcategory of $\Epi(\CCC)$ spanned by morphisms with source $T$.
	
	\begin{proposition}
		\label{Prop:B_Omega_equiv}
		$\infty$-groupoids in $\HHH$ are effective, meaning that the \v{C}ech nerve construction gives an adjoint equivalence of $\infty$-categories $\check{C} \colon \Epi(\HHH) \simeq \Grpd(\HHH)$,
		with left adjoint induced by taking the colimit. Restricting to $\Grp(\HHH)$ gives an adjoint equivalence
		\begin{align*}
			B: \Grp(\HHH)  \rightleftarrows  \Epi_{*/}(\HHH) :  \Omega 
		\end{align*}
		where $\Omega(* \to X)$ is the derived based loop stack $* \times_X *$ (with suppressed $\infty$-group action). 
	\end{proposition}

	Let $G$ be an $\infty$-group in $\HHH$ and $X$ an object in $\HHH$. Then a \textit{$G$-object} in $\HHH$ with underlying object $X$ is an $\infty$-groupoid $B(X,G)$ over $B(*,G)$ of the form $[n] \mapsto G^n \times X$ such that $d_1\colon G \times X \to X$ is the projection. Informally, we say that $X$ is endowed with a $G$-action $\sigma\colon G \times X \to X$ if there is a $G$-object $B(X,G)$ such that $d_0 \simeq \sigma$, and refer to $B(X,G)$ as the bar-construction of $\sigma$.
	
	The $\infty$-category of $G$-objects is the full subcategory of $\Fun(\Delta^\op, \HHH)_{B(*,G)}$ spanned by $G$-objects. Informally, if $X,Y$ are endowed with a $G$-action, then a \textit{$G$-equivariant morphism} $f\colon X \to Y$ is a morphism of the corresponding $G$-objects $f_\bullet\colon B(X,G) \to B(Y,G)$ such that $f_1 \simeq f$. For a $G$-object $B(X,G)$, we write the colimit of the diagram $[n] \mapsto G^n \times X$ as $[X/G]$. If $X=*$, we simply write $BG$. Then the $\infty$-category of $G$-objects is equivalent to the $\infty$-category $\HHH_{/BG}$. 
	
	Let $f\colon T \to [X/G]$ be a morphism. Pulling back along the projection map, we obtain a $G$-object $P$  and a $G$-equivariant morphism $P \to X$, such that $[P/G] \to [X/G]$ is equivalent to $f$. A morphism $P \to T$ obtained this way is a \textit{$G$-torsor} or \textit{$G$-bundle}.
	
	For an object $V$ in $\HHH$, let $\underline{\Aut}(V)$ be the internal automorphism $\infty$-group of $V$. This is constructed as follows. Up to size-issues which we will ignore, there is an object classifier $\widehat{\Ob(\HHH)} \to \Ob(H)$, meaning that any object $X$ of $\HHH$ fits in a Cartesian diagram
	\begin{center}
		\begin{tikzcd}
			X \arrow[d] \arrow[r] & \widehat{\Ob(\HHH)} \arrow[d] \\
			* \arrow[r] & \Ob(\HHH).
		\end{tikzcd}
	\end{center}
	Then $B\underline{\Aut}(V)$ is the $\infty$-image of the morphism $* \to \Ob(\HHH)$ which classifies $V \to *$, where the $\infty$-image is by definition the colimit of the \v{C}ech nerve of $* \to \Ob(\HHH)$. In particular, $* \to B\underline{\Aut}(V)$ is an effective epimorphism, and $\underline{\Aut}(V)$ is by definition the associated $\infty$-group.
	
	For $G$ an $\infty$-group in $\HHH$, the \textit{adjoint action} of $G$ on itself is defined as the $G$-object associated to the projection $\lL(BG) \to BG$. Here, $\lL(BG)$ is the derived (free) loop stack, which is the self-intersection of the diagonal of $BG$, see \S \ref{Par:Loops}.  We write $G_\ad$ for this $G$-object, which has $G$ as underlying object. By definition, this means that $[G_\ad/G] \simeq \lL (BG)$.

	\begin{definition}
		For $\HHH = \Stk$, we introduce the following terminology. An $\infty$-group object, whose underlying object in $\Stk$ is a derived stack/algebraic stack/scheme, is called a \textit{derived group stack/algebraic stack/scheme}. For a derived group scheme $G$, a \textit{derived closed subgroup scheme} is a morphism $H \to G$ of derived group schemes which is a closed immersion on underlying derived schemes. Beware that, despite the name, a derived closed subgroup scheme $H \to G$ is in general not a monomorphism.
	\end{definition}
	
	\subsection{Derived projective spectra}
	Let $X \in \Stk$ be algebraic, let $\MM$ be either $\NN$ or $\ZZ$. As in the ungraded case, we have an $\infty$-category $\QAlg^\MM(X)$ of $\MM$-graded, quasi-coherent $\sO_X$-algebras. The operations in the affine picture on $\NN$-graded and $\ZZ$-graded algebras mentioned in \S \ref{Par:Graded_derived_rings} carry over to the global picture, for which we use the same notation.
	
	Let $\sB \in \QAlg^\NN(X)$ be given. The \textit{irrelevant ideal} $\sB_+$ of $\sB$ is the fiber of the map $\sB \to \sB_0$. Write $V(\sB_+) \to \Spec \sB$ for the closed immersion $\Spec(\sB_0) \to \Spec \sB$. The $\NN$-grading on $\sB$ induces a $\Gm$-action on $\Spec \sB$, which restricts to a $\Gm$-action on $(\Spec \sB )\setminus V(\sB_+)$. Define the (relative) \textit{projective spectrum} of $\sB$ as the quotient
	\begin{align*}
		\Proj \sB \coloneqq [((\Spec \sB )\setminus V(\sB_+)) / \Gm]
	\end{align*}
	
	Projective spectra commute with base-change in the obvious way. It follows that $\Proj \sB$ is always algebraic.  Moreover, if $\sB$ is generated in homogeneous degree 1 over $\sO_X$, then for all derived schemes $f\colon T \to X$, the derived stack $(\Proj \sB) \times_X T \simeq \Proj(f^*\sB)$ is a derived scheme as well. Here and in what follows, an $A$-algebra $B$ is generated by a submodule $B_0 \subset B$ if the natural map $\LSym_A(B_0) \to B$ is surjective (on $\pi_0$).

	\subsection{Weil restrictions}
	Let $f\colon X \to Y$ be a morphism of stacks. Consider the pullback functor $f^*\colon \Stk_Y \to \Stk_X$. Suppose that $f^*$ has a right adjoint $f_*$. Then, for $Z \to X$, we call $f_*Z$ the \textit{Weil restriction} of $Z$ along $f$, written $\Res_f(Z)$.
	
	\begin{remark}
		Although in general the existence of the Weil restriction $\Res_f(-)$ is only a set-theoretic issue, we remark that it exists without any such issues by \cite[Con.~19.1.2.3]{LurieSpectral} when $f$ is representable. Although Lurie's result is more general, and there are more general statements still, this version covers all cases that we will encounter.
	\end{remark}
	
	Since colimits are universal in any $\infty$-topos (up to size issues), the $\infty$-category $\Stk$ is Cartesian closed. For a morphism of derived stacks $X \to Y$, the right adjoint of $X \times_Y (-)\colon \Stk_Y \to \Stk_Y$ is written $\underline{\Map}_{/Y}(X,-)$, and is called the \textit{derived mapping stack}. Then $\Res_f(Z)$ is alternatively given by the stack
	\begin{align} \label{Eq:weil res}
		\Res_f(Z) = \underline{\Map}_{/Y}(X,Z) \times_{\underline{\Map}_{/Y}(X,X)} Y
	\end{align}
	Under certain finiteness conditions, the derived stack $\Res_fZ$ will be algebraic, see \cite[Thm.~5.1.1]{HalpernleistnerMapping}, \cite{Weil}. In terms of higher derived stacks, the algebraicity of $\Res_fZ$ is controlled by the algebraicity of $Z \to X$ and the Tor-amplitude of $X \to Y$. Instead of giving a general statement, we will indicate the algebraicity of the Weil restriction in particular cases when relevant.

	\subsection{Derived Rees algebras}
	Endow $\AA^1 = \Spec \CC[t^{-1}]$ with the $\Gm$-action which corresponds to declaring $t^{-1}$ to be of homogeneous degree $-1$. For $X \in \Stk$ algebraic, consider the map $\zeta_X\colon B\Gm \times X \to [\AA^1/ \Gm] \times X$ induced by the zero section. Then we can Weil-restrict along $\zeta_X$, i.e., we have a right adjoint $\Res_\zeta\colon \Stk_{B\Gm \times X} \to \Stk_{[\AA^1/\Gm] \times X}$ to the pullback functor. 
	
	Let now $Z \to X$ be a morphism of derived stacks, with $X$ still algebraic. Define the \textit{deformation space} of $Z \to X$ as the Weil restriction $\stD_{Z/X} \coloneqq \Res_{\zeta_X}(Z \times B\Gm)$ of $Z \times B\Gm$ along $\zeta_X$. Write $D_{Z/X}$ for the pullback of $\stD_{Z/X}$ along $\AA^1 \times X \to [\AA^1/\Gm] \times X$. In \cite{Weil} it is shown that the deformation space is not only functorial in $Z$, but also in the morphism $Z \to X$. The resulting functor $\stD_{(-)/(-)}$ is a right adjoint, hence in particular is stable under base-change. Consequently, the same is true for $D_{(-)/(-)}$.
	
	From here on, let $Z \to X$ be a closed immersion of derived algebraic stacks. Then $D_{Z/X} \to \AA^1 \times X$ is affine by \cite{HekkingGraded}, and $D_{Z/X}$ carries a $\Gm$-action such that $[D_{Z/X}/\Gm] \simeq \stD_{Z/X}$. We define the (derived) \textit{extended Rees algebra} of $Z$ over $X$ as the $\ZZ$-graded, quasi-coherent $\sO_X[t^{-1}]$-algebra $\sR_{Z/X}^\ext$ such that $D_{Z/X} = \Spec \sR_{Z/X}^\ext$. We let $\sR_{Z/X} \coloneqq (\sR_{Z/X}^\ext)_{\geq 0}$ be the \textit{Rees algebra} of $Z$ over $X$. In the affine case, say with $Z= \Spec B$ and $X = \Spec A$, we  write $R_{B/A}^\ext, R_{B/A}$ for the algebras corresponding to the quasi-coherent algebras $\sR_{Z/X}^\ext, \sR_{Z/X}$.
	
	The main properties of this construction are as follows. The (extended) Rees algebra is stable under base-change in the obvious way. Since $Z \to X$ is affine, we can write $Z = \Spec \sB$ for some quasi-coherent $\sO_X$-algebra $\sB$. We can then recover $\sB$ from the Rees algebra as $(\sR^\extd_{Z/X}/(t^{-1}))_0 \simeq \sB$. In fact, the pullback of $D_{Z/X} \to \AA^1 \times X$ to $\{0\} \times X$ gives the normal bundle $\mathbb{N}_{Z/X}$, living over $\{0\} \times X$ via the map $Z \to \{0\} \times X$.
	
	\begin{example}
		\label{Ex:Rees}
		Let $A \in \Alg$, let $\sigma_1,\dots,\sigma_k$ be cycles $\sigma_i \in \pi_{n_i}(A)$, and let $B$ be the finite quotient $A/(\sigma_1,\dots,\sigma_k)$. Then 
		\begin{align*}
			R^\ext_{B/A} \simeq \frac{A[t^{-1},u_1,\dots,u_n]}{(u_1t^{-1}-\sigma_1,\dots,u_nt^{-1}-\sigma_n)}
		\end{align*}
		where the $u_i$ are free in homogeneous degree 1 and homological degree $n_i$.
	\end{example}
	
	\subsection{Derived blow-ups and virtual Cartier divisors}
	We continue with the closed immersion $Z \to X$ of derived algebraic stacks. We define the \textit{derived blow-up} of $X$ in $Z$ as the projective spectrum of $\sR_{Z/X}$ over $X$, written $\Bl_ZX$. 
	
	The basic properties here are as follows. The construction is stable under arbitrary base-change in the obvious way, and if $Z, X$ are derived schemes, then so is $\Bl_ZX$. It follows that, in general, $\Bl_ZX$ is a derived algebraic stack. The structure map $\Bl_ZX \to X$ is always an equivalence outside $Z$. Over $Z$, we have the \textit{exceptional divisor $E_{Z}X$}, which is the projective normal bundle $\PP(N_{Z/X}) = [(\mathbb{N}_{Z/X} \setminus 0)/\Gm]$, and lives over $Z \to X$ as a virtual Cartier divisor (see below). We recover the classical blow-up of $X_\cl$ in $Z_\cl$ as the schematic closure of $(\Bl_ZX)_\cl \setminus (E_{Z}X)_\cl$ inside $(\Bl_ZX)_\cl$.

	Let $j\colon D \to T$ be a closed immersion of derived algebraic stacks. Then $j$ is called \textit{quasi-smooth of virtual codimension $n$} if, for any derived scheme $T' \to T$, there  are locally defined sections $f_1,\dots,f_n \in \sO_{T'}$ such that, Zariski locally on $T'$, the pullback $D\times_T T'$ is the derived vanishing locus $V(f_1,\dots,f_n) = \Spec (\sO_{T'}/(f_1,\dots,f_n))$. A \textit{virtual Cartier divisor} is a quasi-smooth closed immersion of virtual codimension 1. 
	
	\begin{example}
		The map $\zeta\colon B\Gm \to [\AA^1/\Gm]$ is a virtual Cartier divisor. In fact, this is the universal example: for any virtual Cartier divisor $j\colon D \to T$ there is a map $g\colon T \to [\AA^1/\Gm]$ which classifies $j$ in the sense that $j$ is the pullback of $\zeta$ along $g$.
	\end{example}
	
	\begin{definition}
		Let $Z \to X$ be a morphism of derived algebraic stacks, and $T$ a derived algebraic stack. A \emph{virtual Cartier divisor over $Z\to X$ (on $T$)} is a $T$-point of $\stD_{Z/X}$.
	\end{definition}
	
	The universality of $\zeta$ implies that a virtual Cartier divisor over $Z \to X$ is a commutative diagram of the form 
	\begin{equation}
		\label{E:VCD}
		\begin{tikzcd}
			D \arrow[d, "g"] \arrow[r] & T \arrow[d] \\
			Z \arrow[r] & X
		\end{tikzcd}
	\end{equation}
	where $D \to T$ is a virtual Cartier divisor. 
	
	\begin{definition}
		If $Z \to X$ is a closed immersion, then a virtual Cartier divisor over $Z \to X$ classified by the diagram \ref{E:VCD} is \emph{strict} when the underlying classical diagram is Cartesian, and the induced homomorphism  $g^\ast N_{Z/X} \to N_{D/T}$ is surjective on $\pi_0$. (This terminology deviates from \cite{KhanRydhDBlowup}, but agrees with \cite{Weil}).
	\end{definition}
	
	The derived blow-up represents strict virtual Cartier divisors over $Z \to X$ in the sense that $\Bl_ZX(T)$ is equivalent to the space of strict virtual Cartier divisors over $Z \to X$ on $T$, with the exceptional divisor as the universal example.

	\subsection{Properties and formulas of derived blow-ups}
	
	We now prove and collect a few additional properties of derived blow-ups, which we will use in Section~\ref{Sec:intr bl is der bl}.
	
	As a consequence of the functorial description of the derived blow-up, we have the following proposition, which shows that the classical truncation $(\Bl_Z U)_\cl$ depends on $Z_{\leq 1} \coloneqq \pi_{\leq 1} Z$, hence in particular is invariant under attaching $n$-cells to $\sO_Z$ for $n \geq 2$.
	
	\begin{proposition} \label{cell attachment in deg -2 prop}
		Suppose that $Z' \to Z$ is a closed embedding such that $Z'_{\leq 1} \to Z_{\leq 1}$ is an equivalence. Then the natural morphism $\Bl_{Z'} U \to \Bl_Z U$ induces an isomorphism on classical truncations.
	\end{proposition}
	
	\begin{proof}
		The classical truncations $(\Bl_{Z'} U)_\cl$ and $(\Bl_Z U)_\cl$ are determined by their $T$-points, where $T$ is a classical affine scheme, and a virtual Cartier divisor $D \to T$ is locally (on $T$) the vanishing locus of a section $f \in \sO_T$. In particular, since $T$ is classical, the truncation $D_{\leq 1} \to D$ is an equivalence. Since also $Z'_{\leq 1} \simeq Z_{\leq 1}$, any map $D \to Z$ factors uniquely through $Z'$.
	\end{proof}
	
	We also have the following simple proposition regarding classical truncations of derived blow-ups and pullback squares.
	
	\begin{proposition} \label{pullback square classical truncation blow-ups}
		Suppose we are given a Cartesian diagram
		\begin{center}
			\begin{tikzcd}
				Z' \arrow[r] \arrow[d] & U' \arrow[d] \\
				Z \arrow[r] & U
			\end{tikzcd}
		\end{center}
		of derived schemes, such that $U'_{\cl} \to U_{\cl}$ is an isomorphism. Then the canonical map $(D_{Z'/U'})_\cl \lr (D_{Z/U})_\cl$ is an isomorphism.
	\end{proposition}
	
	\begin{proof}
		Since the deformation space is stable under pullbacks, we have
		\begin{align*}
			(D_{Z'/U'})_\cl \simeq (D_{Z/U} \times_U U')_\cl \simeq (D_{Z/U})_\cl
		\end{align*}
		where the last equivalence follows since $(-)_\cl$ commutes with pullbacks, and since $U_\cl \simeq U'_\cl$ by assumption. 
	\end{proof}
	
	In order to perform explicit computations, we will also need the following lemma. Note that $\stD_{U/U} \simeq Y \times [\AA^1/\Gm]$, hence $D_{U/U} \simeq Y \times \AA^1$.
	
	\begin{lemma} \label{functoriality of deformation spaces lemma}
		Let a sequence $Z \to U \to V$ of derived stacks be given. Then the natural map
		\begin{align*}
			\stD_{Z/U} \to \stD_{U/U} \times_{\stD_{U/V}} \stD_{Z/V}
		\end{align*}
		induced by naturality of $\stD_{(-)/(-)}$, is an equivalence.
		Consequently, the same holds true if we replace $\stD_{(-)/(-)}$ by $D_{(-)/(-)}$ and 	$\Stk_{[\AA^1/\Gm]}$ by $\Stk_{\AA^1}^{\Gm}$.
	\end{lemma}
	
	\begin{proof}
		This will appear in \cite{Weil}.
	\end{proof}
	
	Suppose now that $A$ is any algebra and we are given a sequence of closed embeddings $Z \to W \to U \to V = \Spec A$, corresponding to
	\begin{align*}
		A \to B \to C \to D
	\end{align*}
	where $B = A / (f_1, ..., f_n)$ for certain $f_i \in \pi_0 A$, and $C = B / (\sigma_1, ..., \sigma_k)$ for certain $\sigma_i \in \pi_{d_i} B$, and $D = A / (g_1, ..., g_m)$ for certain $g_i \in \pi_0A$. 
	
	We use multi-index notation, so that $B = A/(\underline{f})$ and so on. From the theory of Rees algebras, we have a diagram
	\begin{center}
		\begin{tikzcd}
			(R^\extd_{U/V})_1 \arrow[r] \arrow[d] & A \arrow[r] \arrow[d] & B \arrow[d] \\
			(R^\extd_{Z/V})_1 \arrow[r]  & A \arrow[r] & D,
		\end{tikzcd}
	\end{center}
	where the rows are fiber sequences. Write $R^\ext_{U/V} = A[t^{-1},\underline{v}]/(t^{-1}\underline{v} - \underline{f})$, and $R^\ext_{Z/V} = A[t^{-1},\underline{w}]/(t^{-1}\underline{w} - \underline{g})$, and $p$ for the composition $\pi_0A[t^{-1},\underline{v}] \to \pi_0R^\ext_{U/V} \to \pi_0R^\ext_{Z/V}$. Here, $\underline{v} = (v_1,\dots,v_n)$ and $\underline{w} = (w_1, \dots, w_m)$ are free in homological degree 0 and homogeneous degree 1.
	
	For each $i$, let $v_i'$ be the image of $v_i$ under $p$. Since $p$ sends $t^{-1}v_i - f_i$ to zero, we get $\lambda_{ji} \in \pi_0A[t^{-1},\underline{w}]$ such that
	\begin{align*}
		t^{-1}v_i' - f_i = \sum\nolimits_j \lambda_{ji}(t^{-1}w_j-g_j)
	\end{align*}
	holds in $\pi_0A[t^{-1},\underline{w}]$. Putting $t^{-1} = 0$ gives us that $f_i = \sum_j \lambda_{ji} g_j$, in particular that $\lambda_{ji} \in \pi_0A$. Putting $t^{-1}=1$ then gives $v_i' = \sum_j \lambda_{ji} w_j$. It follows that $R_{U/V}^\extd \to R_{Z/V}^\extd$ sends $v_i$ to $\sum_j \lambda_{ij} g_j$, in the sense that precomposing with $A[t^{-1},\underline{v}] \to R_{U/V}^\extd$ gives the unique map $A[t^{-1},\underline{v}] \to R_{Z/V}^\ext$ of $A[t^{-1}]$-algebras that sends $v_i$ to $\sum_j \lambda_{ij} g_j$.

	Write $R^\extd_{W/U} = B[t^{-1}\underline{u}]/(t^{-1}\underline{u} - \underline{\sigma})$ where $\underline{u} = (u_1,\dots,u_k)$ are free in homogeneous degree 1 and homological degrees $d_1,\dots,d_k$. Then, with the same argument as before, we get elements $(\tau_1,\dots,\tau_k)$ in $\pi_*R_{Z/U}^\extd$, with $\tau_i$ in homological degree $d_i$ and homogeneous degree 1 such that $\pi_*R^\extd_{W/U} \to \pi_*R^\ext_{Z/U}$ sends $u_i$ to $\tau_i$.

	\begin{proposition}
		\label{Prop:RZU}
		We have 
		\begin{align} \label{formula [-1,0] case}
			R^\extd_{Z/U} \simeq \frac{A[t^{-1},\underline{w}]}{(t^{-1}\underline{w}-\underline{g}, \sum_j \lambda_{1j}w_j,\dots,\sum_j \lambda_{nj}w_j)}
		\end{align}
		where $\underline{w} = (w_1,\dots,w_m)$ are in homological degree 0 and homogeneous degree 1. Likewise, we have
		\begin{align} \label{formula [-2,0] case}
			R^\extd_{Z/W} \simeq R^\extd_{Z/U}/(\underline{\tau})
		\end{align}
		where $\underline{\tau} = (\tau_1,\dots,\tau_k)$ are the elements with $\tau_i$ in homological degree $d_i$ and homogeneous degree 1 described above.
	\end{proposition}
	
	\begin{remark}
		As the proof will make clear, there is only a contractible space of choices involved here, since the equivalence comes about from a \textit{description} of canonically given maps.
	\end{remark}
	
	\begin{proof}
		By Lemma~\ref{functoriality of deformation spaces lemma}, we have $R^\extd_{Z/U} \simeq R^\extd_{U/U} \otimes_{R^\extd_{U/V}} R^\extd_{Z/V}$. We thus get the following diagram, where the transition maps in the bottom right corner are the canonical maps induced by functoriality of $R^\ext_{(-)/(-)}$.
		\[
		\xymatrix{%
			\ZZ[\underline{F}] \ar[rr]^{\underline{F} \mapsto 0} \ar[d]^{\underline{F} \mapsto t^{-1}\underline{v}-\underline{f}} && \ZZ \ar[d] \\
			A[t^{-1},\underline{v}] \ar[rr] \ar[d]^{\underline{v} \mapsto 0} && \frac{A[t^{-1},\underline{v}]}{(t^{-1}\underline{v}-\underline{f})} \ar[d] \ar[rr]^{v_i \mapsto \sum_j \lambda_{ij}w_j} && \frac{A[t^{-1},\underline{w}]}{(t^{-1}\underline{w} - \underline{g})} \ar[d] \\
			A[t^{-1}] \ar[rr] && B[t^{-1}] \ar[rr] && R_{Z/U}^\extd }
		\]   
		where all squares are pushouts, and $\underline{v} = (v_1,\dots,v_n)$ are free in homogeneous degree 1 and homological degree 0, and likewise for $\underline{F} = (F_1,\dots,F_n)$. The first claim follows.
		
		For the second claim, we use $R^\extd_{Z/W} \simeq R^{\extd}_{W/W} \otimes_{R^\extd_{W/U}} R^\extd_{Z/U}$ to get the diagram
		\[
		\xymatrix{%
			\ZZ[\underline{U}] \ar[rr]^{\underline{U} \mapsto \underline{u}t^{-1}-\underline{\sigma}} \ar[d]^{\underline{U} \mapsto 0} && B[t^{-1},\underline{u}] \ar[d] \ar[rr]^{\underline{u} \mapsto 0} && B[t^{-1}] \ar[d] \\
			\ZZ \ar[rr] && \frac{B[t^{-1},\underline{u}]}{(\underline{u}t^{-1}-\underline{\sigma})} \ar[d]^{\underline{u} \mapsto \underline{\tau}} \ar[rr] && C[t^{-1}] \ar[d] \\
			&& R^\extd_{Z/U} \ar[rr] && R_{Z/W}^\extd }
		\]	
		where $\underline{U} = (U_1,\dots,U_k)$ are free in homogeneous degree 1 and homological degrees $d_1,\dots,d_k$. Since all squares are pushouts, the second claim follows. 
	\end{proof}

	\section{Equivariant derived algebraic geometry}
	
	\label{Sec:Derived_equivariant_geometry}
	A main ingredient in the canonical reduction of stabilizers algorithm for algebraic stacks from \cite{EdidinRydh} is the locus $X^{\maxlocus}$ of points with maximal-dimensional stabilizers, for a given algebraic stack $X$. This locus is a closed substack of $X$ in the case that $X$ admits a good moduli space, as shown in \cite[Prop.~B.4]{EdidinRydh}.
	
	In the next section, \S \ref{Sec:Xmax}, we give  a derived analogue of $X^{\maxlocus}$ in the case that $X$ is a derived algebraic stack for which $X_\cl$ admits a good moduli space. In order to do this, in this section we give a treatment of several elements of equivariant geometry and study fixed loci in the derived setting. This allows us to obtain equivariant versions of the results of the previous section, including equivariant derived blow-ups. We note that there is some overlap between what we present here and what is presented in \cite[\S A.3]{LocThm}. 
	
	The statements of this section are derived analogues of familiar statements from classical equivariant geometry, and the arguments follow similar lines to arguments for derived analogues of ordinary (i.e., non-equivariant) classical geometry.
	
	\subsection{Equivariant rings, modules and stacks}
	\label{Par:Equivariant_rings_modules_and_stacks}
	Fix a reductive algebraic group $G$ over $\CC$. By definition, $G$ is a linear algebraic group, hence a smooth closed subscheme of some $\GL_n$ (smooth by Cartier's theorem). A \textit{classical $G$-module $V$} will be an algebraic representation of $G$ on a vector space $V$, i.e., a morphism of group schemes $G \to \GL(V)$. Write $\Rep^G$ for the category of finite dimensional $G$-modules $V$. Recall, since algebraic representations are rational, every $M \in \Rep^G$ splits into a direct sum of simple modules.

	\begin{definition}
		Let the symmetric monoidal $\infty$-category of \textit{$G$-modules}, written $\Mod^G$, be the stabilization of $\PPP_\Sigma(\Rep_G)$ with symmetric monoidal structure induced from $\Rep_G$, and let $\Alg^G$ be the $\infty$-category of connective $\EE_\infty$-algebras in $\Mod^G$.
	\end{definition}
	Recall that the heart $\QCoh(BG)^\heartsuit$ is the category of classical $G$-modules. In particular, we have a canonical inclusion $\Rep_R \to \QCoh(BG)$. Since $\QCoh(BG)_{\geq 0}$ is generated under colimits by the image of $\Rep_G \to \QCoh(BG)$, the canonical map $\PPP_\Sigma(\Rep_G) \to \QCoh(BG)_{\geq 0}$ is an equivalence, hence $\Mod^G \simeq \QCoh(BG)$. We endow $\Mod^G$ with the canonical $t$-structure induced from $\QCoh(BG)$.
	
	Since the $\infty$-category $\Mod^G_{\geq 0}$ of connective $G$-modules is presentable and the tensor product of $\Mod^G$ commutes with colimits in each variable separately, the forgetful functor $\Alg^G \to \Mod^G_{\geq 0}$ admits a left adjoint,  written as
	\begin{align*}
		\LSym^G\colon \Mod^G_{\geq 0} \to \Alg^G.
	\end{align*}
	Let $\Poly^G$ be the full subcategory of $\Alg^G$ spanned by objects of the form $\LSym^G(M)$, for $M \in \Rep^G$. Then, since we are in characteristic zero and the $G$-actions we consider are rational, $\Alg^G$ is freely generated under sifted colimits by $\Poly^G$.  
	
	The forgetful functor $\Mod^G \to \Mod$ is symmetric monoidal, which gives us a canonical forgetful functor $\Alg^G \to \Alg$. Also, the functor $\Mod \to \Mod^G$ which endows a module with trivial $G$-action preserves colimits. We write the right adjoint of this functor as $(-)^{\fix}$. Note that $(-)^\fix$ corresponds to taking the classical fixed part on $\Rep^G$.

	\begin{definition}
		Write $\Aff^G$ for the $\infty$-category $(\Alg^G)^\op$. Endow $\Aff^G$ with the \textit{{\'{e}tale topology}} by declaring a family $\{U_i\to X\}$ in $\Aff^G$ to be a cover if the underlying family of derived affine schemes is an \'{e}tale cover. Write $\Stk^G$ for the $\infty$-category of derived stacks on $\Aff^G$ with respect to the \'{e}tale topology, and $\Spec^G\colon (\Alg^G)^\op \to \Stk^G$ for the Yoneda embedding.
	\end{definition}

	For a derived stack $X$, write $\Aff(X)$ for the $\infty$-category of stacks which are affine over $X$, and $\QAlg(X)$ for the $\infty$-category of quasi-coherent $\sO_X$-algebras, so that we have an equivalence
	\begin{align*}
		\Spec(-) \colon \QAlg(X)^\op \to \Aff(X).
	\end{align*}
	Since we are in characteristic zero, $\QAlg(X)$ is also the $\infty$-category of connective $\EE_\infty$-algebras in $\QCoh(X)$. The case $X=BG$ therefore gives us
	\begin{align*}
		\Aff^G \simeq \Aff(BG)
	\end{align*}
	
	Recall that pulling back along the quotient map $\varphi\colon* \to BG$ induces an equivalence between $\Stk_{BG}$ and the $\infty$-category of stacks with a $G$-action. Since $\varphi$ is affine, this implies that $\Aff^G$ is also the $\infty$-category of affine derived schemes with a $G$-action.

	\begin{proposition}
		\label{Prop:StkG_is_StkBG}
		The $\infty$-category $\Stk^G$ is equivalent to $\Stk_{BG}$.
	\end{proposition}
	\begin{proof}
		Let $T$ be any derived stack. Since $T$ has affine diagonal,  the $\infty$-category $\Aff_{/T}$ of derived affine schemes over $T$ is a full subcategory of $\Aff(T)$. Composition with the inclusion $f \colon \Aff_{/T} \to \Aff(T)$ and left Kan extension gives an adjunction
		\begin{align*}
			f_! : \Stk(\Aff_{/T}) \rightleftarrows \Stk(\Aff(T)) : f^*
		\end{align*}
		Since $f$ is fully faithful, so is $f_!$. Moreover, since $\Stk_T \simeq \Stk(\Aff_{/T})$, the image of $f_!$ contains all stacks over $T$. Since $\Stk(\Aff(T))$ is generated under colimits by stacks affine over $T$, it follows that $f_!$ is essentially surjective. Hence $\Stk_{/T} \simeq \Stk(\Aff(T))$, and putting $T = BG$ gives us what we want.
	\end{proof}
	
	Although we will use the equivalence $\Stk^G \simeq \Stk_{BG}$ seamlessly, we will not conflate a derived stack $X$ over $BG$ with the corresponding derived $G$-stack $P \coloneqq X \times_{BG} \{*\}$. In particular, if $X$ is endowed with a $G$-action as a derived stack over $BG$, then this contains no information regarding the action of $P$. 
	
	\begin{notation}
		\label{Not:BOmega}
		Let $G,G'$ be reductive groups. The equivalences $\Stk^G \simeq \Stk_{BG}$ and $\Stk^{G'} \simeq \Stk_{BG'}$ induce equivalences on functor-$\infty$-categories, which we write as
		\begin{align*}
			B: \Fun(\Stk^G,\Stk^{G'}) \leftrightarrows \Fun(\Stk_{BG}, \Stk_{BG'}) : \Omega
		\end{align*}	
	\end{notation}
	
	\begin{remark}
		Consider the map $\psi\colon BG \to *$. Then the functor $\Stk \to \Stk_{BG}$ given by pulling back along $\psi$ is right adjoint to the functor $\Stk_{BG} \to \Stk$ given by composing with $\psi$. Under the equivalence $\Stk_{BG} \simeq \Stk^G$, this corresponds to the adjunction
		\begin{align*}
			[-/G]: \Stk^G \rightleftarrows \Stk: \iota
		\end{align*}
		where $\iota$ endows a derived stack with the trivial action.

		We also have an adjunction
		\begin{align*}
			(-) \times G: \Stk \leftrightarrows \Stk^G:U 
		\end{align*}
		induced by $\varphi \colon * \to BG$, where $U$ is the forgetful functor, and the $G$-action on $Y \times G$ is given by translation, for $Y \in \Stk$. Essentially, this is because any $G$-equivariant map $Y \times G \to X$ is determined by its restriction to $Y \times \{e\} \to X$.
	\end{remark}

	Let $B \in \Alg^G$. Define the $\infty$-category $\Mod^G_B$ of \textit{$(G,B)$-modules} as the $\infty$-category of $B$-modules in $\Mod^G$, with canonical $t$-structure. Write $\Alg^G_B$ for the slice category $(\Alg^G)_{B/}$. As in the absolute case, we have an adjunction
	\begin{align*}
		\LSym^G_B : \Mod^G_B \leftrightarrows \Alg^G_B : U
	\end{align*}
	where $U$ is the forgetful functor. Let $\Rep^G_B$ be the full-$\infty$-subcategory of $\Mod^G_B$ spanned by objects of the form $B \otimes M$, where $M \in \Rep^G$. As shown in \cite{GDB}, $\Mod_B^G$ is again a derived algebraic context, and $\Rep^G_B$ generates $(\Mod^G_B)_{\geq 0}$ under sifted colimits.

	As in the non-equivariant case, we can globalize this picture via a right Kan extension to a functor
	\begin{align*}
		\QCoh^G(-) \colon (\Stk^G)^\op \to \Cat
	\end{align*}
	such that $\QCoh^G(X)$ is stable, presentable and symmetric monoidal for $X$ algebraic. Moreover, for $f\colon X \to Y$ an equivariant map of derived algebraic stacks, we have an equivariant pullback functor
	\begin{align}
		\label{Eq:equiv_pb}
		(f^G)^* \colon \QCoh^G(Y) \to \QCoh^G(X)
	\end{align}
	which is symmetric monoidal and a left adjoint.
	
	With the same procedure, we get a functor $\QAlg^G(-)$. Now for $X \in \Stk^G$ and $\sA \in \QAlg^G(X)$, it holds that $\sA$ is an $\EE_\infty$-algebra in $\QCoh^G(X)$. We can thus define the $\infty$-category $\QCoh_{\sA}^G(X)$ of $(G,\sA)$-modules as the $\infty$-category of $\sA$-modules in $\QCoh^G(X)$. The subcategory of $\QCoh_{\sA}^G(X)$ spanned by modules with trivial $G$-action is written $\QCoh_{\sA}(X)$, which is equivalent to the $\infty$-category of $\sA$-modules in $\QCoh(X)$.

	\subsection{Restriction and co-induction}
	\label{Par:Restriction_and_coinduction}
	Suppose that a reductive algebraic group $G$ acts trivially on $A \in \Alg$. Then $\Mod^G_A$ is equivalent to the $\infty$-category of $A$-modules endowed with a $G$-action, i.e., we have an equivalence
	\begin{align}
		\label{Eq:ModGA}
		\Mod^G_A \simeq \Mod^G \otimes_{\Mod} \Mod_A \simeq \Fun_\sigma((\Rep^G)^\op,\Mod_A)
	\end{align}
	where $\Fun_\sigma((\Rep^G)^\op,\Mod_A)$ is the $\infty$-category of functors $(\Rep^G)^\op \to \Mod_A$ that send finite coproducts in $\Rep^G$ to products in $\Mod_A$, and $(-)\otimes(-)$ is the Lurie tensor product, see \cite[Rem.~2.2.8]{RaksitHKR}.
	
	Let now $f\colon G \to G'$ be a homomorphism of reductive algebraic groups, and let $G'$ also act trivially on $A$. Write the induced morphism on stacks as $\varphi\colon \Spec A \times BG \to \Spec A \times BG'$, and the adjunction $\QCoh(\Spec A \times BG') \leftrightarrows \QCoh(\Spec A \times BG)$ induced by $\varphi$ as
	\begin{align*}
		\Res_\varphi : \Mod^{G'}_A \leftrightarrows \Mod^G_A : \cInd_\varphi
	\end{align*}
	We call $\Res_\varphi$ the \textit{restriction} functor, and $\cInd_\varphi$ the \textit{co-induction} functor.
	
	Recall (see, e.g., \cite[\S 3.2]{GaitsgoryStudy}), that $\Res_\varphi$ is symmetric monoidal and that $\cInd_\varphi$ is right-lax symmetric monoidal. Since
	$BG'$ has affine and flat diagonal, $BG\to BG'$ is flat and cohomologically
	affine. It follows that the functors $\Res_\varphi$ and $\cInd_\varphi$
	are $t$-exact. The adjunction therefore lifts to the level of algebras, which we again write as
	\begin{align*}
		\Res_\varphi: \Alg^{G'}_A \leftrightarrows \Alg^G_A : \cInd_\varphi
	\end{align*}
	
	\begin{notation}
		\label{Not:Res-cInd}
		Consider the obvious maps $\varphi\colon \Spec A \to \Spec A \times BG'$ and $\psi\colon \Spec A \times BG \to \Spec A$. Then the functor $\Res_\varphi\colon \Mod^{G'}_A \to \Mod_A$ is forgetting the $G'$-action, which we will write as $(-)^u$;  the functor $\Res_\psi\colon \Mod_A \to \Mod^G_A$ endows an $A$-module with trivial $G$-action, which we will write as $(-)^\tv$; and $\cInd_\psi\colon \Mod^G_A \to \Mod_A$ sends an $A$-module with $G$-action $M$ to the $G$-fixed points $M^{\fix }$. Again, we use the same notation for the induced functors on algebras.
	\end{notation}

	Globalizing $(-)^\tv \dashv (-)^\fix$ via right Kan extensions  gives us an adjunction
	\begin{align}
		\label{Eq:i_fix_adj}
		i:\QCoh(X) \leftrightarrows \QCoh^G(X) : (-)^{\fix}
	\end{align}
	for any stack $X$ with trivial $G$-action. 
	
	\begin{definition}
		Let $X \in \Stk$ algebraic with trivial $G$-action and $\sF \in \QCoh^G(X)$ be given. Consider the counit $\sF^\fix \to \sF$ of the adjunction $i \dashv (-)^\fix$. The \textit{moving part} of $\sF$ is the cofiber of the counit, written $\sF \to \sF^\mv$. We write $\QCoh^\mv(X)$ for the essential image of the functor $(-)^\mv \colon \QCoh^G(X) \to \QCoh^G(X)$.
	\end{definition}
	
	Recall that a sequence of colimit-preserving functors 
	\begin{align*}
		\CCC \xrightarrow{i} \DDD \xrightarrow{g} \EEE
	\end{align*}
	between presentable stable $\infty$-categories is \textit{exact} if $gi \simeq 0$, $i$ is fully faithful, and $\EEE$ is the Verdier quotient $\DDD/\CCC$, i.e., it is the localization of $\DDD$ at the morphisms for which the cofiber is in $\CCC$. Such a sequence is \textit{split-exact} if the right adjoints $i \dashv f$ and $g \dashv j$, which exist by the adjoint functor theorem, are such that $fi \simeq \id_{\CCC}$ and $gj \simeq \id_{\EEE}$. 
	
	\begin{remark}
		\label{Rem:Split_exact}
		The proof of the third point of Proposition \ref{Prop:OXG-modules} is inspired by \cite[App.~A]{BerghRydhDestack}, and uses the following facts on (split) exact sequences that can be found in the detailed exposition \cite[App.~A]{MoiHermitian} on the matter. 
		
		A sequence $\CCC \xrightarrow{i} \DDD \xrightarrow{g} \EEE$ with adjoints $j,f$ as above is exact if and only if $i$ is a fully faithful functor with essential image spanned by the objects $M \in \DDD$ such that $gM \simeq 0$. If the sequence is moreover split, then the sequence $\EEE \xrightarrow{j} \DDD \xrightarrow{f} \CCC$ is also a split exact sequence. In this case, the square 
		\begin{center}
			\begin{tikzcd}
				M \arrow[r] \arrow[d] & jg(M) \arrow[d] \\
				if'(M) \arrow[r] & if'jg(M)
			\end{tikzcd}
		\end{center}
		induced by the units of $g \dashv j$ and $f'\dashv i$, is Cartesian, for any $M \in \DDD$.
	\end{remark}

	\begin{proposition}
		\label{Prop:OXG-modules}
		Let $X$ be a derived algebraic stack with a $G$-action and $\varphi_X \colon X \to [X/G]$ the projection map. Consider the equivariant pullback functor 
		\begin{align*}
			(\varphi_X^G)^* \colon \QCoh^G([X/G]) \to \QCoh^G(X). 
		\end{align*}
		as displayed in \ref{Eq:equiv_pb}, where the $G$-action on $[X/G]$ is trivial.
		\begin{enumerate}
			\item Composing $(\varphi_X^G)^*$ with $i_{[X/G]} \colon\QCoh([X/G]) \to \QCoh^G([X/G])$ from (\ref{Eq:i_fix_adj})  induces an equivalence $$\QCoh([X/G]) \simeq \QCoh^G(X).$$
			\item Suppose that the $G$-action on $X$ is trivial and let $\psi_X\colon [X/G] \to X$ be the structure map. Composing the adjunction $\psi_X^* \dashv {\psi_X}_*$ with the equivalence from (i) yields the adjunction
			\begin{align*}
				i: \QCoh(X) \leftrightarrows \QCoh^G(X) : (-)^{\fix}
			\end{align*}
			from (\ref{Eq:i_fix_adj}).
			\item Suppose that the $G$-action on $X$ is trivial. Then the sequence
			\begin{align*}
				\QCoh(X) \xrightarrow{\quad i\quad} \QCoh^G(X) \xrightarrow{(-)^\mv} \QCoh^\mv(X)
			\end{align*}
			is split exact.
		\end{enumerate}	
	\end{proposition}
	\begin{proof} 
		\noindent\textit{(i).} By descent, this is a local question, so we may assume that $X$ is affine, say $X = \Spec A$. Let $\sA \in \QAlg(BG)$ correspond to $A \in \Alg^G$ under the equivalence $\Alg^G \simeq \QAlg(BG)$, so that $[X/G] = \Spec \sA$ and $\QCoh([X/G]) \simeq \QCoh_{\sA}(BG)$. Now the $G$-actions on $BG$, on $[X/G]$, and on $\sA$ are all trivial. Hence, with the same argument as in (\ref{Eq:ModGA}), we also have 
		\begin{align*}
			\QCoh^G([X/G]) \simeq \QCoh_{\sA}^G(BG) \simeq \QCoh^G(BG) \otimes_{\QCoh(BG)} \QCoh_{\sA} (BG)
		\end{align*}
		Moreover, $\QCoh^G(X) \simeq \Mod^G_A \simeq  \QCoh_{\sA}(BG)$ holds by definition of $\Mod^G_A$. The composition $(\varphi_X^G)^* \circ i_{[X/G]}$ is then equivalent to the composition
		\begin{align*}
			\QCoh_{\sA}(BG) \to \QCoh^G(BG) \otimes_{\QCoh(BG)} \QCoh_{\sA}(BG) \to \QCoh_{\sA}(BG)
		\end{align*}
		where the first arrow sends $M \in \QCoh_{\sA}(BG)$ to the external tensor product $\sO_{BG} \boxtimes M$, where $\sO_{BG}$ has trivial $G$-action, and the second arrow sends $N \boxtimes K \in \QCoh^G(BG) \otimes_{\QCoh(BG)} \QCoh_{\sA}(BG)$ to $N^u \otimes K$. The claim follows.
		
		\medskip
		\noindent\textit{(ii).}
		It suffices to show that $i \simeq (\varphi^G_X)^* \circ i_{[X/G]}  \circ \psi_X^*$. Since the action on $X$ is now trivial, we have $\QCoh([X/G]) \simeq \QCoh(BG) \otimes_{\Mod} \QCoh(X)$ by \cite[Prop.~3.3.5]{GaitsgoryStudy}, and  $\psi_X^*$ is equivalent to the functor
		\begin{align*}
			\sO_{BG} \boxtimes (-) \colon  \QCoh(X) \to \QCoh(BG) \otimes_{\Mod} \QCoh(X)
		\end{align*}
		Since $(\varphi^G_X)^* \circ i_{[X/G]}$ is equivalent to the functor $N \boxtimes K \to N \otimes K$ from $\QCoh(BG) \otimes_{\Mod} \QCoh(X)$  to $\QCoh^G(X)$,  the claim follows.
		
		\medskip
		\noindent\textit{(iii).}
		The functor $(-)^{\mv}\colon\QCoh^G(X) \to \QCoh^\mv(X)$ preserves colimits, hence has a right adjoint, say $j\colon \QCoh^\mv(X) \to \QCoh^G(X)$. For $M \in \QCoh^\mv(X)$, take $N \in \QCoh^G(X)$ such that $N^\mv \simeq M$. Consider the following commutative diagram
		\begin{center}
			\begin{tikzcd}
				N^\fix \arrow[r] \arrow[d] & N^\fix \arrow[r] \arrow[d] & (jM)^\fix \arrow[d] \\
				N^\fix \arrow[r] \arrow[d] & N \arrow[r] \arrow[d] & jM \arrow[d] \\
				0 \arrow[r] & N^\mv \arrow[r] & (jM)^\mv
			\end{tikzcd}
		\end{center}
		The middle row is exact, hence so is the top one by exactness and idempotency of $(-)^\fix$, and the right column is exact by definition.  It follows that $(jM)^\fix = 0$, hence $jM \simeq (jM)^\mv$. The diagram also shows that $(-)^\mv \circ (-)^\fix \simeq 0$. 
		
		By construction, for any $K \in \QCoh^G(X)$, it holds that $K^\mv \simeq 0$ if and only if $K^\fix \to K$ is an equivalence.	Since $i$ is fully faithful, the sequence
		\begin{align}
            \label{Eq:fix_mv_split}
			\QCoh(X) \xrightarrow{i} \QCoh^G(X) \xrightarrow{(-)^\mv} \QCoh^\mv(X)
		\end{align}
		is thus split exact, by Remark \ref{Rem:Split_exact}.
	\end{proof}

	\begin{corollary}
		\label{Cor:fix_mv_cofiber}
		Let $G$ act trivially on a derived algebraic stack $X$. Then for each $\sF \in \QCoh^G(X)$ the canonical sequence
		\begin{align}
			\label{Eq:fix_mv_cofiber}
			\sF^\fix \to \sF \to \sF^\mv
		\end{align}
		is split. Consequently, $(-)^\fix$ is also a left adjoint of $i$.
	\end{corollary}
	
	\begin{proof}
        Since the sequence (\ref{Eq:fix_mv_split}) is split, we have a left adjont $f' \dashv i$, and for any $\sF \in \QCoh^G(X)$ a Cartesian diagram
        \begin{center}
            \begin{tikzcd}
                \sF \arrow[r] \arrow[d] & j\sF^\mv \arrow[d] \\
                if'\sF \arrow[r] & if'j\sF^\mv
            \end{tikzcd}
        \end{center}
        For both claims, it suffices to show that $if'j\sF^\mv \simeq 0$. Since $f', i$ are both left adjoint, this reduced to the classical case, which is know. 
	\end{proof}

	\subsection{Relative fixed loci of derived stacks}
	Let $f\colon G \to G'$ be a surjective homomorphism of reductive algebraic groups, and $\varphi\colon BG \to BG'$ the induced map on stacks. 
	\begin{definition}
		Let $f_*\colon \Stk^G \to \Stk^{G'}$ be the functor that sends a derived stack $X\colon \Alg^G \to \Spc$ to 
		\begin{align*}
			\Alg^{G'} \xrightarrow{\Res_\varphi} \Alg^G \xrightarrow{X} \Spc
		\end{align*}
		Let $B, \Omega $ be as in Notation \ref{Not:BOmega}. Write $\varphi_*$ for the functor $Bf_* \colon \Stk_{BG} \to \Stk_{BG'}$. Conversely, let $\varphi^*\colon\Stk_{BG'} \to \Stk_{BG}$ be the pullback functor along $\varphi$, and let $f^*$ be the functor $\Omega \varphi^* \colon \Stk^{G'} \to \Stk^G$.
	\end{definition}
	
	\begin{lemma}
		\label{Lem:Res_fully_faith}
		The functor $\Res_\varphi\colon \Alg^{G'} \to \Alg^G$ is fully faithful.
	\end{lemma}
	
	\begin{proof}
		Since  $\Res_\varphi\colon \Mod^{G'} \to \Mod^G$ is $t$-exact, we can restrict it to a functor between connective objects, written $H\colon \Mod^{G'}_{\geq 0} \to \Mod^G_{\geq 0}$.  It suffices to show that $H$ is fully faithful. Let $h\colon\Rep^{G'} \to \Rep^G$ be the classical restriction functor.
		
		Since $H$ preserves colimits, it is determined by its restriction to $\Rep^{G'}$. Since $H$ restricted to $\Rep^{G'}$ is $h$, we see that $H$ is the left derived functor of $h$ by \cite[Prop.~5.5.8.15]{LurieHTT}. By \cite[Prop.~5.5.8.22]{LurieHTT}, it therefore suffices to show that $H \colon \Rep^{G'} \to \Mod^G_{\geq 0}$ is fully faithful and with essential image contained in the compact projective objects of $\Mod^G_{\geq 0}$.
		
		We already know that the essential image of $H$ on $\Rep^{G'}$ consists of compact projective objects, since $H=h$ on $\Rep^{G'}$. What remains to show is that $h$ is fully faithful, which follows from the fact that $f$ is surjective.
	\end{proof}

	\begin{proposition}\mbox{}
		\label{Prop:Equiv_Weil_Adj}
		The functors $f^*,f_*$ fit into an adjunction
		\begin{align*}
			f^* : \Stk^{G'} \leftrightarrows \Stk^G : f_*
		\end{align*}
		where $f^*$ is fully faithful. Consequently, the functor $\varphi_*$ is the Weil restriction along $\varphi \colon BG \to BG'$.
	\end{proposition}
	\begin{proof}
		Take a left adjoint $F$ to $f_*$ via left Kan extensions. Since $\Res_\varphi$ is fully faithful by Lemma \ref{Lem:Res_fully_faith}, it holds that $F$ and $f^*$ agree when restricted to $\Aff^{G'}$. Since both functors commute with colimits and $\Aff^{G'}$ generates $\Stk^{G'}$ under colimits, $F$ and $f^*$ therefore agree on all of $\Stk^{G'}$.
	\end{proof}

	\subsection{Absolute fixed loci of derived stacks}
	\label{Subsec:Absolute_fixed_loci}
	Consider now, for a reductive algebraic group $G$, the adjunction induced by the canonical map $G \to *$ via Proposition \ref{Prop:Equiv_Weil_Adj}, written as 
	\begin{align*}
		\iota: \Stk \leftrightarrows \Stk^G : (-)^G
	\end{align*}
	\begin{remark} 
		\label{Rem:HLP}
		Let $U \in \Stk^G$ and put $X \coloneqq [U/G]$. Observe that the derived mapping stack $\underline{\Map}(BG, X)$ exists without any set-theoretic issues. Define $\underline{\Map}^G(\ast, X)$ via the Cartesian diagram
		\begin{align*}
			\xymatrix{
				\underline{\Map}^G(\ast, X) \ar[r] \ar[d] & \underline{\Map}(BG, X) \ar[d]^-{(X \to BG)_*} \\
				\ast \ar[r] & \underline{\Map}(BG, BG).
			}
		\end{align*}
		Proposition~\ref{Prop:Equiv_Weil_Adj} and the formula~\eqref{Eq:weil res} for the Weil restriction imply that $\underline{\Map}^G(\ast, X)$ is the Weil restriction $\psi_*X$ along $\psi \colon BG \to *$. Moreover, it holds
		\begin{align*}
			\Stk(T,\psi_*X) \simeq \Stk_{BG}(T \times BG,X) \simeq \Stk^G(T,U) \simeq \Stk(T,U^G)
		\end{align*}
		giving another description of $U^G$ in terms of mapping stacks.
		
		Assume now that $X$ is algebraic and locally of finite presentation. Observe that $\psi \colon BG \to \ast$ is formally proper by \cite[Prop.~4.3.4]{HalpernleistnerMapping}, that $\psi_*\sO_{BG} \simeq \psi_*\psi^*\CC \simeq \CC$ since $\psi^*$ is the fully faithful functor $(-)^\tv \colon \Mod_{\CC} \to \Mod^G_{\CC}$, and that $\psi_*$ is of Tor-amplitude $\geq 0$ by the previous point. Therefore, because $X$ is locally of finite presentation and all stacks are assumed to have affine diagonal, \cite[Thm.~5.1.1]{HalpernleistnerMapping} implies that all  mapping stacks in the above diagram are algebraic, hence that $U^G$ is as well.
	\end{remark}

	Write $\cl\Stk$ for the category of classical stacks, and $\cl\Stk^G$ for the category of classical stacks with $G$-action. We then have an adjunction
	\begin{align*}
		\iota_\cl : \cl\Stk \rightleftarrows \cl\Stk^G : (-)^{G,\cl}
	\end{align*}
	where $\iota_\cl$ endows a stack with trivial $G$-action, and $(-)^{G,\cl}$ takes the classical fixed points. We thus have a diagram
	\begin{equation}
		\label{Eq:FixVsCl}
		\begin{tikzcd}
			\Stk \arrow[r, "\iota"] \arrow[d, swap, "(-)_\cl"] & \Stk^G \arrow[r, "(-)^G"] \arrow[d, "(-)_{\cl,G}"] & \Stk \arrow[d, "(-)_\cl"] \\
			\cl\Stk \arrow[r, swap,  "\iota_\cl"] & \clStk^G \arrow[r, swap,  "(-)^{G,\cl}"] & \clStk
		\end{tikzcd}
	\end{equation}
	where $(-)_{\cl,G}$ is taking the underlying classical $G$-stack. We will often write $(-)_{\cl, G}$ resp.~$(-)^{G,\cl}$ simply as $(-)_\cl$ resp.~as $(-)^G$, which is justified by the first point of the following result.
	
	\begin{proposition}\mbox{}
		\label{FixPtProps}
		\begin{enumerate}
			\item The diagram (\ref{Eq:FixVsCl}) is commutative. 
			\item Suppose that $X \in \Stk^G$ is a derived scheme such that $X_\cl$ is separated. Then $X^G$ is a derived scheme, and the unit $\iota X^G \to X$ of the adjunction $\iota \dashv (-)^G$ is a closed immersion. In particular, for $B \in \Alg^G$, it holds that $(\Spec^G B)^G$ is affine.
		\end{enumerate}
	\end{proposition}
	
	\begin{proof}
		\noindent\textit{(i)}
		Clearly, it holds $\iota_\cl \circ (-)_\cl \simeq (-)_{\cl,G} \circ \iota$. Let now $j\colon\cl\Stk \to \Stk$ and $j^G\colon\cl\Stk^G \to \Stk^G$ be the inclusions. Then also $\iota \circ j = j^G \circ \iota_\cl$. By composing adjunctions, we have
		\begin{align*}
			\iota \circ j \dashv (-)_\cl \circ (-)^G && \& && j^G \circ \iota_\cl \dashv (-)^{G,\cl} \circ (-)_{\cl, G}
		\end{align*}
		and therefore $(-)_\cl \circ (-)^G \simeq (-)^{G,\cl} \circ (-)_{\cl,G}$.
		
		\medskip
		\noindent\textit{(ii)}	
		Let $X \in \Stk^G$ be a derived scheme such that $X_\cl$ is separated. By \cite[Prop.~A.8.10]{PseudoRed} it holds that $X_\cl^G$ is a closed subscheme of $X_\cl$, hence that $X^G$ is a closed subscheme of $X$.
	\end{proof}
	
	We can now give an explicit description of $\sO_{(\Spec B)^G}$ for $B \in \Poly^G$. By Proposition \ref{FixPtProps}, the adjunction $\iota \dashv (-)^G$ restricts to affine objects. On the algebra side, this gives us an adjunction
	\begin{align*}
		(-)_G : \Alg^G \leftrightarrows \Alg : (-)^\tv
	\end{align*}
	where $(-)^\tv$ is the functor which endows an algebra with trivial $G$-action. 
	
	\begin{definition}
		Let $A \in \Alg^G$ and $(M \in \Mod^G_A)_{\geq 0}$. Define 
		\begin{align*}
			A[M] \coloneqq \LSym^G_A(M)
		\end{align*} 
		
		Let now $B \in \Alg^G_A$ and $\sigma\colon M \to B$ a map of $(G,A)$-modules be given. Define $B/(\sigma)$ as the pushout
		\begin{center}
			\begin{tikzcd}
				A[M] \arrow[d, "\sigma"] \arrow[r, "p"] & A \arrow[d] \\
				B \arrow[r] & B/(\sigma)
			\end{tikzcd}
		\end{center}
		where $p$ is induced by the map $M \to 0$. 
	\end{definition}
	
	Suppose that $M\in \Rep^G$, say $\dim M =d$. Then $\sigma\colon \CC[M[k]] \to B$ induces elements $v_1,\dots,v_d \in \pi_{k}(B)$, and the underlying algebra of $B/(\sigma)$ is equivalent to the derived quotient $B/(v_1,\dots,v_d)$. This follows because forgetting the action is a left adjoint, see Notation \ref{Not:Res-cInd}.
	
	\begin{proposition} \label{Prop:3.20} 
		For $M \in \Mod^G_{\geq 0}$ it holds $\CC[M]_G \simeq \CC[M^\fix]$.
	\end{proposition}
	
	\begin{proof}
		Let $R \in \Alg$. By Corollary \ref{Cor:fix_mv_cofiber}, we have that $(-)^\fix$ is a left adjoint of $(-)^\tv$. We thus have
		\begin{align*}
			\Alg(\CC[M]_G,R) \simeq \Alg^G(\CC[M],R^\tv) \simeq  \Mod^G(M, R^\tv) \simeq \Mod(M^\fix,R)
		\end{align*}
		from which the claim follows.
	\end{proof}

	The same methods also imply the following statement about relative fixed loci.

	\begin{proposition} \label{Prop:Rel fixed locus}
		Let $G \to G'$ be a surjective homomorphism of reductive groups with kernel $H$ and $X \in \Stk^G$. Then $H$ is reductive and the Weil restriction of $[X/G]$ along the morphism $BG \to BG'$ is given by the $H$-fixed locus $[X^H / G']$.
	\end{proposition}

	\subsection{The cotangent complex of the fixed locus}
    Let $X$ be a derived algebraic stack with a $G$-action, and write $\varphi : X \to [X/G]$ for the projection map. By Proposition \ref{Prop:OXG-modules}, the pullback $\varphi^*$ is the forgetful functor $\QCoh^G(X) \to \QCoh(X)$. Since $\BL_X \simeq \varphi^*\BL_{[X/G]/BG}$, it follows that $\BL_X$ carries a canonical $G$-action, with which we tacitly endow it throughout.

    \begin{example}
		\label{Ex:LXcanGact}
	Recall that the Lie algebra $\fg$ is the tangent bundle at the identity $e \colon * \to G$, hence $\fg^\vee \simeq e^* \BL_G \simeq \BL_{*/BG}$. We thus have an exact sequence
 \begin{align*}
     \varphi^* \BL_{[X/G]} \to \BL_X \to \fg^\vee \otimes \sO_X
 \end{align*}
 In particular, when $\BL_X$ is connective, then we can think of $\varphi^*\BL_{[X/G]}$ pictorially as the chain complex $\BL_X \to \fg^\vee \otimes \sO_X$.
	\end{example}
 
	\begin{proposition}
		\label{Prop:CotXG}
		Let $X$ be a derived algebraic $G$-stack, write $i\colon X^G \to X$ for the canonical map. Then the exact sequence
		\begin{align*}
			i^* \BL_X \to \BL_{X^G} \to \BL_{X^G/X}
		\end{align*}
		is equivalent to the exact sequence
		\begin{align*}
			i^* \BL_X \to (i^* \BL_X)^{\fix} \to (i^* \BL_X)^{\mv}[1]
		\end{align*}
		induced by the splitting $i^*\BL_X \simeq (i^* \BL_X)^{\fix} \oplus (i^* \BL_X)^{\mv}$ from Corollary \ref{Cor:fix_mv_cofiber}.
	\end{proposition}
	\begin{proof}
		By construction, the $G$-action on $X^G$ is trivial. Let $j \colon [X^G/G] \to [X/G]$ and $p \colon [X^G/G]  \to X^G$ be the canonical maps. By Proposition \ref{Prop:OXG-modules}, the functor $j^*$ is equivalent to the $G$-equivariant pullback $(i^G)^*$, and the functor $p_*$ is equivalent to $(-)^\fix$. It follows that $p_*j^*\BL_{[X/G]/BG} \simeq (i^* \BL_X)^{\fix}$. 
		
		Recall from Remark \ref{Rem:HLP} that the Weil restriction $\psi_*[X/G]$ along $\psi \colon BG \to *$ is $X^G$. Then the pullback of $\psi$ is the map $p$, and $j$ is induced by the counit of the adjunction $\psi^* \dashv \psi_*$. Moreover, by Proposition \ref{Prop:OXG-modules}, it holds that $p_*$ is also a left adjoint of $p^*$. Therefore, by the formula for the cotangent complex of Weil restrictions given in \cite{Weil}, it also holds  $p_*j^*\BL_{[X/G]/BG} \simeq \BL_{X^G}$. The claim follows.
	\end{proof}

	\subsection{Base-change of fixed loci}
	
	We have the following \'{e}tale base change lemma. Let $G$ be a reductive algebraic group as usual.
	
	\begin{lemma} \label{Lem:etale base change of fixed locus}
		Let $U \to V$ be a $G$-equivariant morphism of derived schemes, write $f \colon [U/G] \to [V/G]$ for the induced map,  and consider the natural commutative square
		\begin{center}
			\begin{tikzcd}
				{[U^{G^0}/G]} \arrow[d] \arrow[r] & {[V^{G^0}/G]} \arrow[d] \\
				{[U/G]} \arrow[r, "f"] & {[V/G]},
			\end{tikzcd}
		\end{center}
		where $G^0 \subseteq G$ denotes the identity component of $G$. If $f$ is \'{e}tale and $f_\cl$ is separated, then the square is Cartesian.
	\end{lemma}
	
	\begin{proof}
		Write $X = [U/G]$ and $Y = [V/G]$. By naturality of $(-)^{G^0}$, if $V^{G^0} = \emptyset$, then $U^{G^0} = \emptyset$, and there is nothing to prove. So we assume that $V^{G^0} \not= \emptyset$. Now $U^{G^0} \not = \emptyset$ if and only if $X$ and $Y$ have the same maximal stabilizer dimension. Hence, this is true for the underlying classical truncations, since we have a natural isomorphism $[U_\cl^{G^0} /G] = (X_\cl)^{\maxlocus} \to (Y_\cl)^{\maxlocus} \times_{Y_\cl} X_\cl = [V_\cl^{G^0} / G] \times_{Y_\cl} X_\cl$ by the discussion following the statement of Proposition~C.5 in \cite{EdidinRydh}.	
		
		By $G$-equivariance and \'{e}taleness of the morphism $X \to Y$, we have a commutative diagram
		\begin{align*}
			\xymatrix@C=0.01em{
				\BL_Y |_{X^{G^0}} \ar[d] & \simeq & (\BL_Y |_{Y^{G^0}})^{\fix}|_{X^{G^0}} \ar[d] & \oplus & (\BL_Y|_{Y^{G^0}})^{\mv}|_{X^{G^0}} \ar[d] \\
				\BL_X |_{X^{G^0}} & \simeq & (\BL_X |_{X^{G^0}})^{\fix} & \oplus & (\BL_X|_{X^{G^0}})^{\mv}
			}
		\end{align*}
		where the vertical arrows are equivalences and the splittings are with respect to the $G^0$-action.
		
		But $\BL_{X^{G^0}} \simeq (\BL_X|_{X^{G^0}})^{\fix}$ and $\BL_{Y^{G^0}} \simeq (\BL_Y|_{Y^{G^0}})^{\fix}$ and it thus follows that $\BL_{X^{G^0} / Y^{G^0}} \simeq 0$ and $X^{G^0} \to Y^{G^0}$ is \'{e}tale.
		
		Now $X^{G^0} \to Y^{G^0}$ factors through the \'{e}tale morphism $X \times_Y Y^{G^0} \to Y^{G^0}$, so the map $X^{G^0} \to X \times_Y Y^{G^0}$ must also be \'{e}tale and hence an isomorphism, as it induces an isomorphism on classical truncations.
	\end{proof}
	
	\subsection{Equivariant standard forms} Our discussion so far allows us to give natural equivariant extensions of standard forms and their existence. 
	
	To this end, let $\modd^G$ be the 1-category of presheaves on $\Rep^G$ that send finite coproducts in $\Rep^G$ to products of sets. Then $\modd^G$ is the category of discrete $G$-modules $(\Mod^G)^\heartsuit$, and $\Mod^G$ is the unbounded derived category of $\modd^G$. A \emph{$G$-cdga} is a commutative algebra object in the category $\Ch_{\geq 0}(\modd^G)$ of chain complexes in homological degree $\geq 0$, with the natural symmetric monoidal structure. Similarly, we define a \emph{graded-commutative $G$-ring} as a graded-commutative algebra object in $\modd^G$.
	
	\begin{definition}
		Let an algebra $A \in \Alg^G$ and $x \in \Spec A$  fixed by $G$ be given. A $G$-cdga model $R$ of $A$ is said to be \textit{in standard form} if $R_0$ is smooth with $\Omega_{R_0}$ free, and the underlying graded-commutative $G$-ring of $R$ is freely generated over $R_0$ on a finite number of generators. If this holds, then we say that $R$ is \emph{minimal at $x$} if the underlying cdga (which is non-equivariantly in standard form) is minimal at $x$.
	\end{definition}
	
	Remark \ref{Rem:standard form cell attachment} also holds in the equivariant case. That is, if a $G$-cdga $R$ is in standard form, then it can be obtained from $R_0$ by equivariant cell-attachements $R(0) \to R(1) \to \dots \to R(n)$ such that, for all $k$,  the underlying graded-commutative $G$-ring of $R(k)$ is the subring of the underlying graded-commutative $G$-ring of $R$ generated by elements in homogeneous degree $\leq k$, and $R(k+1) = R(k) /(\sigma)$ for a $G$-equivariant map $\sigma : M_{k+1}[k] \to R(k)$, where $M_{k+1}$ is a free $G$-module over $R(k)$.

	As in the non-equivariant case, we say that $B \in \Alg^G_A$ (or the corresponding map on spectra) is \textit{locally of finite $G$-presentation} if $B$ is a compact object in $\Alg^G_A$. Likewise, we call $B$ \emph{finitely $G$-presented} if it is contained in the smallest subcategory of $\Alg^G_A$ that contains $\Poly^G_A$ and is closed under finite colimits. We note that any classical, finitely presented, affine scheme $T \in \Stk^G$ admits a $G$-equivariant closed embedding into a smooth classical scheme $\Spec^G(\CC[M])$ for some $M \in \Rep^G$. Moreover, $B$ is finitely $G$-presented if and only if $B$ is finitely presented and admits a $G$-action (one can see this by the argument sketched in the proof of the following lemma). Then $\Spec^G B \to \Spec^G A$ is locally of finite $G$-presentation if and only if $\Spec^G B$ is Zariski locally in $\Stk^G$ finitely $G$-presented over $\Spec^G A$. We then have the following.
	
	\begin{lemma}\label{Lem:Equiv alg are standard form}
		Let an algebra $A \in \Alg^G$ and $x \in \Spec A$  fixed by $G$ be given. Suppose that $A$ is locally of finite $G$-presentation over $\CC$. Then, up to equivariant Zariski localizing around $x$, there is a $G$-cdga model $R$ of $A$ in standard form. We can arrange $R$ to be such that $R_0 = \CC[M]$ for some $M \in \Rep^G$. Alternatively, we can arrange $R$ to be minimal at $x$.
		
		In addition, in both cases, up to further equivariant Zariski shrinking, the underlying $R(i)$-modules $M_{i+1}$ may be taken, for all $i\geq 0$, in the form $M_{i+1} = R(i) \otimes_{\CC} W_{i+1}$ for $G$-representations $W_{i+1} \in \Rep^G$.
	\end{lemma}

	\begin{proof}
		The same argument that proves Lemma~\ref{Lem:Alg are standard form} applies equivariantly to establish everything, but the last assertion. 
		
		We briefly sketch how to equivariantly adapt the steps of the proof of \cite[Theorem~4.1]{JoyceSch}. Write $X = \Spec^G A$ for brevity. Since $A$ is locally of finite $G$-presentation over $\CC$, we may assume, up to $G$-equivariant Zariski open localizing around $x$, that $A$ is finitely $G$-presented. Following \cite[Theorem~4.1]{JoyceSch}, we may now either choose a $G$-equivariant closed embedding $X_\cl \to \Spec^G(\CC[M])$ for some $M \in \Rep^G$ or, using the reductivity of $G$, after further equivariant Zariski open localizing around $x \in X_\cl$, a $G$-equivariant closed embedding $X_\cl \to \Spec^G (R_0)$ into a smooth scheme, that is minimal at $x$ for some finitely presented $R_0 \in \Alg^G$. 
		
		We now continue with the minimal case, as the case with $R_0 = \CC[M]$ proceeds similarly.
		We observe that since $[\Spec^G (R_0) / G]$ is smooth over $BG$, the closed embedding $[X_\cl / G] \to [\Spec^G (R_0) / G]$ can be lifted over $BG$ to a closed embedding $[X/G] \to [\Spec^G (R_0) / G]$, which by Proposition~\ref{Prop:StkG_is_StkBG} amounts to a $G$-equivariant map $R_0 \to A$, lifting $R_0 \to \pi_0(A)$. Let $F$ be the fiber of $R_0 \to A$. Then, up to possible further equivariant Zariski open shrinking around $x$, there exists a free $G$-module $M_1$ over $R_0$ that surjects onto $\pi_0(F)$ and whose fiber at $x$ is isomorphic to $\pi_0(F)|_x$. For this to be possible, we use the reductivity of $G$ and that $x$ is fixed by $G$. The morphism $M_1 \to \pi_0(F) \to R_0$ determines a $G$-cdga $R(1)$, together with a $G$-equivariant morphism $R(1)\to A$ that is an isomorphism on $\pi_0$. This showcases the first step of the construction. One argues similarly to construct the remaining $R(i)$, by following the proof of \cite[Theorem~4.1]{JoyceSch}.

		To establish the last assertion, it clearly suffices to prove that for any $G$-equivariant vector bundle $\wW$ on a derived affine $G$-scheme $U$ there exists a $G$-invariant Zariski open neighbourhood $x \in U_x \subseteq U$ and a $G$-equivariant isomorphism $\wW|_{U_x} \cong W \otimes_{\CC} \oO_U$, where $W$ is a $G$-representation.
		
		Since $x$ is fixed by $G$, $G$ acts on the fiber $W \coloneqq \wW|_x$ and we have a $G$-equivariant surjection $\wW = \Gamma(U,\wW) \to W$. Since $G$ is reductive, we may split this surjection and consider $W \subseteq \Gamma(U,\wW)$. This induces a $G$-equivariant morphism $W \otimes_{\CC} \oO_U \to \wW$. For $U$ classical, the locus where this map is an isomorphism is clearly a $G$-invariant open subscheme of $U$ which contains $x$, and the general case reduces to the classical case, since $\wW$ is a vector bundle. This concludes the proof.
	\end{proof}
	
	Now, our prior discussion of fixed loci immediately implies an explicit description for $(\Spec B)^G$ when $B$ has a model in standard form, given in the next propositions.
	
	For $R$ a model in standard form for $B \in \Alg$, we write $B(i)$ for the algebra in $\Alg$ corresponding to the cdga $R(i)$. 
	
	\begin{proposition} \label{Prop:3.27}
		Suppose that $B \in \Alg^G$ has a model $R$ in standard form such that $M_{i+1} = R(i) \otimes_{\CC} W_{i+1}$ for some $G$-representations $W_{i+1}$. Then $B_G$ also has a model $R_G$ in standard form with $(R_G)_0 = \sO_{(\Spec R_0)^G}$ and generated by the modules $N_{i+1}$ defined inductively by $N_{i+1} = R(i)_G \otimes_{\CC} W_{i+1}^{\fix}$, where $R(i)_G$ is the model in standard form for $B(i)_G$ generated by $N_j$ for $j \leq i$.
		
		In addition, for a point $x \in \Spec B$ fixed by $G$, if $R$ is minimal at $x$, we can arrange $R_G$ to be minimal at $x$ as well.
	\end{proposition}
	
	\begin{proof}
		Recall that $R_0$ is smooth and classical by definition. We first show that $(\Spec R_0)^G$ is smooth and classical and coincides with the classical fixed locus. To this end, write $X = \Spec R_0$. Then, by \cite[Lemma~5.1]{Drezet}, for any $G$-fixed point $x$, there is a Zariski open $G$-invariant neighbourhood $U_x \subseteq X$ together with an \'{e}tale $G$-equivariant morphism $U_x \to T_x X$, where $T_x X$ denotes the tangent space of $X$ at $x$. Let $U$ be the disjoint union of the $U_x$ as $x$ ranges over the fixed points of $X$. We obtain a $G$-equivariant \'{e}tale morphism $f \colon U \to X$. Since $f = f_\cl$ is separated, by Lemma~\ref{Lem:etale base change of fixed locus}, $U^G = X^G \times_{X} U$. Applying the same lemma to each morphism $U_x \to T_x X$, which is \'{e}tale and separated, we get $(U_x)^G = U_x \times_X (T_x X)^G$. By Proposition~\ref{Prop:3.20}, since $T_x X = \Spec ( \CC[M] )$ for some $M \in \Rep^G$, it follows that $(T_x X)^G$ is smooth and classical, and hence so is $(U_x)^G$, and therefore $U^G$ as well. Then, $U^G = X^G \times_{X} U$, and the fact that $U^G \to X^G$ is surjective on points by construction, implies that $X^G$ must be smooth and classical. In particular, it must coincide with the classical fixed locus in this case.
		
		For $i\geq 0$, we assume by induction that the cdga $R(i)_G$, generated by $N_j$ for $j \leq i$, is a model for $B(i)_G$, and put $N_{i+1} \coloneqq R(i)_G \otimes_{\CC} W_{i+1}^{\fix}$. Then we have a (strictly) commutative diagram of cdgas
		\begin{equation}
			\label{Eq:Ri+1}
			\begin{tikzcd}
				\CC \arrow[d] \arrow[r] & \CC[W_{i+1}[i]] \arrow[d, "\alpha"] \arrow[rr, "{W_{i+1}[i] \mapsto 0}"] && \CC \arrow[d] \\
				R(i) \arrow[r] & \LSym_{R(i)}(M_{i+1}[i]) \arrow[rr, "{M_{i+1}[i] \mapsto 0}"] \arrow[d] && R(i) \arrow[d]\\
				& R(i) \arrow[rr] && R(i+1)
			\end{tikzcd}
		\end{equation}
		where the square on the left is the natural one coming from the fact that $M_{i+1} = R(i) \otimes W_{i+1}$, and the bottom square comes from the definition of $R(i+1)$. One shows that all squares are homotopy pushouts, using that $M_{i+1}$ is free and \cite[\S 2.2]{JoyceSch}. Define $R(i+1)_G$ as the homotopy pushout of the map $\CC[W_{i+1}^\fix[i]] \to R(i)_G$ (induced by $\alpha$) along the map $\CC[W_{i+1}^\fix[i]] \to \CC$ induced by the zero map. Then, since $(-)_G$ commutes with pushouts, we see that $R(i+1)_G$ is a model for $B(i+1)_G$, generated by $N_j$ for $j \leq i+1$. This concludes the first claim.
		
		Let a point $x \in \Spec B$ be fixed by $G$ such that $R$ is minimal at $x$. By \cite[Prop.~2.12]{JoyceSch}, the cotangent complex for $R$  restricted to $x$  has underlying chain complex
		\begin{align*}
			\cdots \to W_{i+1}|_x \to W_i|_x \to W_{i-1}|_x \to \cdots
		\end{align*}
		and the cotangent complex for $R_G$ restricted to $x$ has underlying chain complex
		\begin{align*}
			\cdots \to W_{i+1}^\fix|_x \to W_i^\fix|_x \to W_{i-1}^\fix|_x \to \cdots
		\end{align*}
		The claim thus holds by definition of minimality.
	\end{proof}
	
        \subsection{Equivariant blow-ups}
        With similar arguments as in the non-equivariant case, we can apply the machinery of \cite{HekkingGraded, Weil} to the $G$-equivariant setting to produce $G$-equivariant Rees algebras and blow-ups. This is carried out in \cite{GDB} in much more generality. We summarise the results here. 
	
        For $Z \to X$ a $G$-equivariant closed immersion, we have a \emph{$G$-equivariant extended Rees algebra} $\sR_{Z/X}^{G,\extd}$ with a similar universal property as $\sR_{Z/X}^{\extd}$, but now against all $G$-equivariant $T \to X$. We define the \emph{$G$-equivariant blow-up} as
        \[ \Bl^G_ZX \coloneqq [ ((\Spec^G \sR_{Z/X}^{G}) \setminus V ((\sR_{Z/X}^{G})_+) ] \]
        where $\sR_{Z/X}^{G} \coloneqq (\sR_{Z/X}^{G,\extd})_{\geq 0}$.  The following can be shown by comparing universal properties.
	
	\begin{proposition}
		\label{Prop:G_equiv_blow-up}
		It holds that $[\Bl^G_{Z/X}/G] \simeq \Bl_{[Z/G]/[X/G]}$.
	\end{proposition}

        One can show the next result along similar lines as done in \cite{HekkingGraded} or as in \cite{GDB}.
	\begin{proposition}
		\label{Prop:equivariant_quotient_formula}
		Take $A \in \Alg^G$ and let $\sigma\colon M \to A$ be a map of connective $(G,A)$-modules. Put $B \coloneqq A/(\sigma)$. Then we can compute the $G$-equivariant extended Rees algebra of $B$ over $A$ as
		\begin{align*}
			R^{G,\ext}_{B/A} \simeq \frac{A[t^{-1},M]}{(t^{-1}M-\sigma)}
		\end{align*}
		where the quotient is with respect to the map $M \to A[t^{-1},M]$ which is induced by substracting $\sigma\colon M \to A \to A[t^{-1},M]$ from the map 
		\begin{align*}
			\LSym^G_{A[t^{-1}]}(\times t^{-1}) \colon \LSym^G_{A[t^{-1}]}(M[t^{-1}]) \to \LSym^G_{A[t^{-1}]}(M[t^{-1}])
		\end{align*}
	\end{proposition}

	\begin{example}
		\label{Ex:equiv_Rees}
		Suppose that $M \in (\Mod_A^G)_{\geq 0}$ is of the form $M = M_1[k_1] \oplus \dots \oplus M_n[k_n]$, where each $M_i \in \Rep_A^G$ is of the form $M_i = A \otimes V_i$, for $V_i \in \Rep^G$ of dimension $d_i$, and $k_1,\dots,k_n \in \NN$. For each $i$, let $\underline{v}_i= (v_{i1},\dots,v_{id_i})$ be a basis for the underlying $A$-module of $M_i$. 
		
		Let a map $\sigma\colon M \to A$ be given, and let $B$ be the quotient $A/(\sigma)$. Then the restrictions $M_i[k_i] \to A$ of $\sigma$ induce elements $\underline{a}_i = (a_{i1},\dots,a_{id_i}) \in \pi_{k_i}(A)$. From Proposition \ref{Prop:equivariant_quotient_formula}, it follows that the underlying $A[t^{-1}]$-algebra of $R^{G,\ext}_{B/A}$ is 
		\begin{align*}
			\frac{A[t^{-1},\underline{x}_1,\dots,\underline{x}_n]}{(\underline{x}_1t^{-1}-\underline{a}_1,\dots,\underline{x}_nt^{-1}-\underline{a}_n)}
		\end{align*}
		where each $\underline{x}_i$ is a sequence of free variables $x_{i1},\dots,x_{id_i}$ in homogeneous degree 1 and homological degree $k_i$.
	\end{example}

	\begin{example}
		\label{Ex:Rees_BG}
		Let $A \in \Alg^G$ be given. Take a connective $(G,A)$-module $M$ of the form $M = M_1[k_1] \oplus \dots \oplus M_n[k_n]$, with each $M_i \in \Rep^G_A$, say $M_i = A \otimes V_i$ for some $V_i \in \Rep^G$. Let $\sigma\colon M \to A$ be a given map of $(G,A)$-modules, and put $B \coloneqq A/(\sigma)$. We will give a formula for $R^{G,\ext}_{B_G/B}$. 
		
		Observe that
		\begin{align*}
			A[M]_G \simeq (\LSym_A^G(A \otimes V))_G \simeq (\LSym^G(V) \otimes A)_G \simeq A_G[M^{\mv}]
		\end{align*}
		It follows that $B_G$ fits into the following pushout square
		\begin{center}
			\begin{tikzcd}
				A_G[M^{\mv}] \arrow[r, "s"] \arrow[d, "z"] & A_G \arrow[d] \\
				A_G \arrow[r] & B_G
			\end{tikzcd}
		\end{center}
		where $z$ is induced by the zero map and $s$ by the restriction of $\sigma$ to $M^{\mv}$. It follows that $B^{\mv}$ is the cofibre of $M^{\mv} \to A^{\mv}$.
		
		Write $M$ as $M' \oplus M''$, where $M'$ is the direct sum of those $M_i[k_i] \subset M$ for which the restriction $M' \to A$ is nonzero. Then $M'' \to A$ is zero. It follows that
		\begin{align*}
			B^\mv \simeq N \coloneqq A^{\mv}/M'' \oplus M'[1]
		\end{align*}
		hence, by Example \ref{Ex:equiv_Rees},
		\begin{align*}
			R^{G,\ext}_{B_G/B} \simeq \frac{B[t^{-1},N]}{(t^{-1}N -\tau)}.
		\end{align*}	
	\end{example}

	\section{Derived loci of maximal stabilizer dimension}
	\label{Sec:Xmax}
	Let $X$ be a derived algebraic stack whose classical truncation $X_\cl$ is Noetherian and admits a good moduli space. In \cite{EdidinRydh}, it is shown that there is a canonical closed substack $X_\cl^{\maxlocus}$ of $X_\cl$ which is the locus of points of $X_\cl$ with stabilizer group of maximal dimension. Our objective in this section is to establish one of the central results of this paper, the existence of a derived locus of maximal-dimensional stabilizers $X^{\maxlocus}$. Our definition of $X^{\maxlocus}$ is directly inspired by the definition of $X_\cl^\maxlocus$ in \cite{EdidinRydh}, and gives a derived enhancement. Along the way we will review derived loop stacks and derived bundles twisted by inner automorphisms, and give a derived Luna \'{e}tale slice theorem.
	
	\subsection{Derived loop stacks}
	\label{Par:Loops}
	Let $X \in \Stk$ be a derived stack. Recall that the \textit{derived loop stack $\lL X$} of $X$ is the self-intersection of the diagonal $\Delta_X\colon X \to X \times X$, which is algebraic if $X$ is. Let $k\colon T \to X$ be a morphism of derived stacks. Then the \textit{stabilizer} of $X$ in $T$ is the pullback
	\begin{align*}
		G_k \coloneqq G_T \coloneqq T \times_X \lL X
	\end{align*}
	Since $(\lL X)_\cl$ is the inertia stack of $X_\cl$, it holds that $(G_T)_\cl$ is the stabilizer $I_{X_\cl}$ of $X_\cl$ at $T_\cl$. We call $G_T$ \textit{connected} if $G_T\to T$ has connected classical fibers. If $\sO_T$ is a field, then $k\colon T \to X$ corresponds to a point $x \in X$, and we call $G_x \coloneqq G_k$ the \textit{stabilizer} (at $x$). Note that $(G_x)_\cl$ is a classical affine group scheme.
	
	\begin{example} \label{Example:loop stack of quotient}
		The basic example is when $X$ is of the form $[U/G]$, where $U$ is a derived scheme, for which it holds
		\begin{align*}
			\lL X \simeq [S_U /G].
		\end{align*}
		Here, $S_U$ is the derived stabilizer group scheme $S_U$, defined as the fibre product $(G \times U) \times_{U \times U} U$ of the map $(\sigma,\pi_2)\colon G \times U \to U \times U$ (where $\sigma$ is the $G$-action) with the diagonal. See \cite[\S 4.4]{ToenDerived}.
	\end{example}
	
	Let $f\colon X \to Y$ be a morphism of derived stacks. Then we have a natural morphism $\lL f\colon\lL X \to \lL Y \times_Y X$ of derived stacks over $X$. We call $f$ \textit{monodromy-free} if $\lL f$ is an equivalence. Likewise, we call $f$ \textit{stabilizer-preserving} if $(\lL f)_\cl$ is an equivalence. If the latter is the case, then for any $k\colon T \to X$ with $T$ classical, the morphism $(\lL f)_\cl$ induces an equivalence $(G_k)_\cl \simeq (G_{fk})_\cl$.
	
	We call a morphism of derived algebraic stacks $f\colon X \to Y$ \textit{separated} (\textit{proper}) if $f_\cl$ is separated (proper), as in \cite[\href{https://stacks.math.columbia.edu/tag/04YW}{Tag 04YW}]{stacks-project}. If $f\colon  X \to Y$ is representable, then $f$ is separated if and only if $\Delta_f$ is a closed immersion \cite[\href{https://stacks.math.columbia.edu/tag/04YS}{Tag 04YS}]{stacks-project}. Hence, since in the general case the diagonal is representable, one can also define $f$ to be separated by asking $\Delta_f$ to be of finite type, universally closed, and separated. The basic examples are: affine morphisms are separated \cite[\href{https://stacks.math.columbia.edu/tag/06TZ}{Tag 06TZ}]{stacks-project}, and closed immersions are proper \cite[\href{https://stacks.math.columbia.edu/tag/0CL8}{Tag 0CL8}]{stacks-project}.
	
	\begin{example}
		\label{Ex:GMS_separated}
		Let $X$ be a derived algebraic stack. Since we assume that all stacks have affine diagonal, in particular $X$ has separated diagonal. Therefore, for $T$ a separated derived algebraic stack, any morphism $T \to X$ is separated by \cite[\href{https://stacks.math.columbia.edu/tag/050M}{Tag 050M}]{stacks-project}.
		
		Suppose that we have a Cartesian diagram
		\begin{center}
			\begin{tikzcd}
				X \arrow[r, "f"] \arrow[d] & Y \arrow[d] \\
				M \arrow[r] & N
			\end{tikzcd}
		\end{center}
		of derived algebraic stacks, where the vertical arrows are good moduli spaces. Then if $M$ is separated, so is $f$. 
	\end{example}
	
	\begin{remark}
		\label{Rem:lLf_pullback_diagonal}
		Let $f\colon X \to Y$ be a morphism of derived algebraic stacks. Then we have a Cartesian diagram
		\begin{equation}
			\label{Eq:loop diagram}
			\begin{tikzcd}
				\lL X \arrow[d, "\lL f"] \arrow[r] & X \arrow[d, "\Delta_f"] \\
				f^* \lL Y \arrow[d]\arrow[r] & X \times_Y X \arrow[d] \arrow[r] & Y \arrow[d] \\
				X \arrow[r, "\Delta_X"] & X \times X \arrow[r] & Y \times Y
			\end{tikzcd}
		\end{equation}
		which shows that $\lL f$ is a pullback of $\Delta_f$. In particular, if $f$ is representable and separated, then $\lL f$ is a closed immersion.
	\end{remark}

	\begin{proposition}
		\label{Rem:etale_stabilizer_preserving}
		Let $f\colon X \to Y$ be an \'{e}tale morphism of derived algebraic stacks such that $f_\cl$ is stabilizer-preserving. Then $f$ is monodromy-free.
	\end{proposition}
	
	\begin{proof}
		Since $f$ is \'{e}tale, the Cartesian diagram
		\begin{align*}
			\xymatrix{
				X \times_Y X \ar[d] \ar[r] & X \ar[d]^-{f} \\
				X \ar[r]^-{f} & Y
			}
		\end{align*}
		implies that $X \times_Y X \to X$ is \'{e}tale. Since the composition $X \xrightarrow{\Delta_f} X \times_Y X \to X$ is the identity, we conclude that $\Delta_f$ is \'{e}tale. By diagram~\eqref{Eq:loop diagram}, since $\lL f$ is a pullback of $\Delta_f$, it must also be \'{e}tale. But $(\lL f)_\cl = \lL (f_\cl)$ is an isomorphism, since $f_\cl$ is stabilizer-preserving, so $\lL f$ is an equivalence.
	\end{proof}

	\begin{lemma}
		\label{Lem:Llf_clopen_repsepet}
		Suppose that $f\colon X \to Y$ is representable, separated and \'{e}tale. Then $\Delta_f$, hence $\lL f$, is a closed and open immersion.
	\end{lemma}
	
	\begin{proof}
		By Remark \ref{Rem:lLf_pullback_diagonal}, it suffices to show that an \'{e}tale closed immersion $h\colon S \to T$ is an open immersion. This reduces to the affine case since $h$ is affine, and then to the classical case since $h$ is \'{e}tale. This last case is known, see \cite[\href{https://stacks.math.columbia.edu/tag/025G}{Tag 025G}]{stacks-project}.
	\end{proof}
	
	We recall some facts about derived loop stacks from \cite{BenZviLoop}. Let $S^1$ be the stackification of the presheaf that sends any $T \in \Aff$ to the circle. Then it holds
	\begin{align*}
		S^1 \simeq B\ZZ \simeq * \bigsqcup_{* \sqcup *} *
	\end{align*}
	where $B\ZZ$ is the classifying space of the constant group scheme $\ZZ$. Then for any derived stack $X$ it holds
	\begin{align*}
		\lL X \simeq \underline{\Map}(S^1,X)
	\end{align*}
	For any $f\colon X \to Y$ it follows that $f^* \lL Y$ is the stack over $X$ which sends $T \to X$ to the space of commutative diagram
	\begin{center}
		\begin{tikzcd}
			T \arrow[d] \arrow[r] & X \arrow[d] \\
			S^1 \times T \arrow[r] & Y   
		\end{tikzcd}
	\end{center}
	which we think of as loops on $Y$ based at $f$. The morphism $\lL X \to f^* \lL Y$ then sends $T \times S^1 \to X$ to the loop $T \times S^1 \to X \to Y$, which is based at $f$. In particular, $\lL f$ is a morphism of derived group stacks. This also explains the term ``monodromy-free''. 
	
	\begin{notation}
		For a morphism $f \colon x \to y$ in an $\infty$-category $\CCC$, we write $\CCC_{x / \cdot / y}$ for the double slice-category $(\CCC_{/y})_{f/}$, so that an object of $\CCC_{x/\cdot /y}$ is a factorization $x \to w \to y$ of $f$.
	\end{notation}
	
	\begin{remark}
		\label{Rem:B_loop_adjunction}
		Let $T$ be a derived stack. Throughout, we will write $B_T$  for the functor $B_T\colon  \Grp(\Stk_{/T}) \to \Stk_{T/}$ which sends a derived group stack $G$ over $T$ to the colimit of the bar-construction $B(T,G)$ in $\Stk_T$. In analogy to the loop space-suspension adjunction from topology, there is an adjunction
		\begin{align*}
			B_T :  \Grp(\Stk_T)\rightleftarrows  \Stk_{T/}: \Omega_T
		\end{align*}
		such that $\Omega_T(T \to X) \simeq \lL X\times_X T$, and with $B_T$ fully faithful, which we now explain.
		
		We have a sequence of adjunctions
		\begin{align*}
			\Grp(\Stk_T)\rightleftarrows  \Epi_{T/}(\Stk_T) \rightleftarrows \Stk_{T/\cdot / T} \rightleftarrows \Stk_{T/}
		\end{align*}
		where: the first adjunction on the left is the adjoint equivalence from Proposition \ref{Prop:B_Omega_equiv}; the second adjunction in the middle has left adjoint the inclusion, and as right adjoint the functor which sends $g \colon T \to Y$ over $T$ to the $\infty$-image $T \to \im (g)$; and the third adjunction has left adjoint the forgetful functor, and as right adjoint the functor which sends $T \to X$ to $T \to X \times T$. 
		
		The composition of all of these left adjoints is $B_T$. Write $\Omega_T$ for the right adjoint of $B_T$. The for $f \colon T \to X$ given, it holds that 
		\[ \Omega_T(X) \simeq T \times_{\im (f,\id_T)} T \simeq T \times_{X \times T} T \] 
		where the second equivalence follow from the fact that $\im(f, \id_T) \to X \times T$ is a monomorphism. Then the first claim, $\Omega_T(T \to X) \simeq \lL X\times_X T$, follows from the Cartesian diagram
		\begin{center}
			\begin{tikzcd}
				\lL X \times_X T \arrow[r] \arrow[d] & T \arrow[d] \arrow[r] & X \arrow[d] \\
				T \arrow[r] & X \times T \arrow[r] & X \times X
			\end{tikzcd}
		\end{center}
		and the second claim follows from the familiar fact that $G \simeq T \times_{BG} T$ for any $\infty$-group $G$ over $T$.
	\end{remark}
	
	\subsection{Twisted forms}
	We introduce $\infty$-bundles which are twisted versions of principal $G$-bundles, for a given reductive group $G$. This will be used to describe $[U/G]^{\maxlocus}$ in the cases relevant for the proof of Theorem \ref{thm existence of X max}.
	
	\begin{definition}
		For $f \colon  T \to BG$, say classified by the $G$-bundle $P \to T$, define the \textit{inner form} $\widetilde{G}_f$ of $G$ via the pullback
		\begin{equation}  \label{Eq:def of G_f tilde}
			\begin{tikzcd}
				\widetilde{G}_f \arrow[r] \arrow[d] & \lL (BG) \arrow[d] \\
				T \arrow[r] & BG.
			\end{tikzcd}
		\end{equation}
	\end{definition}
	Observe that $\widetilde{G}_f$ is a pullback of an $\infty$-group, hence is an $\infty$-group itself. In the terminology from  \cite[Prop.~4.10]{NikolausPrincipal} it holds that $\widetilde{G}_f$ is the $G_\ad$-fiber $\infty$-bundle $P \times_G G_\ad$, since $\lL (BG) = [G_\ad/G]$ by definition. Clearly, we have $\tilde{G}_f \simeq \Omega_T(BG)$.

	Recall that the identity component $G^0$ is an open and closed, normal subgroup of $G$, such that $G/G^0$ is finite, hence reductive. Since $G^0 \subset G$ is normal, the adjoint action on $G_\ad$ restricts to the identity component, and we write $G^0_\ad$ for the corresponding $G$-object. Then we define the \emph{twisted form} $\tilde{G}_f^0$ of $G$ via the pullback
	\begin{center}
		\begin{tikzcd}
			\tilde{G}_f^0 \arrow[r] \arrow[d] & {[G^0_\ad/G]} \arrow[d] \\
			T \arrow[r] & BG
		\end{tikzcd}
	\end{center}
	The group structure on $G^0$ induces a group structure on $\tilde{G}_f^0$. The open and closed inclusion $G^0 \to G$ induces a morphism $\tilde{G}_f^0 \to \tilde{G}_f$, which exhibits $\tilde{G}_f^0$ as open and closed subgroup of $\tilde{G}_f$. 
	
	\begin{lemma} 
		\label{Lem:inner_forms}
		Let $G$ be an algebraic group of dimension $d$, and $\widetilde{G}_f$ an inner form of $G$, classified by $f\colon  T \to BG$ which has corresponding $G$-bundle $P$. 
		\begin{enumerate}
			\item $\widetilde{G}_f$ and $\widetilde{G}_f^0$ are derived algebraic group schemes that are smooth of dimension $d$ over $T$, and 
			$\widetilde{G}_f^0 \to T$ has connected fibers. 
			\item The canonical map
			$B_T(\widetilde{G}_f^0) \to T \times_{B(G/G^0)} BG$ is an equivalence.
		\end{enumerate}
	\end{lemma}
	
	\begin{proof}
		\noindent\textit{(i)} Since $\lL (BG) \simeq [G_\ad/G]$, it holds that $\lL (BG)$ is fiberwise an algebraic group scheme, smooth and of dimension $d$ over $BG$, so by diagram~\ref{Eq:def of G_f tilde} the same is true for $\widetilde{G}_f \to T$. A similar argument goes for $\tilde{G}_f^0$.
		
		Pulling back $\tilde{G}_f^0$ along the $G$-bundle $P \to T$, we obtain $q \colon P \times G^0_\ad \to P$. Since $G^0$ is connected, $q$ has connected fibers.
		
		\medskip
		\noindent\text{(ii)} 
        Note that the bar-construction $B(*,G^0)$ is $G$-equivariant, where $G$ acts on $B(*,G^0)_n = (G^0)^{\times n}$ diagonally and by conjugation. Since $[-/G]$ commutes with colimits, we thus obtain a $G$-action on $BG^0$, written $BG^0_\ad$, such that $[G^0_\ad/G] \to BG$ factors through $[BG_\ad^0/G] \to BG$.
  
        Consider the diagram
		\begin{center}
			\begin{tikzcd}
				\tilde{G}_f^0 \arrow[d, "\varphi",""{name=D1}] \arrow[r] & {[G^0_\ad/G]} \arrow[d, "\psi",""{name=D2}]  \\
				B_T(\tilde{G}_f^0) \arrow[d,""{name=E1}] \arrow[r] & {[BG^0_\ad / G]}\arrow[r] \arrow[d,""{name=E2}] & BG \arrow[d,""{name=E3}] \\
				T \arrow[r] & BG \arrow[r] & B(G/G^0)
				\arrow[phantom,from=D1,to=D2,"\scriptstyle(\alpha)"]
				\arrow[phantom,from=E1,to=E2,"\scriptstyle(\gamma)"]
				\arrow[phantom,from=E2,to=E3,"\scriptstyle(\delta)"]
			\end{tikzcd}
		\end{center}
		By pulling back along $* \to BG$, we see that $(\delta)$ is Cartesian. Since geometric realizations commute with pullbacks, also $(\alpha)$ is Cartesian. Now because the rectangle $(\alpha)+(\gamma)$ is Cartesian and the morphisms $\varphi, \psi$ are effective epimorphisms, it follows that the square $(\gamma)$ is Cartesian as well, which implies the claim.
	\end{proof}
	
	\subsection{{A derived Luna \'{e}tale slice theorem}}
	We give a derived version of the \'{e}tale local structure of algebraic stacks in terms of quotient stacks.
	
	\begin{notation}
		For $X \in \Stk$, write $\Aff^{\et}_X$ for the $\infty$-category of affine
		\'etale morphisms $T\to X$.
	\end{notation}
	
	Write $\lvert X \rvert$ for the equivalence class of morphisms $x\colon \Spec k \to X$ with $k$ a field, where $x$ is equivalent to $x'\colon \Spec k' \to X$ if there is a common field extension $k,k' \subset K$ such that $x_{\mid \Spec K} \simeq x'_{\mid \Spec K}$. Consider the set of points, written $\lvert X \rvert$, as a topological space by declaring $U_0 \subset \lvert X \rvert$ to be open if there is an open immersion $U \to X$ such such that $\lvert U \rvert \to \lvert X \rvert$ has image $U_0$. Since open immersions are monomorphisms, passing from an open immersion $U \to X$ to the corresponding open subset $\lvert U \rvert \subset \lvert X \rvert$ constitutes an equivalence of partially ordered sets. The induced map $\lvert X_\cl \rvert \to \lvert X \rvert$ is a homeomorphism, as the following result shows.
	
	\begin{lemma}[Topological invariance]
		\label{Lem:topological_invariance}
		Let $X$ be a quasi-compact derived algebraic stack. Then pulling back along $X_\cl \to X$ induces an equivalence $\Aff_X^{\et} \simeq \Aff_{X_\cl}^{\et}$.
	\end{lemma}
	
	\begin{proof}
		The affine case can be found in \cite[Lem.~2.15]{GaitsgoryStudy}.
		
		For the general case, let $U \to X$ be a smooth atlas, where $U$ is a derived scheme. We can assume that $U$ is affine since $X$ is quasi-compact. Then the commutative diagram
		\begin{center}
			\begin{tikzcd}
				U_\cl \arrow[d] \arrow[r] & U \arrow[d] \\
				X_\cl \arrow[r] & X
			\end{tikzcd}
		\end{center}
		is Cartesian. Now the claim follows from the affine case by descent.
	\end{proof}

	Suppose that $X \in \Stk$ is such that $X_\cl$ admits a good moduli space $X_\cl \to Y$. We say that a morphism $f \colon [U/G] \to X$ is \textit{strongly \'{e}tale} if $U \in \Aff^G$, $f$ is \'{e}tale and $f_\cl$ fits in a Cartesian diagram 
	\begin{align*}
		\xymatrix{
			[U_\cl/G] \ar[r]^-{f_\cl} \ar[d] & X_\cl \ar[d] \\
			U_\cl \git G \ar[r] & Y.
		}
	\end{align*}
	
	Observe that $f_\cl$ then has to be stabilizer-preserving and therefore, by Proposition~\ref{Rem:etale_stabilizer_preserving}, also monodromy-free. Moreover, $f$ is representable and separated by Example \ref{Ex:GMS_separated}, since  $U_\cl \git G$ is affine, hence separated. We record these observations in the following proposition. 
	\begin{proposition} \label{Prop:strongly etale is monodr-pres and separated}
		Let $X$ be a derived  algebraic stack such that $X_\cl$ admits a good moduli space. Any strongly \'{e}tale morphism $f \colon [U/G] \to X$ is also representable, separated and monodromy-free.
	\end{proposition}
	
	We now prove a derived Luna \'{e}tale slice result.
	
	\begin{proposition}
		\label{Prop:luna_etale_slic}
		Let $X$ be a derived algebraic stack such that $X_{\cl}$ is Noetherian and admits a good moduli space. Then $X$ is covered by derived stacks of the form $h\colon [U/G] \to X$, where $U \in \Aff^G$ is such that $U^G \not = \emptyset$, $G$ is reductive, and $h$ is affine and strongly \'{e}tale (hence monodromy-free by Remark  \ref{Rem:etale_stabilizer_preserving}).
	\end{proposition}
	
	\begin{proof}
		Since $X_\cl$ is Noetherian, the set of closed points is dense in $\lvert X_\cl \rvert$. By the Luna \'{e}tale slice theorem for classical algebraic stacks from \cite[Thm.~A.1]{EdidinRydh}, we therefore have a cover of $X_\cl$ by  stacks of the form $h_\cl\colon [U_\cl/G] \to X_\cl$, where $U_\cl$ is a classical $G$-scheme such that $U_\cl^G \not=\emptyset$, the map $h_\cl$ is \'{e}tale, affine and stabilizer-preserving, and $G$ is a stabilizer of some closed point in the image of $U_\cl \to X_\cl$.
		
		Since $X_\cl$ is Noetherian, it is quasi-compact. By Lemma \ref{Lem:topological_invariance}, we thus have a Cartesian diagram
		\begin{center}
			\begin{tikzcd}
				{[U_{\cl}/G]} \arrow[d, "h_\cl"] \arrow[r, "f"] & V \arrow[d, "h"] \\
				X_\cl \arrow[r] & X
			\end{tikzcd}
		\end{center}
		where the vertical maps are affine and \'{e}tale, and $f_\cl$ is an isomorphism. By \cite[Lem.~4.2.5]{HL}, the derived stack $V$ is of the form $[U/G]$, for some derived $G$-scheme $U$, and $f$ is induced by a $G$-equivariant map $f'\colon U_\cl \to U$ such that $f'_\cl$ is an isomorphism.
		
		The collection of morphisms $h\colon [U/G] \to X$, indexed by the elements $h_\cl\colon [U_\cl/G] \to X_\cl$ appearing in the cover coming from the Luna \'{e}tale slice theorem, is an \'{e}tale cover, because $\lvert X_\cl \rvert \simeq \lvert X \rvert$.
	\end{proof}

	\subsection{The derived locus $X^{\maxlocus}$ of maximal-dimensional stabilizers}
	
	Throughout this subsection, let $X$ be a derived algebraic stack whose classical truncation $X_\cl$ admits a good moduli space $Y$. We moreover assume that $X_\cl$ is Noetherian. Recall that all stacks have affine diagonal by assumption. In what follows, $G^0 \subseteq G$ denotes the identity component of any (possibly disconnected) algebraic group $G$.
	
	The goal is to show the following.
	
	\begin{theorem} \label{thm existence of X max}
		Let $d$ be the maximal dimension of stabilizers of points of $X$. Then there exists a canonical closed immersion $X^{\maxlocus} \to X$ with the following property: For any affine, strongly \'{e}tale morphism $[U/G] \to X$, where $U \in \Aff^G$ and $G$ is reductive with $\dim G = d$, there exists a Cartesian square
		\begin{align*}
			\xymatrix{
				[U^{G^0} / G] \ar[r] \ar[d] &  \ar[d] [U/G] \ar[d] \\
				X^{\maxlocus} \ar[r] & X.
			}
		\end{align*}
		Then $(-)_\cl$ commutes with $(-)^\maxlocus$, and $X_\cl^\maxlocus$ is equivalent to the classical locus of points of $X_\cl$ with maximal stabilizer dimension defined in \cite[Prop.~C.5]{EdidinRydh}.
	\end{theorem}

	Since $\lvert X_\cl \rvert \simeq \lvert X \rvert$, the map $x \mapsto \dim G_x$ is upper semi-continuous on $\lvert X \rvert$ by the same argument as for the classical case, see \cite{EdidinRydh}. We let $X^{\leq d} \to X$ be the open immersion corresponding to the open subset of $\lvert X \rvert$ consisting of points $x$ with $\dim G_x \leq d$. Observe, if $X$ is classical then so is $X^{\leq d}$, and in general it holds $(X^{\leq d})_\cl \simeq (X_\cl)^{\leq d}$.
	
	\begin{definition}
		For $d\in \NN$, let $X^d$ be the stack over $X$ such that $X^d(T \to X)$ is the space of closed derived subgroup schemes $H^0 \to T \times_X \lL X$ which are smooth over $T$ with connected fibers of dimension $d$. Naturality in $T$ is defined by pullback. If $d$ is the maximal dimension of stabilizers of $X$, we define $X^{\maxlocus} \coloneqq X^d$.
	\end{definition}
	
	\begin{remark}
		\label{Rem:Xmax_FOPonStacks_quotient}
		Let $T \in \Stk_X$ be of the form $[V/G]$ and $X$ of the form $[U/G]$, where $U, V$ are  affine derived $G$-schemes, and $G$ is reductive. Then it holds that $X^d(T \to [U/G])$ is the space of derived closed subgroup stacks $H^0 \to T \times_X \lL X$, such that after pulling back along $q\colon  V \to [V/G]$ it holds that $q^*H^0 $ is a derived closed subgroup scheme smooth over $V$ with connected fibers of dimension $d$. 
	\end{remark}
	
	In order to prove several properties satisfied by $X^{\maxlocus}$, we will need a few auxiliary results.
	
	\begin{proposition}
		\label{Prop:slice_FOP}
		Suppose that $X$ is of the form $[U/G]$ for an affine derived scheme $U$ and a reductive group $G$. Then the derived stack $[U^{G^0}/G]$  has functor of points
		\begin{align*}
			[U^{G^0}/G](f\colon  T \to X) \simeq \Stk_{T/ \cdot /BG}(T \times_{B(G/G^0)} BG,X),
		\end{align*}
	\end{proposition}
	
	\begin{proof}
		Write $\pi\colon  BG \to B(G/G^0)$ and $p\colon X = [U/G] \to BG$ for the projections. Let $f\colon  T \to X$ be given, and let $\pi_!$ be the left adjoint of $\pi^*$, the functor which sends $S \to BG$ to $S \to B(G/G^0)$. Then it holds, using the universal property of the Weil restriction,
		\begin{align*}
			\Stk_{X}(T,\pi^*\pi_*X) & \simeq \Stk_{BG}(T,\pi^*\pi_*X) \times_{\Stk_{BG}(T,X)} \{f\} \\
			& \simeq \Stk_{BG}(\pi^*\pi_!T,X) \times_{\Stk_{BG}(T,X)} \{f\} \\
			& \simeq \Stk_{T/ \cdot /BG}(T \times_{B(G/G^0)} BG,X). & \qedhere
		\end{align*}
	\end{proof}
	
	Let $p\colon X \to BG$ and $f\colon  T \to X$ be given. Then we have a commutative diagram
	\begin{center}
		\begin{tikzcd}
			f^* \lL (X) \arrow[d] \arrow[r] & \lL (BG) \arrow[d] \\
			T \arrow[r, "pf"] & BG
		\end{tikzcd}
	\end{center}
	and thus a morphism $h\colon f^* \lL (X) \to \widetilde{G}_{pf}$ over $T$ by definition of inner forms. In what follows, we will consider $f^* \lL (X)$ as derived stack over $\widetilde{G}_{pf}$ via $h$. 
	
	\begin{proposition}
		\label{Prop:slice_FOP_2}
		Suppose that $X$ is of the form $[U/G]$ for an affine derived scheme $U$ and a reductive group $G$ of dimension $d$, write $p\colon  [U/G] \to BG$ for the projection. Then the derived stack $X^d$  has functor of points
		\begin{align*}
			X^d(f\colon  T \to X) \simeq \Grp(\Stk_{\widetilde{G}_{pf}})(\widetilde{G}^0_{pf}, f^*\lL X).
		\end{align*}
	\end{proposition}
	
	\begin{proof}
		Recall that $\tilde{G}_{pf}^0\to \tilde{G}_{pf}$ is a closed and open subgroup
		which is smooth with connected fibers over $T$. Since $f^*\lL X\to \tilde{G}_{pf}$ is a closed subgroup, it follows that any homomorphism
		$\tilde{G}_{pf}^0\to f^*\lL X$ is a closed immersion. We thus obtain a natural map
		\begin{align}
			\varphi \colon \Grp(\Stk_{\widetilde{G}_{pf}})(\widetilde{G}^0_{pf}, f^*\lL X) \to  X^d(f \colon T \to X)
		\end{align}
		which we claim is an equivalence.
		
		Let $ H^0 \to f^\ast \lL X$ be any closed derived subgroup scheme, smooth over $T$ and with connected fibers of dimension $d$. Consider the diagram
		\begin{align*}
			\xymatrix{
				& & {\widetilde{G}_{pf}^0} \ar[d]^-{\beta} \\
				H^0 \ar[r] \ar@{..>}[urr]^-{\alpha} & f^\ast \lL X \ar[r] & f^\ast p^* \lL (BG) = {\widetilde{G}_{pf}}
			}
		\end{align*}
		Since $\beta$ is a closed and open immersion, the existence of the dotted arrow can be checked on classical truncations, on which this holds by \cite[Appendix~C]{EdidinRydh}, and on which moreover it is an isomorphism [ibid.]. Since $\beta$ is a monomorphism, the arrow $\alpha$ is canonically unique and satisfies that $\alpha_\cl$ is an isomorphism. Since $H^0, \widetilde{G}_{pf}^0$ are flat over $T$, also $\alpha$ is invertible. Clearly, the inverse of $\alpha$ induces a morphism of derived group stacks $\tilde{G}^0_{pf} \to f^* \lL X$ over $\tilde{G}_{pf}$, which shows that $\varphi$ is surjective on $\pi_0$. 
		
		By our argument above, the fiber of $\varphi$ over any point is the space of automorphisms of the derived group stack $\widetilde{G}^0_{pf}$ over $\widetilde{G}_{pf}$, which is contractible, since the morphism $\widetilde{G}^0_{pf} \to \widetilde{G}_{pf}$ is a monomorphism. This concludes the proof.
	\end{proof}

	Theorem~\ref{thm existence of X max} will now be an immediate consequence of the following.
	\begin{proposition} 
		\label{Prop:prop_xmax}
		Let $X,Y$ be a derived algebraic stacks such that $X_\cl,Y_\cl$ are Noetherian and admit good moduli spaces.
		\begin{enumerate}
			\item $\lvert X^d \rvert \to \lvert X \rvert$ is injective, and identifies $\lvert X^d \rvert$ with the subset $\lvert X \rvert^d \subset \lvert X \rvert$ of points with stabilizer dimension $d$.
			\item The construction $(-)^d$ is functorial in representable and separated morphisms.
			\item If $g\colon X \to Y$ is representable, separated and \'{e}tale, then the natural map $X^d \to X \times_Y Y^d$ is an equivalence.
			\item Suppose that $X$ is of the form $[U/G]$ for an affine derived scheme $U$ and a $d$-dimensional reductive group $G$. If $G^0$ acts trivially on $U$, then $[U/G]^d \to [U/G]$ is an equivalence. 
			\item In general, for $X=[U/G]$ as in the previous point but with possibly non-trivial $G^0$-action on $U$, the morphism $[U^{G^0}/G] \to [U/G]$ induces an equivalence $[U^{G^0}/G] \simeq [U/G]^d$. 
			\item The derived stack $X^d$ is algebraic.
			\item The canonical map $(X^d)_\cl\to X_\cl$ is the closed substack
			described in \cite[Appendix~C]{EdidinRydh}.
			\item $X^d \to X$ factors through $X^{\leq d}$ via a (necessarily unique) closed immersion $X^d \to X^{\leq d}$.
		\end{enumerate}
	\end{proposition}
	\begin{proof}
		\noindent\textit{(i)}
		Let $x \colon  \Spec k \to X$ be given. Then 
		\begin{align*}
			(G_x)^0_\cl \subset G_x = \lL (X) \times_X \Spec k
		\end{align*}
		is the unique closed subgroup of $G_x$ of dimension $d' \coloneqq \dim G_x$ which is smooth and connected. From this, (i) follows.
		
		\medskip
		\noindent\textit{(ii)}
		Let $g\colon X \to Y$ be representable and separated. For any $T \to X$, we have a Cartesian square
		\begin{equation}
			\label{Eq:naturality_Xd}
			\begin{tikzcd}
				T \times_X \lL X \arrow[r] \arrow[d] & \lL X \arrow[d, "\lL g"] \\
				T \times_Y \lL Y \arrow[r] & X \times_Y \lL Y,
			\end{tikzcd}
		\end{equation}
		hence $T \times_X \lL X \to T \times_Y \lL Y$ is a closed immersion by Remark \ref{Rem:lLf_pullback_diagonal}. The claim follows.

		\medskip
		\noindent\textit{(iii)} Assume moreover that $g \colon  X \to Y$ is \'{e}tale. Then $\lL g \colon  \lL X \to X \times_Y \lL Y$ is a closed and open immersion by Lemma~\ref{Lem:Llf_clopen_repsepet}. From the diagram (\ref{Eq:naturality_Xd}), it follows that $\varphi \colon T \times_X \lL X \to T \times_Y \lL Y$ is an open and closed immersion, hence a monomorphism, for any $f \colon T \to X$. Let a $T$-point of $Y^d$ be classified by $\alpha \colon H^0_Y  \to T\times_Y \lL Y$, and let $H^0_X \to T \times_X \lL X$ be the pullback of $\alpha$ along $\varphi$. Then the projection $\gamma \colon H^0_X \to H^0_Y$ is an open and closed immersion over $T$, which is an equivalence since $H^0_Y\to T$ has connected fibers. We have thus described a $T$-point of $X^d$ such that composing with $\varphi$ recovers $\alpha$.
		
		\medskip
		\noindent\textit{(iv)}
		Let $X = [U/G]$ with $U$ affine and $\dim G = d$, and suppose that $G^0$ acts trivially on $U$. Let $S_U$ be the derived stabilizer group scheme $(G \times U) \times_{U \times U} U$ as in Example \ref{Example:loop stack of quotient}. Since $G^0$ acts trivially on $U$, we have a natural map $G^0 \times U \to S_U$ over $G \times U$, and thus morphisms
		\begin{align*}
			[G^0 \times U/G] \xrightarrow{\beta} \lL X \xrightarrow{\alpha} [U/G]
		\end{align*}
		over $[U/G]$. Since $* \to BG$ is smooth, so is $U \to [U/G]$. Likewise, $G^0 \times U \to U$ is smooth, as it is the pullback of the smooth map $G^0 \to *$ along $U \to *$. It follows that $\alpha \beta$ is smooth. Clearly, the fibers of $\alpha\beta$ are connected and of dimension $d$. Moreover, since $U$ is affine, the morphism $S_U \to G \times U$ is separated, hence $G^0 \times U \to S_U$ is a closed immersion, since $G^0 \times U \to G \times U$ is.  By Remark \ref{Rem:Xmax_FOPonStacks_quotient}, we therefore have a morphism $\varphi\colon  [U/G] \to [U/G]^d$ over $[U/G]$. To see that it is an equivalence, write $\chi\colon  [U/G]^d \to [U/G]$ for the natural map. By construction, it holds $\chi \varphi \simeq \id_{[U/G]}$. 
		
		To show that also $\varphi \chi \simeq \id_{[U/G]^d}$, let $F^0 \to \chi^*\lL X$ be the derived closed subgroup stack classifying the identity on $[U/G]^d$ as stacks over $[U/G]$. Then $F^0$ is universal in the following sense: for a $T$-point $f \in X(T)$, we have that any $f' \in X^d(T)$, say classified by $H^0 \to f^* \lL X$, trivially satisfies $\chi f' \simeq f$, and hence that $H^0 \simeq {f'}^*F^0$ by definition of $(-)^d$. In particular, for $\psi \colon [U/G]^d \to [U/G]^d$ any map over $[U/G]$,  since $\chi \psi \simeq \chi$, it follows that $\psi$ is classified by $F^0$ as well, hence $\psi \simeq \id_{[U/G]^d}$. Since $\chi \varphi \chi \simeq \chi$, the claim follows by applying this observation to $\psi \coloneqq \varphi\chi$.
		
		\medskip
		\noindent\textit{(v)}
		Let still $X = [U/G]$, put $Y \coloneqq  [U^{G^0}/G]$, and write $p \colon X \to BG$ for the projection. Since $Y \to X$ is representable and separated, we have a natural map $Y^d \to X^d$ by (ii). Also, since $G^0$ acts trivially on $U^{G^0}$, by (iv) the structure map $Y^d \to Y$ is an equivalence. We therefore get a natural morphism $h \colon Y \to X^d$ over $X$, which we claim to be an equivalence.
		
		For any map $f \colon T \to X$, using Propositions~\ref{Prop:slice_FOP} and \ref{Prop:slice_FOP_2}, the morphism $h(T) : Y(T) \to X^d(T)$ is equivalent to
		\begin{align} \label{loc 5.3}
			\Stk_{T/\cdot/BG}(T \times_{B(G/G^0)} BG, X) \xrightarrow{h(T)} \Grp(\Stk_{\widetilde{G}_{pf}})(\widetilde{G}_{pf}^0, f^\ast \lL X) .
		\end{align}
		
		We need to check that this morphism is an equivalence. To this end, observe that
		\[    X^d(T) \simeq \Grp(\Stk_T)_{/\Omega_T(BG)}(\tilde{G}^0_{pf}, \Omega_T(X)) \simeq \Stk_{T/\cdot/BG}(B_T(\tilde{G}^0_{pf}),X)   \]
		by the adjunction $B_T \dashv \Omega_T$ from Remark \ref{Rem:B_loop_adjunction}. Since $B_T(\widetilde{G}_f^0) \simeq T \times_{B(G/G^0)} BG$ by Lemma  \ref{Lem:inner_forms}, the claim follows. 

		\medskip
		\noindent\textit{(vi)}
		By the \'{e}tale local structure of $X$ (Propositions \ref{Prop:luna_etale_slic} and \ref{Prop:strongly etale is monodr-pres and separated}), we can reduce to the case where $X=[U/G]$ by (iii). The statement then follows from (v).
		
		\medskip
		\noindent\textit{(vii)} Let $f \colon T \to X$ be given with $T$ a classical scheme, so that we have a unique factorization $f \colon T \to X_\cl \to X$. The space $(X^d)_\cl(T)$  thus classifies derived closed subgroup schemes $H^0 \to T \times_X \lL X$ over $T$. Since $T$ is classical and $H^0 \to T$ is smooth, we must have $H^0 = (H^0)_\cl$, and $H^0 \to T \times_X \lL X$ is naturally equivalent to the closed embedding of group schemes $H^0 \to (T \times_X \lL X)_\cl = T \times_{X_\cl} (\lL X)_\cl = T \times_{X_\cl} \lL X_\cl$. These are the $T$-points of $(X_\cl)^d$ by \cite[Prop.~C.5]{EdidinRydh}.
		
		\medskip
		\noindent\textit{(viii)} Since $X^{\leq d} \to X$ is an open immersion, the factorization exists uniquely by (i). To check that the induced map $X^d \to X^{\leq d}$ is a closed immersion, we reduce to the classical case by (vii). By the argument for (vii), the statement then follows from \cite[Prop.~C.5]{EdidinRydh}.
	\end{proof}

	\begin{proof}[Proof of Theorem~\ref{thm existence of X max}]
		Everything follows from Proposition \ref{Prop:luna_etale_slic} and Proposition \ref{Prop:prop_xmax}.
	\end{proof}
	
	The functoriality of $(-)^d$ shown in Proposition \ref{Prop:prop_xmax}, together with the functoriality of $(-)^{\maxlocus}$ in the classical case from \cite[Appendix C]{EdidinRydh}, has the following immediate consequence:
	
	\begin{corollary}
		If $X' \to X$ is \'{e}tale, representable and separated, then $X^{\maxlocus} \times_X X' \simeq {X'}^{\maxlocus}$, provided that $X$ and $X'$ have the same maximal stabilizer dimension. If not, then the maximal stabilizer dimension of $X'$ is strictly smaller than the one for $X$, so that $X^{\maxlocus} \times_X X' \simeq \emptyset$.
	\end{corollary}

	\section{Intrinsic blow-ups as equivariant derived blow-ups} \label{Sec:intr bl is der bl}
	
	Let $U$ be a classical scheme with an action by a reductive group $G$. The main aim of this section is to prove that the $G$-intrinsic blow-up of $U$, introduced in \cite{KLS}, has a natural derived enhancement which is obtained by a corresponding derived blow-up construction, namely blowing up $U$ along its derived fixed locus $U^G$.
	
	To this end, we begin with a review of intrinsic blow-ups and then proceed to establish the main result comparing intrinsic and derived blow-ups by using standard equivariant local models for $U$. For simplicity, throughout, we initially assume that $G$ is connected and then explain how to adapt the constructions in the non-connected case. We conclude with some remarks and a few instructive examples.
	
	The operations of intrinsic blow-up and its derived enhancement will be fundamental building blocks in our derived stabilizer reduction procedure in Section~\ref{Sec:der stab red}.

	\subsection{Background on intrinsic blow-ups}
	
	Let $U$ be a classical affine scheme with an action of a reductive group $G$. Assume for simplicity that $G$ is connected. 
	
	Take an equivariant closed embedding $U \to V$ into a smooth $G$-scheme $V$ and let $I$ be the ideal defining $U$. Since $U \subseteq V$ is $G$-equivariant, $G$ acts on $I$ and we have a decomposition $I = I^{\fix} \oplus I^{\mv}$ into the fixed part of $I$ and its complement as $G$-representations.

	Let $V^G$ be the fixed point locus of $G$ inside $V$, defined by the ideal generated by $\oO_V^{\mv}$, and let $\pi \colon \mathrm{bl}_G(V) \to V$ be the blow-up of $V$ along $V^G$.
	Let $E\subseteq \mathrm{bl}_G(V)$ be its exceptional divisor and $\xi \in \oO_{\mathrm{bl}_G(V)}(E)$ the tautological defining equation of $E$. 
	
	$G$-equivariance implies that $\pi^{-1} (I^{\mv}) \subseteq \xi \cdot \oO_{\mathrm{bl}_G(V)}(-E)$ (cf.\ \cite[Section~2.2]{KLS}), and consequently,
	\begin{equation*}\label{xi}
		\xi^{-1} \pi^{-1} (I^{\mv}) \subseteq \oO_{\mathrm{bl}_G(V)}(-E) \subseteq \oO_{\mathrm{bl}_G(V)}.
	\end{equation*}
	We define $I^{\intr}\subseteq\oO_{\mathrm{bl}_G(V)}$ to be the ideal sheaf
	\begin{align} \label{tilde I}
		I^{\intr} = \text{ideal sheaf generated by } \pi^{-1}(I^{\fix}) \text{ and } \xi^{-1} \pi^{-1} (I^{\mv}).
	\end{align}

	\begin{definition}[\cite{KLS}] \label{intrinsic blow-up def}
		The (classical) \emph{$G$-intrinsic blow-up} of $U$ is the classical subscheme $U^{\intr}_\cl\subseteq \mathrm{bl}_G(V)$ defined by the ideal $I^{\intr}$. Observe that $U^{\intr}_\cl \subset \mathrm{bl}_G(V)$ is $G$-invariant. For $X = [U/G]$, we write $X^{\intr}_\cl \coloneqq [U^{\intr}/G]$. 
	\end{definition}
	
	In \cite{KLS}, it is proved by hands-on computation that $U^{\intr}_\cl$ is intrinsic to the $G$-scheme $U$. Moreover, by \cite{KLS, Sav}, the construction globalizes to any scheme $U$ with a $G$-action and a Zariski open cover by $G$-invariant subschemes.

	\begin{example} \label{intrinsic blow-up example}
		Suppose that $G = \Gm$ is the one-dimensional torus acting on the affine plane $V = \CC^2_{x,y}$ with weights $1$ and $-1$ on the coordinates $x$ and $y$ respectively, and $U \subseteq V$ is the closed $G$-invariant subscheme cut out by the ideal $I = (x^2 y, x y^2)$. 
		
		$V^{\intr}_{\cl}$ is the blow-up of $V$ along the fixed locus $V^G = \lbrace 0 \rbrace$ cut out by the ideal $\oO_V^{\mv} = (x,y)$. 
		Writing $u,v$ for the homogeneous coordinates on the exceptional divisor $\PP^1$, we have that $G$ acts linearly on $u, v$ with weights $1, -1$ respectively, and the blow-down map $V^{\intr}_{\cl} \to V$ is locally given on coordinates by $x \mapsto \xi u, y \mapsto \xi v$.
		
		Since $I^{\mv} = (x^2 y, xy^2)$,  the closed subscheme $U^{\intr}_{\cl} \subseteq V^{\intr}_{\cl}$ is locally cut out by the ideal $I^{\intr} = (\xi^2 u^2 v, \xi^2 u v^2)$.
	\end{example}
	
	We now briefly explain how one can proceed if $G$ is not connected. 
	
	Suppose that $U \to V$ is a $G$-equivariant embedding into a smooth $G$-scheme $V$. Let, as before, $I$ be the ideal of $U$ in $V$. Let $G^0$ be the connected component of the identity. This is a normal, connected subgroup of $G$ of finite index. Let $I = I^{\fix} \oplus I^{\mv}$ be the decomposition of $I$ into fixed and moving parts with respect to the action of $G^0$. Using the normality of $G^0$, we see that the fixed locus $V^{G^0}$ is a closed, smooth $G$-invariant subscheme of $V$, and also that $I^{\fix},\ I^{\mv}$ are $G$-invariant. 
	
	Let $\pi \colon \mathrm{bl}_{V^{G^0}} V \to V$ be the blow-up of $V$ along $V^{G^0}$ with exceptional divisor $E$ and local defining equation $\xi$. Then, as before, take $I^{\intr}$ to be the ideal generated by $\pi^{-1}(I^{\fix})$ and $\xi^{-1}\pi^{-1}(I^{\mv})$. Everything is $G$-equivariant, and we define $U^{\intr}_{\cl}$ as the subscheme of $\mathrm{bl}_{V^{G^0}} V$ defined by the ideal $I^{\intr}$. At the level of quotient stacks, $X^{\intr}_{\cl} \coloneqq [U^{\intr}_{\cl} / G]$ is the intrinsic blow-up of $X \coloneqq [U/G]$.

	\subsection{Intrinsic blow-ups are classical truncations of derived blow-ups}
	
	Consider a derived affine scheme $U$, quasi-compact and locally of finite presentation, with an action by a connected reductive group $G$ such that the (derived) fixed locus $(U_\cl)^G$ is non-empty. The following is the main result of this section.
	
	\begin{theorem} \label{intr bl is der bl thm}
		The $G$-intrinsic blow-up of the classical truncation $U_\cl$ is naturally isomorphic, as a classical scheme with $G$-action, to the classical truncation $(\Bl_{U^G} U)_\cl$ of the derived blow-up of $U$ along the fixed locus $U^G$. 
	\end{theorem}

	\begin{proof}
		The statement is local. Let $x \in (U_\cl)^G$. 
		
		By Lemma~\ref{Lem:Equiv alg are standard form}, after possibly (Zariski) equivariantly shrinking $U$ around $x$, we may reduce to the case where $U = \Spec R'$ and $R'$ corresponds to a cdga in standard form (with homological indexing) 
		\begin{align}
			R = [ \dots \lr R_{2} \lr R_{1} \lr R_0],
		\end{align}
		where $R_0 = \CC[x_1, ..., x_N]$, $G$ acts linearly on the variables $x_i$, and $R'$ is generated in each positive degree by a $G$-representation $V_i$ with variables $y_j^i$ corresponding to a basis for each $V_i$. 
		
		Let $V = \Spec R_0$. By Proposition \ref{Prop:3.27}, we then know that $U^G = \Spec S$, where $S$ is a cdga in standard form, with $S_0 = \oO_{V^G}$, and $S$ generated in each positive degree $i$ by the $G$-fixed variables of $V_i$.

		Consider now the auxiliary scheme $F^1 = \Spec S^1$, where $S^1_0 = \oO_{V^G}$ and $S^1$ is generated in degree $1$ by the $G$-fixed variables of $V_{1}$ and by all of the variables of $V_i$ for $i \geq 2$.
		
		There is a composition $U^G \to F^1 \to U$, where the morphism $U^G \to F^1$ is easily seen to be a sequence of cell attachments in degrees bigger than $1$, using the fact that $V_{2}^{\mv}$ maps to $0$ in $S_1^{1}$ by $G$-equivariance.
		
		In particular, Proposition~\ref{cell attachment in deg -2 prop} gives an isomorphism 
		\begin{align}
			(\Bl_{U^G} U)_\cl \simeq (\Bl_{F^1} U)_\cl,
		\end{align}
		
		and we are thus reduced to calculating $(\Bl_{F^1} U)_\cl$.
		
		Let $A \subseteq R$ be the subalgebra of $R$ generated by $R_0$ and the $G$-fixed variables in degree $1$. 
		
		Without loss of generality, let $x_1, ..., x_l$ be the moving variables in $R_0$ and $f_1, ..., f_n \in (x_1, ..., x_l) = \left( R_0^{\mv} \right)$ the images of the moving variables in $V_{1}$. Set 
		$$B = A / (f_1, ..., f_n),\ D = A / (x_1, ..., x_l)$$
		with the obvious cell attachment maps.
		
		Let $Z = \Spec D,\ X = \Spec B,\ Y = \Spec A$. It is clear that $B$ is the subalgebra of $R$ generated by $R_0$ and $R_{1}$ and we have a morphism $U \to X$.
		
		We then have $Z_\cl = (U_\cl)^G$, $X_\cl = U_\cl$ and a Cartesian square
		\begin{align}
			\vcenter{\xymatrix{
				F^1 \ar[d] \ar[r] & U \ar[d] \\
				Z \ar[r] & X.
			}}
		\end{align}
		
		Proposition~\ref{pullback square classical truncation blow-ups} then implies that 
		\begin{align}
			(\Bl_{F^1} U)_\cl \simeq (\Bl_Z X)_\cl.
		\end{align}
		
		But, by \eqref{formula [-1,0] case}, it immediately follows that $(\Bl_Z X)_\cl$ is the $G$-intrinsic blow-up of $X_\cl = U_\cl$, as we want, obtained using the equivariant closed embedding $U_\cl \to V$.
		
		Indeed, $V^G = \Spec R_0/(x_1,\ldots,x_l)$, since $x_1, \ldots, x_l $ are the moving variables among the $x_i$. 
		
		Letting $I$ be the ideal of the canonical inclusion $U_\cl \to V$, $I$ is then the image of $R_{1}$, hence $I^{\mv} = (f_1,\dots,f_n)$. Write $f_i = \sum_{ij} \lambda_{ij} x_j$. Thus in the Rees algebra
		\begin{align*}
			R_{V^G/V}^{G, \extd}= \frac{R_0[t^{-1},w_1,\dots,w_l]}{(t^{-1}w_1-x_1,\dots,t^{-1}w_l-x_l)}
		\end{align*}
		it holds that $\xi^{-1} \pi^{-1}f_j = \sum_{ij} \lambda_{ij} w_j$, where $\pi\colon \mathrm{bl}_{V^G}V \to V$ is the projection. Likewise, $\pi^{-1}(I^{\fix})$ is generated by $\partial V_{1}^{\fix}$ and hence equals the image $\partial A_{1}$, where $\partial$ is the boundary map.
		
		We have that the intrinsic blow-up of $U_\cl$ is the projective spectrum of the discrete quotient
		\begin{align*}
			\frac{R_0[t^{-1},w_1,\dots,w_l]}{(t^{-1}\underline{w}-\underline{x}, \partial A_{1}, \sum_{1j}\lambda_{1j}w_j,\dots, \sum_{nj} \lambda_{nj}w_j)} \simeq \frac{ \pi_0(A)[t^{-1},\underline{w} ] }{ (t^{-1}\underline{w} - \underline{x}, \sum_{ij} \lambda_{ij} w_j) }
		\end{align*}
		which we recognize as $\pi_0 R^{G, \ext}_{Z/X}$, using~\eqref{formula [-1,0] case}.
	
	\end{proof}
	
	A special case of the theorem which is worth noting is when $U$ is classical to begin with.
	
	\begin{corollary}
		Let $U$ be a classical affine scheme, of finite type and locally of finite presentation as a derived scheme, with a $G$-action. Then its classical $G$-intrinsic blow-up $U^{\intr}_{\cl}$ is the classical truncation of the derived blow-up $\Bl_{U^G} U$ of $U$ with center the \textbf{derived} fixed locus $U^G$.
	\end{corollary}
	
	\begin{remark} \label{Rem: lft vs lfp}
		Even though the above corollary is stated for a classical scheme of finite type and of local finite presentation as a derived scheme, this is not a restriction, and it still holds in the more general case of a scheme of finite type. This is because the proof of Lemma~\ref{Lem:Equiv alg are standard form} for a classical $G$-scheme of finite type shows that one can obtain a local model $U = \Spec R$ as a possibly infinite sequence of equivariant cell attachments, which are finitely many in each homological degree, i.e., a possibly infinite equivariant standard form. This, however, does not affect the computations, as is clear from our argument above. We elect to assume local finite presentation for simplicity. The same applies to Section~\ref{Sec:der stab red}.
	\end{remark}
	
	\begin{remark} 
		Since $U^G$ is well-defined, it is now conceptually clear why the intrinsic blow-up only depends on the $G$-scheme $U$ and not the auxiliary data used to define it, which is a priori  not obvious and had to be checked by hand in \cite{KLS}.
		
		Another non-obvious consequence of the above results is the fact that the classical truncation of the blow-up $\Bl_{U^G} U$ for a derived $G$-scheme $U$ only depends on its classical underlying $G$-scheme $U_\cl$.
	\end{remark}
	
	\subsection{A couple of examples} 
	\label{Subsec:Coupleoexample}
	We give two examples which hopefully clarify the arguments in this section and will also be useful guides for intuition. Both are derived enhancements of Example~\ref{intrinsic blow-up example} and we can directly verify that the classical truncation of the derived blow-up in each of them coincides with the intrinsic blow-up described in Example~\ref{intrinsic blow-up example}.
	
	Throughout, we put $A \coloneqq \CC[x,y]$ and $B \coloneqq A/(xy^2,yx^2)$, so that 
	\[ (U \to V) \coloneqq (\Spec B \to \Spec A) \]
	is the derived critical locus of $f = x^2y^2$, considered as map $V \to \AA^1$. Write $Z = \Spec \CC$, and consider the closed embedding $Z \to U$ given by the quotient map $B \to A/(x,y) \simeq \CC$. 
	
	\begin{example}
		\label{Ex:Firstocouple}
		Lemma \ref{functoriality of deformation spaces lemma} gives us that $D_{Z/U} \simeq D_{Z/V} \times_{D_{U/V}} D_{U/U}$, and Proposition \ref{Prop:RZU}, gives
		\[ 	R_{Z/U}^\extd \simeq \frac{A[t^{-1},w_1,w_2]}{(t^{-1}w_1-x, t^{-1}w_2 - y, y^2w_1,x^2w_2)} \simeq \frac{A[It,t^{-1}]}{(xy^2t,yx^2t)} \]
	\end{example}	
	
	\begin{example} 
\label{Ex:Secondocouple}
		The Koszul complex defining the algebra $B$ looks like $A \to A \oplus A \to A$, which gives us an element $(x,-y) \in A \oplus A$ that gives rise to a nonzero element 
		\begin{align*}
			\sigma \in \pi_1 B \simeq \pi_0 \Map(\LSym(B[1]), B)
		\end{align*}
		and thus a map $f_\sigma \colon \LSym(B[1]) \to B$. Let $C$ be the pushout of $f_\sigma$ along the map $\LSym(B[1]) \to B$ induced by the zero-map $B[1] \to B$, so that $C = B/(\sigma)$. 
		
		Put $W = \Spec C$. Write $\tau$ for the element of $R^{\extd}_{Z/U}$ in homological degree $1$, corresponding to the composition $f_\sigma \colon \LSym(B[1]) \to B \to R_{Z/U}^\extd$. Then, again by Proposition \ref{Prop:RZU}, it holds
		\[ R^\extd_{Z/W} \simeq R^\extd_{Z/U}/(\tau) \] 
		In terms of the Koszul complex, $\tau$ corresponds to the element $(xt, -yt)$ in homological degree $1$ of the complex $R \to R \oplus R \to R$, where $R = A[It,t^{-1}]$ for $I=(x,y)$

	\end{example}

	\section{Derived stabilizer reduction of Artin stacks} \label{Sec:der stab red}
	Let $X$ be a derived algebraic stack. Throughout this section, we assume that $X$ is quasi-compact and locally of finite presentation, and that $X_\cl$ admits a good moduli space $q \colon X_\cl \to M$. Observe that $X_\cl$ is then of finite type, hence Noetherian.  Recall that all derived stacks (hence all classical stacks, good moduli spaces, etc.) are assumed to have affine diagonal.
	\medskip
	
	Under these assumptions, an intrinsic stabilizer reduction procedure was developed in \cite{KLS, Sav}, producing a canonical sequence of classical Artin stacks
	\begin{align} \label{classical intr stab red seq}
		X_\cl^0 = X_\cl,\ X_\cl^1 \coloneqq \hat{X}_\cl^0,\ \ldots,\ X_\cl^m \coloneqq \hat{X}_\cl^{m-1}
	\end{align}
	where each $X_\cl^i$ is obtained by applying an operation called ``Kirwan blow-up'' on the preceding stack $X_\cl^{i-1}$ (denoted by the top hat), and the maximal stabilizer dimension of the stacks in the sequence is strictly decreasing, so that $X_\cl^m$ has only finite stabilizers. The Deligne--Mumford stack $\widetilde{X}_\cl \coloneqq X_\cl^m$ is called the \emph{intrinsic stabilizer reduction} of $X_\cl$.
	
	In this section, we use the results obtained so far to upgrade this construction to the derived context. More precisely, we define a derived Kirwan blow-up operation whose classical truncation is the classical Kirwan blow-up. This allows us to construct a derived stabilizer reduction procedure by a canonical sequence of derived Artin stacks
	\begin{align}
		X^0 = X,\ X^1 \coloneqq \hat{X}^0,\ \ldots,\ X^m \coloneqq \hat{X}^{m-1}
	\end{align}
	whose classical truncation is~\eqref{classical intr stab red seq}. We will define the derived Deligne--Mumford stack $\widetilde{X} \coloneqq X^m$ to be the derived stabilizer reduction of $X$.
	
	\subsection{Review of the classical case} 
	\label{Sub:KirwanClassical}
	We briefly recall the notions of intrinsic and Kirwan blow-up for a classical Artin stack $X_\cl$ with a good moduli space $q \colon X_\cl \to M$, as considered in \cite{Sav}. Let $M^{\maxlocus}$ be the good moduli space of $X_\cl^{\maxlocus}$.
	\medskip
	
	The classical intrinsic blow-up 
	$$\pi^{\intr} \colon {X}_\cl^{\intr} \to X_\cl$$ 
	of $X_\cl$ is defined by \'{e}tale descent using affine, strongly \'{e}tale morphisms $[U_\cl/G] \to X_\cl$, where $G$ runs along all reductive stabilizer groups of $X_\cl$ of maximum dimension. 
	Given such, ${X}_\cl^{\intr}$ is obtained through a corresponding affine, \'{e}tale, stabilizer-preserving cover by morphisms
	\begin{align}
		[U^{\intr}_\cl / G] \to {X}_\cl^{\intr}
	\end{align}
	glued together with the complement $X_\cl \setminus X_\cl^{\maxlocus} \subseteq X_\cl$ of $X^{\maxlocus}$, which is unaffected by the intrinsic blow-ups of the cover.
	\medskip
	
	The \emph{classical Kirwan blow-up} is the open substack of \emph{semi-stable points} $\left( X_\cl^{\intr} \right)^{\mathrm{ss}}$ of the intrinsic blow-up $X_\cl^{\intr}$, whose complement by definition is the \emph{unstable locus}
	$(X_\cl^{\intr})^{\mathrm{us}} \subseteq X_\cl^{\intr}$, defined in \cite{Sav}. 
	
	We briefly remind the reader that it is determined as follows: Let $x \in X_\cl^{\max}$ be any closed point with reductive stabilizer $G=G_x$ of maximal dimension. Let $[U_\cl/G] \to X_\cl$ be a strongly \'{e}tale morphism, whose image contains $x$, where $U_\cl$ is classical, affine and equipped with a $G$-equivariant closed embedding $U_\cl \to V$ into a smooth, affine, classical $G$-scheme $V$. Write $\pi^{\intr} \colon V^{\intr} \to V$, as usual, for the blow-up of $V$ along $V^{G^0}$, and $q \colon V \to V \git G$ for the good moduli space morphism. By definition, we have a $G$-invariant closed embedding $U^{\intr} \to V^{\intr}$. The unstable locus of $V^{\intr}$ is then the strict transform of the saturation of the closed subscheme $V^{G^0} \git G = q(V^{G^0})$ inside $V \git G$, i.e.
	\begin{align*}
	    (V^{\intr})^{\mathrm{us}} := \overline{(\pi^{\intr})^{-1} \left( q^{-1}(V^{G^0}\git G) \setminus V^{G^0} \right) },
	\end{align*}
	so that $(V^{\intr})^{\mathrm{ss}} = V^{\intr} \setminus (V^{\intr})^{\mathrm{us}}$. In \cite{EdidinRydh}, the quotient stack $ [ (V^{\intr})^{\mathrm{ss}} / G ] $ is referred to as the Reichstein transform of $[V/G]$. We then define $(U^{\intr})^{\mathrm{ss}} = U \cap (V^{\intr})^{\mathrm{ss}}$.

	One may check that this is independent of choices and hence well-defined. We then take $(X_\cl^{\intr})^{\mathrm{ss}}$ to be the open substack of $X_\cl^{\intr}$ whose points are in the image of the composition $[(U_\cl^{\intr})^{\mathrm{ss}} / G] \to [U_\cl^{\intr} / G] \to X_\cl^\intr$ for all possible choices of $x \in X_\cl^{\max}$ and strongly \'{e}tale morphisms $[U_\cl / G] \to X$ in the above. 
	\medskip
	
	An important property established in \cite{Sav} is the existence of a good moduli space $(X^\intr_\cl)^{\ss} \to \hat{M}$ for the Kirwan blow-up such that the natural induced morphism $\hat{M} \to M$ is proper. By construction, the maximum stabilizer dimension of points in $(X^\intr_\cl)^{\ss}$ is strictly lower than that of $X_\cl$.

	\subsection{Derived intrinsic and Kirwan blow-ups}
	\label{Subsec:Derived_intrinsic_and_Kirwan_blow-ups}
	If $X$ is a derived Artin stack satisfying the assumptions of this section, we have constructed earlier a canonical, closed immersion $X^{\maxlocus} \to X$, which parametrizes points in $X$ with stabilizer of maximal dimension.
	
	As in the classical case, we may thus give the following definition.
	
	\begin{definition}
		The \textit{derived intrinsic blow-up} of $X$ is defined as the stack over $X$
		\[\pi^{\intr} \colon X^{\intr} \coloneqq \Bl_{X^{\maxlocus}} X \lr X.\]
	\end{definition}
	
	As expected, the derived intrinsic blow-up is a canonical derived enhancement of the classical intrinsic blow-up (which also justifies the notation).
	
	\begin{proposition} \label{Prop:stacky intr blow-up is derived blow-up}
		The classical truncation of $X^{\intr}$ is the classical intrinsic blow-up of the classical truncation $X_\cl$ of $X$.
	\end{proposition}
	
	\begin{proof}
		By the base change property of $X^{\maxlocus}$ along a cover by affine, strongly \'{e}tale morphisms from affine quotient stacks, we may reduce to the case of a derived affine quotient stack $[U/G]$ for which $[U/G]^{\maxlocus} = [U^{G^0} / G]$. But then the statement follows immediately by Theorem~\ref{intr bl is der bl thm} by working with the connected component $G^0$ and $G$-equivariance.
	\end{proof}
	
\begin{lemma}
	\label{Lem:Xmaxlfp}
	The derived stack $X^{\maxlocus}$ is locally of finite presentation. Consequently, the morphism $\pi^{\intr} \colon X^{\intr} \to X$ is locally of finite presentation.
\end{lemma}
	
	\begin{proof}
		In \cite{Weil} it is shown that $\Bl_{X^{\maxlocus}}X$ is an open substack of the deformation space $\stD_{X^{\maxlocus}/X}$, and that Weil restrictions preserve the property of being locally finitely presented. Since $X$ is assumed to be locally of finite presentation, it thus suffices to show that $X^{\maxlocus}$ is locally of finite presentation. 
		
		As before, taking slices and using Theorem \ref{thm existence of X max}, we reduce to the case $X= [U/G]$, where $U$ is a derived $G$-scheme, locally of finite presentation. Since $U^{G^0}$ is obtained by Weil-restricting along $BG \to *$, the claim follows.
	\end{proof}
	
Let $\sA$ be a discrete, finitely generated, quasi-coherent, graded $\sO_{X_\cl}$-algebra, and write $q \colon X_\cl \to Y$ for the good moduli space. Recall that the \emph{saturated projective spectrum} $\Proj^{q}_{X_\cl} \sA$ (relative to $q$) is the largest open substack of $\Proj_{X_\cl} \sA$ such that the natural map
\[ \Proj^{q}_{X_\cl} \sA \to \Proj_Y q_* \sA \]
is well-defined, in which case it is a good moduli space \cite[\S 3.1]{EdidinRydh}. In case $\sA$ is the Rees algebra of a closed immersion, this is called the \emph{saturated blow-up}.

\begin{definition}
	Let $\sB$ be a quasi-coherent, graded $\sO_{X}$-algebra such that $\pi_0\sB$ is finitely generated. Then the \emph{derived saturated projective spectrum} is the open substack $\Proj^q_X \sB$ of the derived projective spectrum $\Proj_X \sB$ such that
	\[ (\Proj^q_X \sB)_\cl \to (\Proj_X \sB)_\cl \cong \Proj_{X_\cl} \pi_0 \sB \]
	is the saturated projective spectrum of $\pi_0 \sB$.
\end{definition}

By Lemma \ref{Lem:Xmaxlfp} it holds that $\pi_0\sR_{X^{\maxlocus}/X}$ is finitely generated. The following therefore makes sense.

\begin{definition}
	The \textit{derived Kirwan blow-up} of $X$ is defined as
	\[ \pi \colon \hat{X} \coloneqq \Proj^q_X \sR_{X^{\maxlocus}/X}\lr X. \]
	The \emph{semi-stable locus} of $X^\intr$ is the open substack $(X^\intr)^{\ss} \subseteq X^{\intr}$ whose classical truncation is the classical semi-stable locus $(X^{\intr}_\cl)^{\ss} \subseteq X^{\intr}_\cl$ discussed in \S~\ref{Sub:KirwanClassical}.
\end{definition}

The connection with the classical Kirwan blow-up is as follows.

\begin{proposition}
	\label{Prop:KirwanSaturated}
	The derived Kirwan blow-up $\hat{X}$ coincides with the semi-stable locus $(X^\intr)^{\ss}$ as open substacks of $X^{\intr}$.
\end{proposition}

\begin{proof}
	It suffices to show that the semi-stable locus of $X^{\intr}_\cl$ discussed in \S~\ref{Sub:KirwanClassical} coincides with the saturated blow-up $\hat{X}_\cl = \Proj^{q}_{X_\cl} \pi_0 \sR_{X^{\maxlocus}_\cl/X_\cl}$. By \cite[Thm.~4.14]{Sav}, \cite[Prop.~3.6]{EdidinRydh} and since $\hat{X}_\cl, (X^{\intr}_\cl)^{\ss}$ are open substacks of $X^{\intr}_\cl$, we reduce to $X = [U/G]$ with $U$ a classical affine scheme. Then, by definition of the semi-stable locus and by \cite[Prop.~3.6]{EdidinRydh}, we further reduce to the case where $U$ is smooth. Now the claim follows from \cite[Prop.~4.5]{EdidinRydh} and \cite[Thm.~4.5]{Sav}.
\end{proof}

	We summarize the fundamental properties of the derived intrinsic and Kirwan blow-up in the following theorem, which combines the above (and Proposition \ref{Prop:prop_xmax}) with the properties of the corresponding classical blow-ups given in \cite[Thm.~4.7]{Sav}. 
	
	\begin{theorem} \label{der intr kir blow-up thm}
		Let $X$ be a derived algebraic stack satisfying the assumptions of this section: $X$ is quasi-compact and locally of finite presentation, and $X_\cl$  admits a good moduli space $q \colon X_\cl \to M$ (with affine diagonal by assumption).
		\begin{enumerate}
			\item The Kirwan blow-up $\pi \colon \widehat{X} \to X$ also satisfies the assumptions of this section.
			\item The maximum stabilizer dimension of closed points in $\hat{X}$ is strictly smaller than that of $X$.
			\item For any affine, strongly \'{e}tale morphism $[U/G] \to X$ with $U$ affine, the base change $\hat{X} \times_{X} [U/G]$ is naturally isomorphic to the derived Kirwan blow-up of $[U/G]$. 
			\item $\pi |_{\pi^{-1}(X^{\mathrm{s}})}$ is an isomorphism over the open locus $X^{\mathrm{s}}$ of properly stable points, i.e., closed points $x \in \lvert X \rvert$ with finite stabilizer such that $q^{-1}(q(x)) = \lbrace x \rbrace$.
			\item 	The classical truncations of $X^{\intr}$ and $\hat{X}$ are the classical intrinsic and Kirwan blow-ups of the classical truncation $X_\cl$ respectively.
		\end{enumerate} 
	\end{theorem}

\begin{remark}
	By Proposition \ref{Prop:KirwanSaturated}, we can also use \cite{EdidinRydh} to study $\hat{X}$. For example, \cite[Prop.~3.4]{EdidinRydh} produces an explicit, global description of the good moduli space $\hat{q} \colon \hat{X}_\cl \to \hat{M}$, thus circumventing the gluing argument from \cite{Sav}.
\end{remark}
	
	\begin{example}
		Let $Y \coloneqq [\AA^2/\Gm]$, where $\Gm$ acts on $\AA^2$ with weights $1,-1$, and let $Z \to Y$ be the closed substack corresponding to $\{0\} \to \AA^2$. Since $Y^{\max} = Z$, from \cite[Ex.\ 4.6]{EdidinRydh} we know that 
		\[ \hat{Y} \cong (\Bl_{Z}Y) \setminus [(D_1 \sqcup D_2)/\Gm] \]
		where $D_1,D_2$ are the strict transforms of the coordinate axes. 
	\end{example}
	
	\subsection{Derived stabilizer reduction of Artin stacks} 
	\label{Subsec:Derived_stabilizer_reduction_of_Artin_stacks}
	We are now in position to carry out the derived stabilizer reduction procedure, as described in the introductory part of the section.
	
	By Theorem~\ref{der intr kir blow-up thm}, the Kirwan blow-up $\hat{X}$ preserves the assumed properties of $X$. In particular, we may repeatedly apply it until the maximum stabilizer dimension becomes zero and we obtain a derived Deligne--Mumford stack $\widetilde{X}\coloneqq X^m$ by the canonical sequence
	\begin{align}
		X^0 = X,\ X^1 \coloneqq \hat{X}^0,\ \ldots,\ X^m \coloneqq \hat{X}^{m-1}.
	\end{align}
	
	\begin{definition}
		The derived Deligne--Mumford stack $\widetilde{X} \to X$ is called the \emph{derived stabilizer reduction} of $X$.
	\end{definition}
	
	It is clear by Theorem~\ref{der intr kir blow-up thm} that the classical truncation of $\widetilde{X}$ is the intrinsic stabilizer reduction of $X_\cl$.
	
	\begin{remark}
		By Remark~\ref{Rem: lft vs lfp}, if $X$ is a classical Artin stack of finite type, then its derived stabilizer reduction $\tilde{X}$ is well-defined, with classical truncation the intrinsic stabilizer reduction of $X_\cl$.
	\end{remark}
	
	\subsection{Connections with other approaches} We conclude with a brief discussion regarding the connections between the following stabilizer reduction approaches in the literature:
	\begin{enumerate}
		\item Kirwan's original desingularization procedure for classical, smooth quotient stacks obtained by Geometric Invariant Theory (GIT) \cite{Kirwan}.
		\item Edidin--Rydh's canonical reduction of stabilizers for classical (possibly singular) Artin stacks with good moduli spaces \cite{EdidinRydh}.
		\item The intrinsic stabilizer reduction procedure introduced in \cite{KLS, Sav} for (possibly singular) GIT quotient stacks and Artin stacks with good moduli spaces.
		\item The derived stabilizer reduction procedure of the present paper.
	\end{enumerate}
	
	In general, (ii) is a generalization of (i), while we have shown that (iv) is always a derived enhancement of (iii).
	
	Both (ii) and (iii) are sequences of two-step operations consisting of a blow-up and the deletion of certain unstable points (this two-step operation in (ii) is called a saturated blow-up, while in (iii) a Kirwan blow-up). The main difference between (ii) and (iii) is in the blow-up being used, which can now be succinctly formulated as follows, using Theorem \ref{der intr kir blow-up thm}. For a classical stack $X$, the classical blow-up $\Bl^\cl_{X_\cl^{\maxlocus}} X$ along the classical locus $X_\cl^{\maxlocus}$ is used in (ii), while the underlying classical stack $(\Bl_{X^{\maxlocus}} X)_\cl$ of the derived blow-up of $X$ (viewed as a derived stack) along the derived locus $X^{\maxlocus}$ is used in (iii). Recall here that $X^{\maxlocus}_\cl$ is both the classically defined locus and the underlying classical stack of the derived locus.
	
	For classical GIT quotient stacks that are smooth, all four constructions coincide: approaches (i)--(iii) are literally identical, while in (iv) it suffices to notice that the derived intrinsic blow-up of a classical smooth Artin stack is the same as the classical intrinsic blow-up (since the derived $X^{\maxlocus}$ is smooth and coincides with the classical $X^{\maxlocus}$) so that the derived Kirwan blow-up is classical and equals the saturated blow-up in each step. 
	
	However, in the singular case, the derived enhancement $X^{\maxlocus}$ of the classical locus $X_\cl^{\maxlocus}$ is in general not itself classical, as can be seen in the following example. Hence, the outputs of (ii) and (iv) will typically differ.
	
	\begin{example}
		Suppose that $G = \Gm =  \CC^\times$ is the one-dimensional torus acting on the affine plane $V = \CC^2_{x,y}$ with weights $1$ and $-1$ on the coordinates $x$ and $y$ respectively, and that $U \subseteq V$ is the closed $G$-invariant subscheme cut out by the ideal $I = (xy)$. Let $X \coloneqq [U/G]$. Then the classical locus $X^{\maxlocus}_\cl$ is the quotient stack $\Spec \CC[x,y]/(x,y) \times BG$. 
		
		To describe the derived locus $X^{\maxlocus}$, we can find a model $B$ for $\oO_U$ in standard form, and then apply Proposition~\ref{Prop:3.27}. In this case, it is clear that we can take $B(0) = \CC[x,y]$ and $M_0 = B(0) \otimes_{\CC} W_0$, where $W_0$ is the trivial $1$-dimensional $G$-representation and $M_0 \to B(0)$ is multiplication by $xy$. But then, $B_G$ is freely generated by $\CC[x,y]/(x,y)$ in degree $0$ and $W_0$ in degree $1$, therefore the derived $X^{\maxlocus} \simeq \Spec B_G \times BG$ is not classical and differs from its classical counterpart.
	\end{example}
	
    \section{Quasi-smooth stacks} \label{Sec:quasi-smooth}
    
    In this short section, we consider the derived stabilizer reduction of a quasi-smooth derived Artin stack $X$, satisfying the usual conditions of being quasi-compact, locally of finite presentation, and $X_\cl$ admitting a good moduli space $q \colon X_\cl \to M$.
    
    Recall that $X$ is quasi-smooth if its cotangent complex $\BL_X$ is perfect of (homological) Tor-amplitude $[-1,1]$. The main result of the section is as follows.
    
    \begin{theorem} \label{thm:stab red of quasi-smooth}
        The derived stabilizer reduction $\tilde{X}$ is a quasi-smooth Deligne--Mumford stack. In particular, the intrinsic stabilizer reduction $\tilde{X}_\cl$ of $X_\cl$ admits a natural perfect obstruction theory $\BL_{\tilde{X}}|_{\tilde{X}_\cl} \to \BL_{\tilde{X}_\cl}$ in the sense of Behrend--Fantechi \cite{BehFan} and induces virtual fundamental classes $[\tilde{X}_\cl]^\vir \in A_\ast (\tilde{X}_\cl)$ and $[\oO_{\tilde{X}_\cl}^\vir] \in K_0(\tilde{X}_\cl)$ in Chow and $K$-theory.
    \end{theorem}
    
    To prove the theorem, we make use of the following proposition.
    
    \begin{proposition} \label{prop:quasi-smooth blow-up}
        Let $f \colon Z \to X$ be a closed embedding of quasi-smooth derived schemes. Then $\Bl_Z X$ is a quasi-smooth derived scheme.
    \end{proposition}
    
    \begin{proof}
        The statement is local, so we may assume that $X = \Spec B$ and $Z = \Spec D$ are affine and $f$ is induced by a map of rings $B \to D$. Moreover, since $X$ and $Z$ are quasi-smooth, we may also assume that $B = A / (f_1, \ldots, f_n)$ and $D= A / (g_1, \ldots, g_m)$ where $A = \oO_V$ for a smooth, classical affine scheme $V$ and $f_i, g_j \in \pi_0 A$. It is then an immediate application of Proposition~\ref{Prop:RZU} that $\Bl_Z X$ is quasi-smooth as well.
    \end{proof}
    
    \begin{remark}
        The preceding proposition has also been used in recent work of Zhao \cite{Zhao1, Zhao2}.
    \end{remark}
    
    We may now deduce that quasi-smoothness is preserved by derived intrinsic blow-ups.
    
    \begin{proposition} \label{prop:intr blowup quasi-smooth}
        Suppose that $X$ is a quasi-smooth derived Artin stack. Then $X^{\max}$ and $X^{\intr} = \Bl_{X^{\max}} X$ are quasi-smooth.
    \end{proposition}
    
    \begin{proof}
        Both statements are \'{e}tale local, so by Lemma~\ref{Lem:Alg are standard form} and Proposition~\ref{Prop:luna_etale_slic} we may assume that $X = [U/G]$ where $U = \Spec R \in \Aff^G$, $R$ is a $G$-cdga in standard form and $U^{G^0}$ is non-empty. The fact that $X$ is quasi-smooth implies that $R$ is generated over $R_0$ on a finite number of generators in degree $1$, given by an $R_0$-module $M_1 = R_0 \otimes_{\CC} W_1$ for a $G$-representation $W_1 \in \Rep^G$. Proposition~\ref{Prop:3.27} then implies that $U^{G^0}$ is of the same form and hence quasi-smooth. In particular, $X^{\max} = [U^{G^0}/ G]$ is quasi-smooth. Finally, $\Bl_{U^{G^0}} U$ is quasi-smooth by Proposition~\ref{prop:quasi-smooth blow-up} and, applying Proposition~\ref{Prop:G_equiv_blow-up}, $\Bl_{X^{\max}} X = [ \Bl_{U^{G^0}} U / G]$ is quasi-smooth.
    \end{proof}
    
    \begin{proof}[Proof of Theorem~\ref{thm:stab red of quasi-smooth}]
        The derived Kirwan blow-up $\hat{X}$ is an open substack of the derived intrinsic blow-up $X^\intr$, so Proposition~\ref{prop:intr blowup quasi-smooth} implies that $\hat{X}$ is quasi-smooth. Since $\tilde{X}$ is obtained via a sequence of Kirwan blow-ups, it must be quasi-smooth.
    \end{proof}
    
    In Section~\ref{Sec: DT VW}, we will use this result to define virtual classes for moduli stacks of sheaves on surfaces.
    
    \begin{remark} \label{Rem:Khan comparison}
        Khan \cite{KhanVirtualofStacks} has defined virtual fundamental classes $[X]^\vir$ for any quasi-smooth derived Artin stack $X$. For a quasi-smooth Deligne--Mumford stack $X$, his construction agrees with the virtual fundamental classes in Chow and $K$-theory defined using the machinery of perfect obstruction theories. Writing $\pi \colon \tilde{X} \to X$ for the natural projection map, it is an interesting question to see how the virtual classes of the stabilizer reduction $\tilde{X}$ relate to those of $X$ via the map $\pi$, especially in light of the potential enumerative applications in Section~\ref{Sec: DT VW}.
    \end{remark}
	
	\section{\texorpdfstring{$(-1)$}{-1}-shifted symplectic stacks} \label{Sec:Shifted Symplectic}
	
	$(-1)$-shifted symplectic derived Artin stacks play a fundamental role in enumerative geometry and Donaldson--Thomas (DT) theory, since they naturally arise as derived moduli stacks of sheaves and perfect complexes on Calabi--Yau threefolds. In \cite{KLS, Sav}, it is shown that if $X$ is such a stack, then the intrinsic stabilizer reduction $\widetilde{X}_\cl$ admits a semi-perfect obstruction theory of virtual dimension zero and hence by \cite{LiChang} a virtual fundamental cycle $[\widetilde{X}_\cl^{\mathrm{\vir}}] \in A_0(\widetilde{X}_\cl)$, whose degree was defined to be the generalized DT invariant via Kirwan blow-ups associated to $X_\cl$. This was possible by rather cumbersome local computations.
	
	In this section, we pursue a more detailed study of the derived stabilizer reduction $\widetilde{X}$ to put this result on a much more robust footing. Namely, we will show that, by truncating the derived tangent complex of $\widetilde{X}$, there is a natural way to recover the above semi-perfect obstruction theory. As an application, we construct generalized Vafa--Witten invariants in Section~\ref{Sec: DT VW}.
	
	\subsection{Brief review of shifted symplectic geometry}
	Shifted symplectic structures were introduced by Pantev--To\"{e}n--Vaqui\'{e}--Vezzosi in \cite{PTVV}. Their definition is given in the affine case first, and then generalized by showing that the local notion satisfies smooth descent. We use the presentation from \cite{PantevVezzosiSymplPois}, but with different grading. See also \cite{RaksitHKR}.

	The local definition is as follows. Let $\CC[\epsilon]$ be the graded algebra, where $\epsilon$ is free in homogeneous degree 1 and homological degree 1, hence with $\epsilon^2 \simeq 0$. An object in $\Mod^{\ZZ}_{\CC[\epsilon]}$ is then a graded $\CC[\epsilon]$-module $\bigoplus_p M^p$, and multiplication by $\epsilon$ induces maps $M^{p} \to M^{p+1}[1]$. Then the functor $\Alg^{\ZZ}_{\CC[\epsilon]} \to \Alg$, which sends an algebra $\bigoplus_p B^p$ to $B^0$, has a left adjoint, written
	\[ \dR_{(-)} \colon \Alg \to  \Alg^\ZZ_{\CC[\epsilon]} \]
	The algebra $\dR_A$ for $A \in \Alg$ is called the \emph{derived de Rham} algebra. The underlying graded module of $\dR_A$ is 
	\[ \dR_A \simeq \bigoplus\nolimits_{p \geq 0} \Lambda^p_A (\BL_A) [p] \simeq \LSym_A(\BL_A[1](-1)) \]
	where $(-)(q)$ means twisting in homogeneous degree, so that $(M(q))^p = M^{p+q}$.
	
	The space of \emph{$n$-shifted, closed $p$-forms} on $A$ is 
	\[ \sA^{p,\cd}(A,n) \coloneqq  \Mod^{\ZZ}_{\CC[\epsilon]}(\CC(-p)[p-n], \dR_A) \] 
	An element $\omega \in \sA^{p,\cd}(A,n)$ thus induces an element in 
	\[ \dR_A^p[n-p] \simeq \Lambda^p \BL_A [n]  \]
	In particular, for $p=2$, we obtain a map $\omega \colon \CC \to  \BL_A \wedge \BL_A [n]$, hence a morphism $\omega^{\flat} \colon \BL_A^{\vee} \to \BL_A[n]$ by adjunction. We say that $\omega$ is \emph{non-degenerate} if $\omega^{\flat}$ is an equivalence. 
	
	\begin{definition}
		Suppose that $A \in \Alg^G$. Then the standard $G$-action on $\BL_A$ induces a $G$-action on $\dR_A$, and $\omega \in \dR_A$ is called \emph{$G$-invariant} if it is invariant with respect to this $G$-action.
	\end{definition}
	
	The functor $\sA^{p, \cd}(-,n)$ is a derived stack. For $X \in \Stk$, the space of $n$-shifted, closed $p$-forms on $X$ is then
	\[ \sA^{p,\cd}(X,n) \coloneqq \Stk(X,\sA^{p,\cd}(-,n)) \simeq \lim\nolimits_{\Spec A \to X} \sA^{p,\cd}(A,n) \] 
	An $n$-shifted, closed $p$-form $\omega$ on $X$ thus induces a homotopy coherent family of $\omega_A \in \sA^{p,\cd}(A,n)$, indexed over all $\Spec A \to X$. It can be shown that, for $p=2$, this again induces a morpshim $\omega^\flat \colon \BL^\vee_X \to \BL_X[n]$, and that $\omega$ is \emph{non-degenerate} if $\omega^\flat$ is an equivalence.

	\begin{definition}[{\cite[Def.~1.18]{PTVV}}, {\cite[Def.~1.10]{PantevVezzosiSymplPois}}]
		An $n$-shifted, closed $2$-form $\omega$ on $X$ is called an $n$-\textit{shifted symplectic structure} if it is non-degenerate. If such $\omega$ exists, we say that $X$ is $n$-shifted symplectic.
	\end{definition}

	\subsection{Local structure of $(-1)$-shifted symplectic stacks} We first introduce notation that will help us denote cdga's more efficiently. Recall that we use homological indexing notation.
	
	\begin{definition}
		Let $G$ be a reductive group, $V$ a smooth $G$-scheme, $\wW_{1}$ and $\wW_{2}$ be $G$-equivariant vector bundles on $V$ and $\delta_{2} \colon \wW_{2} \to \wW_{1}$, $\delta_{1} \colon \wW_{1} \to \oO_V$ be two $G$-equivariant morphisms such that $\delta_1 \delta_2 = 0$.
		
		Then we write $A = K(G, V, \wW_{1}, \wW_{2}, \delta_{1}, \delta_{2})$ for the $G$-equivariant sheaf of cdga's generated by $\oO_V, \wW_{1}$ and $\wW_{2}$ in degrees $0, 1$ and $2$ respectively and whose differential is determined by the maps $\delta_{1}, \delta_{2}$. 
		
		For brevity we write $\wW^{i}$ for $\wW_{i}^\vee$, where $i=1,2$, and $\delta^{2}$ for $\delta_{2}^\vee$.
		
		In the case where $\wW_{i} = W_{i} \otimes_{\CC} \oO_V$ for $G$-representations $W_{i}$, we also use the shorthand notation $K(G, V, W_{1}, W_{2}, \delta_{1}, \delta_{2})$.
	\end{definition}
	
	\begin{proposition}
		Let $A = K(G, V, \wW_{1}, \wW_{2}, \delta_{1}, \delta_{2})$ where $V$ is affine and write $U = \Spec A$. Then, after possibly (Zariski) equivariantly shrinking around fixed points $x \in U^G$, $A$ can be written in the form 
		$$A=K(G, V, W_{1}, W_{2}, \delta_{1}, \delta_{2})$$ for some $G$-representations $W_{1}, W_{2}$.
	\end{proposition}
	
	\begin{proof}
		Lemma~\ref{Lem:Equiv alg are standard form} applies verbatim to prove the statement.
	\end{proof}
	
	The following lemma is a strengthening of the preceding proposition in the $(-1)$-shifted symplectic case. Recall that a smooth scheme $V$ Zariski-locally admits an \'{e}tale morphism $V \to \AA^d$ for some $d$. The local sections $x_i \in \sO_V$ induced by the corresponding map $\sO_{\AA^d} \to \sO_V$ are called \emph{\'{e}tale coordinates}. For a $G$-action on $V$, we say that $G$ acts \emph{linearly} on given \'{e}tale coordinates, if there is a homomorphism $\varphi \colon G \to \GL_d$ such that $V \to \AA^d$ is $G$-equivariant for the $G$-action on $\AA^d$ induced by $\varphi$. 
	
	\begin{lemma} \label{equivariant darboux lemma}
		Let $U$ be an affine derived $G$-scheme, locally of finite presentation, such that $X = [U / G]$ is a $(-1)$-shifted symplectic derived quotient stack, and let  $x \in U^G$. Let $\fg$ be the Lie algebra of the reductive group $G$. Then after possible (Zariski) equivariant shrinking around $x$, we may assume that $\sO_U$ has a model $A=K (G, V, W_{1}, W_{2}, \delta_{1}, \delta_{2})$, where: 
		\begin{enumerate}
			\item $W_{1}$ is the vector space spanned by $\frac{\partial}{\partial x_i}$ for a set of \'{e}tale coordinates $\lbrace x_i \rbrace$ of $V$ on which $G$ acts linearly, so that $W_{1} \otimes_{\CC} \oO_V \cong T_V$ as $G$-equivariant vector bundles on $V$.
			\item Under this identification, $\delta_{1}$ maps each $\frac{\partial}{\partial x_i}$ to $\frac{\partial f}{\partial x_i} \in \oO_V$ for a $G$-invariant regular function $f \colon V \to \AA^1$.
			\item $W_{2} = \fg$ and the map $\delta_{2} \colon \fg \otimes_{\CC} \oO_V \to W_{1} \otimes_{\CC} \oO_V \cong T_V$ is the infinitesimal derivative of the $G$-action on $V$.
		\end{enumerate}
	\end{lemma}
	
	\begin{proof}
		By Lemma~\ref{Lem:Equiv alg are standard form}, up to equivariant shrinking around $x$, we may assume that $\sO_U$ has a $G$-equivariant model $A$ in standard form, i.e. generated by $\oO_V$ in degree $0$, where $V$ is a smooth affine $G$-scheme, and $G$-equivariant vector bundles $\wW_{i} = W_i \otimes_{\CC} \oO_V$ on $V$ in positive degrees, where as usual $W_i = \wW_{i}|_x$ is the restriction to $x \in V$. We then have
		\begin{align*}
			\BL_X|_x \simeq \BL_A|_x \cong \left[ \ldots \lr \wW_{2}|_x \lr \wW_{1}|_x \lr \Omega_V |_x \lr \fg^\vee \otimes_{\CC} \oO_V \right]
		\end{align*}
		by Example \ref{Ex:LXcanGact} and the general formula for the cotangent complex of a cdga in standard form (we conflate $\BL_A$ with the explicit model from \cite[\S 2.3]{JoyceSch}).
		
		Moreover, we may assume that $\Spec A$ is minimal at $x$, meaning that the differentials of $\BL_A|_x$ vanish. Since $X$ is $(-1)$-shifted symplectic, there exists a quasi-isomorphism $\BL_A|_x \cong \BL_A|_x^\vee [1]$ and then minimality implies that we must have $\wW_{i} = 0$ for $i>2$. Moreover, we get $G$-equivariant isomorphisms $W_{1} \cong T_V|_x, W_{2} \cong \fg$, and we can pick a set of \'{e}tale coordinates on $V$ with a linear $G$-action using a basis for $\Omega_V|_x$ (as $G$-module) and lifting through $\mathrm{d} \colon \oO_V \to \Omega_V$.
		
		Consider now the smooth morphism $\Spec A \to [\Spec A / G]$. The pullback to $\Spec A$ of the shifted symplectic structure $\omega$ on $X$ is a $G$-invariant $(-1)$-shifted closed $2$-form. Since $G$ is reductive, the arguments of \cite[Subsection~5.2]{JoyceSch} and the proof of \cite[Theorem~5.18]{JoyceSch} extend verbatim in this $G$-equivariant setting to show that the differential $\delta_{1}$ maps $\frac{\partial}{\partial x_i}$ to $\frac{\partial f}{\partial x_i} \in \oO_V$ for a $G$-invariant regular function $f \in \oO_V^G$.
		
		Finally, the pullback of the quasi-isomorphism $\BL_{[A/G]} \cong \BL_{[A/G]}^\vee [1]$ to $\Spec A$ and then to $V$ (where we consider the natural morphism of graded algebras $A \to \oO_V$) gives rise to a $G$-equivariant morphism of complexes
		\begin{align*}
			\xymatrix@C=3.3em{
				\fg\otimes_{\CC} \oO_V \ar[r]^-{d\delta_{2}} \ar[d] & T_V \ar[r]^-{d\delta_{1}} \ar[d]^-{\id} & \Omega_V \ar[r]^-{\sigma} \ar[d]^-{\id} & \fg^\vee \otimes_{\CC} \oO_V \ar[d] \\
				\fg \otimes_{\CC} \oO_V \ar[r]^-{\sigma^\vee} & T_V \ar[r]^-{(d\delta_{1})^\vee} & \Omega_V \ar[r]^-{(d\delta_{2})^\vee} & \fg^\vee \otimes_{\CC} \oO_V,
			}
		\end{align*}	 
		where the two middle vertical arrows are the identity maps and, by symmetry, $(d\delta_{1})^\vee = d\delta_{1}$.
		
		The differentials of both complexes vanish at $x$ and hence it follows that the two outer vertical arrows must be isomorphisms after possible shrinking around a $G$-invariant Zariski neighbourhood of $x$. Thus we may identify $d \delta_{2}$ with the infinitesimal derivative $\sigma^\vee$ of the $G$-action. By comparing degrees, $d \delta_{2}$ determines the morphism $\delta_{2} \colon \fg \otimes_{\CC} \oO_V \to T_V$, which thus must equal $\sigma^\vee$, so we have established (i)--(iii).
	\end{proof}
	
	\begin{remark}
		In fact, we can prove a slightly stronger statement giving an equivariant version of the derived Darboux theorem of \cite{JoyceArt}. Namely, using the description of shifted symplectic structures on quotient stacks $[\Spec A / G]$ \cite{Yeung}, the $(-1)$-shifted symplectic form on $X$ can be brought to a canonical Darboux form.
	\end{remark}
	
	\begin{remark}
		We will actually only need conditions (i) and (iii) in the ensuing computations in this paper. 
	\end{remark}
	
	\subsection{An intrinsic blow-up computation} The intrinsic blow-up of a $(-1)$-shifted symplectic derived stack will in general not be $(-1)$-shifted symplectic. An explicit example can be constructed using the affine scheme $W$ in Example~\ref{Ex:Secondocouple}. Thus, in order to work with the tangent complex of the derived stabilizer reduction of a $(-1)$-shifted symplectic stack, we will need to get a better handle on the local structure of its iterated intrinsic blow-ups. To this end, we introduce further terminology.
	
	\begin{definition}
		Let $A = K(G, V, \wW_{1}, \wW_{2}, \delta_{1}, \delta_{2})$. 
		
		We say that $A$ satisfies property ($\dagger$) if
		there exists a $G$-invariant Cartier divisor $D$ on $V$ and a commutative diagram of $G$-equivariant morphisms
		\begin{align}
			\xymatrix{
				\Omega_V(-D) \ar[r] \ar[dr] & \wW^{1} \ar[d] \ar[dr]^-{\delta^{2}} & \\
				& \fg^\vee(-D) \ar[r] \ar[dr] & \wW^{2} \ar[d] \\
				& & \fg^\vee\otimes_{\CC} \oO_V
			}
		\end{align}
		where $\Omega_V(-D) \to \fg^\vee(-D)$ is induced by the derivative of the $G$-action, $\fg^\vee(-D) \to \fg^\vee\otimes_{\CC} \oO_V$ is the natural inclusion, $\wW^{2} \to \fg^\vee\otimes_{\CC} \oO_V$ is injective, and for any point in $V$ with closed $G$-orbit and reductive stabilizer $H \subseteq G$, the composition $$\wW^{1}|_{V^H} \to \fg^\vee(-D) \to \fh^\vee (-D)$$ vanishes.
		
	\end{definition}

	Let $X = [U/G]$, where $U$ is an affine derived scheme. If $U$ has a model $A$ in $G$-equivariant standard form, we write $\mathbb{T}_X|_{X_\cl}$ for the complex on the (classical) derived category of $[\Spec \pi_0 A/G]$ induced by pulling back $\BL_{[\Spec A/G]}^\vee$ to $X_\cl$ (where $\BL_{[\Spec A/G]}^\vee$ again refers to the explicit model). Furthermore, we write $\tau^{[0,1]} (-)$ for the truncation in homological degrees $[-1,0]$, which corresponds to the $[0,1]$-truncation in cohomological indexing notation. We do this to better match with existing notation on perfect obstruction theories. In particular, $(\tau^{[0,1]} E)^\vee$ lives in cohomological degree $[-1,0]$, when $\tau^{[0,1]} E$ is perfect.

	\begin{proposition} \label{truncation is perfect prop}
		Suppose that $A$ satisfies property $(\dagger)$ and $X = [\Spec A / G]$ is Deligne--Mumford. Then the truncation $E^\bullet = \tau^{[0,1]} \mathbb{T}_X|_{X_\cl}$ is a perfect two-term complex, whose dual $E_\bullet \coloneqq (E^\bullet)^\vee$ defines a perfect obstruction theory\footnote{For us, a perfect obstruction theory means either a morphism $E_\bullet \to (\BL_{X_\cl})_{\leq 1}$ or, as is more standard, a morphism $E_\bullet \to \BL_{X_\cl}$, which is an isomorphism on $h^0$ and surjective on $h^{1}$ (in homological grading). From now on, we work with the former, weaker version.} on $X_\cl$.
	\end{proposition}
	
	\begin{proof}
		By assumption, the injection $\fg^\vee (-D) \to \fg^\vee\otimes_{\CC} \oO_V$ factors through the injection $\wW^{2} \to \fg^\vee\otimes_{\CC} \oO_V$ and hence $\fg^\vee(-D) \to \wW^{2}$ is injective. In particular, it follows that the kernel of $\delta^{2}$ is equal to the kernel of the morphism $\wW^{1} \to \fg^\vee (-D)$. Since $X$ is Deligne--Mumford, the dual $\Omega_V (-D) \to \fg^\vee(-D)$ of the derivative is surjective, and thus so is $\wW^{1} \to \fg^\vee (-D)$. Therefore, the kernel $\fF \coloneqq \ker \delta^{2}$ is a vector bundle. 
		
		We then have (suppressing the restriction to $X_\cl$ on the right-hand side)
		$$\mathbb{T}_X|_{X_\cl} \cong [\fg\otimes_{\CC} \oO_V \lr T_V \xrightarrow{(d\delta^{1})^\vee} \wW^{1} \xrightarrow{\delta^{2}} \wW^{2}]$$
		in homological degrees $[-2,1]$ and, writing $T_{V/G}$ for the cokernel of the derivative $\fg\otimes_{\CC} \oO_V \to T_V$,
		$$E^\bullet = \tau^{[0,1]} \mathbb{T}_X|_{X_\cl} \cong [ T_{V/G} \lr \fF ], $$
		a two-term complex of vector bundles. 
		
		To see that $E_\bullet = (E^\bullet)^\vee$ defines a perfect obstruction theory, let $I =\im(\delta_{1})$ be the ideal sheaf of $X_\cl$ inside $V$. Since $\delta^{1}$ factors through $\fF$, we have the  following commutative diagram
		\begin{align*}
			\xymatrix{
				\wW_{2} \ar[r]^-{\delta_{2}} & \wW_{1} \ar[r]^-{d\delta_{1}} \ar[d] & \Omega_V \ar[r] \ar@{=}[d] & \fg^\vee\otimes_{\CC} \oO_V \ar@{=}[d]   \\
				&  \fF^\vee \ar[r] \ar[d] & \Omega_V \ar[r] \ar@{=}[d] & \fg^\vee\otimes_{\CC} \oO_V \ar@{=}[d]  \\
				&  I/I^2 \ar[r]^{d} &  \Omega_V \ar[r] & \fg^\vee\otimes_{\CC} \oO_V,
			}
		\end{align*}
		where both left-most vertical arrows are surjective, and thus a morphism
		\[ \BL_X|_{X_\cl} \to (\BL_{X_\cl})_{\leq 1} \]
		since $\BL_X|_{X_\cl}$ is equivalent to the top row, and $(\BL_{X_\cl})_{\leq 1}$ to the bottom row in the diagram.
		
		It follows that $E_\bullet = [\fF^\vee \to \Omega_{V/G}]$ gives a perfect obstruction theory on $X_\cl$, where $\Omega_{V/G} = T_{V/G}^\vee = \ker(\Omega_V \to \fg^\vee\otimes_{\CC} \oO_V)$.
	\end{proof}
	
	Let now $X = \Spec A$, where $A = K(G, V, W_{1}, W_{2}, \delta_{1}, \delta_{2})$ satisfies property ($\dagger$), and let $x \in X$ be a point with closed $G$-orbit and (reductive) stabilizer $H$ of maximal dimension. 
	
	Using Luna's \'{e}tale slice theorem \cite{Drezet}, we take a slice $x \in S \subseteq V$ for the $G$-action on $V$ at $x$ giving rise to an affine, strongly \'{e}tale morphism $\phi \colon S \times_H G \to V$.
	
	The isomorphism $[S/H] \cong [S \times_H G/G]$ then gives an equivalence between $H$-equivariant sheaves on $S$ and $G$-equivariant sheaves on $S \times_H G$. We will freely use this identification henceforth.
	
	The first part of the following proposition follows by definition, and the second part from Proposition \ref{Prop:3.27}.
	
	\begin{proposition}
		Let $B = K(G, S \times_H G, W_{1}, W_{2}, \phi^\ast \delta_{1}, \phi^\ast \delta_{2})$ and set $U = \Spec B$. Then $B$ also satisfies property ($\dagger$) and the natural morphism $U \to X$ is $G$-equivariant with $[U/G] \to [X/G]$ strongly \'{e}tale.
		
		Moreover, if we write $W_{i}^{\fix, H}$ for the fixed part of $W_{i}$ considered as an $H$-representation, $\wW_{i}^{\fix, H}$ for the $G$-equivariant summand of $\wW_{i}$ corresponding to $W_{i}^{\fix, H}\otimes_{\CC} \oO_S$ under the isomorphism $[S / H] \cong [S \times_H G / G]$, and let 
		$$C = K \left( G, S^H \times_H G, \wW_{1}^{\fix, H}, \wW_{2}^{\fix, H}, \delta_{1}|_{S^H \times_H G}, \delta_{2}|_{S^H \times_H G} \right),$$
		then $[\Spec C / G]$ is $[U/G]^{\maxlocus}$.
	\end{proposition}
	
	We are now able to explicitly compute the blow-up of $U$ along $U^G = \Spec C$, after setting up some more notation.
	
	Without loss of generality, we may assume that $x_1, \ldots, x_l$ are the $H$-moving coordinates of $S$, which lift to minimal generators of the ideal of $S^H \times_H G$ in $S \times_H G$ and on which the $H$-action is linear.
	
	Pick bases $\lbrace \underline{w}_{1}^f \rbrace$ for $W_{1}^{\fix, H}$ and $\lbrace \underline{w}_{1}^m \rbrace$ for $W_{1}^{\mv, H}$. Denote their union by $\lbrace \underline{w}_{1} \rbrace$. 
	
	By $H$-equivariance, the images of $\lbrace \underline{w}_{1}^m \rbrace$ in $S \times_H G$ under $\delta_{1}$ must land in the ideal $(x_1, \ldots, x_l)$. Write $s_i^{\mv} = \sum_k \alpha_{ik} x_k$ for the image of the $i$-th basis element $w_{1,i}^m$ and $s_i^{\fix} \in \oO_V$ for the image of the $i$-th basis element $w_{1,i}^f$.
	
	Since $\delta_{2}|_{V^H}$ vanishes (using property ($\dagger$)), picking a basis for $(W_{2})^{\mv, H}$, the images of its elements in $\wW_{1}$ must land inside $(x_1, \ldots, x_l)\wW_{1}$ by $G$-equivariance. Write $p_j^{\mv} \coloneqq \sum_{k, \ell} \beta_{jk\ell} x_\ell \cdot w_{1,k} $ for the image of the $j$-th basis element. Picking a basis for $(W_{2})^{\fix, H}$, write $p_j^{\fix}$ for the image of the $j$-th basis element $w_{2,j}^\fix$.
	
	\begin{lemma}
		Let $U=\Spec B$ and $U^G = \Spec C$ be as in the previous proposition. Then 
		\begin{align} \label{rees of slice}
			R_{U^G/U} \simeq \frac{B[t^{-1}, \underline{v}]}{(\underline{v}t^{-1} - \underline{x}, \lbrace \sum_k \alpha_{ik} v_k \rbrace, \lbrace \sum_{k, \ell} \beta_{jk\ell} v_\ell \cdot w_{1,k} \rbrace )}.
		\end{align}
	\end{lemma}
	
	\begin{proof}
		For brevity, denote (by slight abuse of notation) $V = S \times_H G$. Define:
		\begin{enumerate}
			\item $Q$ to be the derived affine $G$-scheme with 
			\begin{align} \label{loc Q}
				\oO_Q = \oO_U / (\underline{x}).
			\end{align}
			\item $Y$ to be the derived affine $G$-scheme with 
			\begin{align} \label{loc Y}
				\oO_Y = \oO_V / ( \underline{s}^f )
			\end{align}
			obtained by the $G$-equivariant cell attachment $W_{1}^{\fix,H}[0] \to \oO_V$ in degree $0$.
			Observe that we can also write 
			\begin{align} \label{loc Q 2}
				\oO_Q = \oO_Y / ( \underline{x}, \underline{s}^{\mv}, \underline{0}),
			\end{align}
			with the appropriate $G$-equivariant cell attachments corresponding to $\underline{x}, \underline{s}^{\mv}$ understood and the cell attachment $W_{2}[1] \to \oO_Y$ given by the zero morphism. This is possible due to the fact that the natural morphism $W_{2}[1] \to \oO_U$ restricts to zero after quotienting out by the ideal $(\underline{x})$ by assumption.
		\end{enumerate}
		
		By definition, we similarly have that 
		\begin{align} \label{loc Z}
			\oO_{U^G} = \oO_Y / ( \underline{x},  \underline{0}^{\fix}),
		\end{align}
		where the cell attachment for $\underline{x}$ is as above and we use the zero cell attachment $W_{2}^{\fix, H} [1] \to \oO_Y$ for $\underline{0}^{\fix}$.
		
		We have an induced commutative square of closed embeddings
		\begin{align*}
			\xymatrix{
				Z = U^G \ar[r] \ar[d] & U \ar[d] \\
				Q \ar[ur] \ar[r] & Y
			}
		\end{align*}
		and hence by Lemma~\ref{functoriality of deformation spaces lemma} it follows that
		$$R_{U^G / U} = R_{Q/U} \otimes_{R_{Q/Y}} R_{Z / Y}. $$
		
		Using~\eqref{loc Q}, \eqref{loc Y}, \eqref{loc Q 2}, \eqref{loc Z} and the ($G$-equivariant) finite quotient formula of Proposition~\ref{Prop:equivariant_quotient_formula}, we obtain
		\begin{align*}
			R_{Q/U} & = \frac{\oO_U[t^{-1}, \underline{v}]}{(\underline{v} t^{-1} - \underline{x})} \\
			R_{Z/Y} & = \frac{\oO_Y[t^{-1}, \underline{v}, \underline{u}^\fix]}{(\underline{v} t^{-1} - \underline{x},\ \underline{u}^\fix t^{-1})} \\
			R_{Q/Y} & = \frac{\oO_Y[t^{-1}, \underline{v},\ \underline{e}^{\mv}, \underline{u}]}{(\underline{v} t^{-1} - \underline{x},\ \underline{e}^{\mv}t^{-1} - \underline{s}^{\mv},\ \underline{u} t^{-1}) }.
		\end{align*}
		
		The morphism $R_{Q/Y} \to R_{Z/Y}$ maps $\underline{e}^{\mv}$ and $\underline{u}^{\mv}$ to zero, while the morphism $R_{Q/Y} \to R_{Q/U}$ maps $\underline{e}_i^{\mv}$ to $\sum_k \alpha_{ik} v_k \in R_{Q/U}$ and $u_j^{\mv}$ to $\sum_{k, \ell} \beta_{jk\ell} v_\ell \cdot w_{1,k} \in \sR_{Q/U}$.
		
		Thus, the homotopy pushout diagram
		\begin{align*}
			\xymatrix{
				\ZZ[\underline{e}^{\mv}, \underline{u}^{\mv}] \ar[r] \ar[d] & \ZZ \ar[d] \\
				R_{Q/Y} \ar[r] \ar[d] & R_{Z/Y} \ar[d] \\
				R_{Q/U} \ar[r] & R_{U^G / U}
			}
		\end{align*}
		immediately gives~\eqref{rees of slice}.
	\end{proof}
	
	We obtain the following immediate corollary.
	
	\begin{corollary} \label{Cor:6.10}
		In the above notation, write $Y = S \times_H G$ and $GY^H = S^H \times_H G$. Then $Y^{\intr} = \Bl_{GY^H} Y$ and $\pi \colon \hat{Y} = (\Bl_{GY^H} Y)^{\mathrm{ss}} \to Y$ with exceptional divisor $E$. 
		
		Consider the sheaf of cdga's $\hat{B}$ generated by $\oO_{\hat{Y}}$ in degree $0$, $\pi^\ast \wW_{1}^{\fix, H} \oplus \pi^\ast \wW_{1}^{\mv, H} (E)$ in degree $1$ and $\pi^\ast \wW_{2}^{\fix, H} \oplus \pi^\ast \wW_{2}^{\mv, H} (2E)$ in degree $2$ with differentials induced by~\eqref{rees of slice}.
		
		Then $\hat{B}$ satisfies property $(\dagger)$ and the derived Kirwan blow-up $\hat{U}$ is given by $\Spec \hat{B}$.
	\end{corollary}
	
	\begin{proof}
		We only need to check that property ($\dagger$) is satisfied. But this follows verbatim from the argument used in the proof of \cite[Lemma~5.3]{KLS}.
	\end{proof}

	\subsection{Perfectness of the truncated tangent complex, virtual fundamental cycle and virtual structure sheaf} Let now $X$ be a derived Artin stack satisfying our usual assumptions, which include the existence of a good moduli space for the classical truncation $X_\cl$.
	
	\begin{lemma} \label{propagation of property B}
		Suppose that $X$ admits a strongly \'{e}tale cover by quotient stacks $[\Spec A / G]$ where $A$ satisfies property $(\dagger)$. Then the same holds for its Kirwan blow-up $\hat{X}$.
	\end{lemma}
	
	\begin{proof}
		This is immediate from Corollary~\ref{Cor:6.10}.
	\end{proof}
	
	\begin{corollary} \label{cor 6.15}
		The derived stabilizer reduction $\widetilde{X}$ of a $(-1)$-shifted symplectic derived Artin stack $X$ admits a strongly \'{e}tale cover by quotient stacks $[\Spec A / G]$ where $A$ satisfies property $(\dagger)$.
	\end{corollary}
	
	\begin{proof}
		By Proposition~\ref{Prop:luna_etale_slic}, we know that $X$ admits a strongly \'{e}tale cover by affine quotient stacks $[\Spec A / G]$. Using Lemma~\ref{equivariant darboux lemma}, we deduce that after possible equivariant shrinking, $A$ may be taken to satisfy property $(\dagger)$. Using the Fundamental Lemma~\cite[Theorem~6.10]{AlperFundLemma}, further shrinking ensures that the morphism $[\Spec A / G] \to X$ remains strongly \'{e}tale.
		
		Since $\widetilde{X}$ is obtained by $X$ through a sequence of Kirwan blow-ups, the claim follows by Lemma~\ref{propagation of property B} and induction.
	\end{proof}
	
	This corollary together with Proposition~\ref{truncation is perfect prop} imply the main result of this section. To match up with cohomological indexing, we write $(-)^{\geq -n}$ for the truncation $\tau_{\leq n}(-)$. Likewise, we write $h^k(-) \coloneqq \pi_{-k}(-)$.

	\begin{theorem} \label{thm trunc perf}
		Let $X$ be a derived $(-1)$-shifted symplectic Artin stack of finite type and $\widetilde{X}$ its derived stabilizer reduction. Then the truncation $E^\bullet = \tau^{[0,1]} \mathbb{T}_{\widetilde{X}} |_{\widetilde{X}_\cl}$ is a perfect complex of virtual dimension $0$.
		
		Moreover, let $\lbrace U_\alpha \to \widetilde{X} \rbrace$ be a strongly \'{e}tale cover by quotient stacks $[\Spec A / G]$ where $A$ satisfies property $(\dagger)$. Write $E_\bullet \coloneqq (E^\bullet)^\vee$.
		
		Then the complexes $E_\bullet|_{(U_\alpha)_\cl}$ are part of the data of a semi-perfect obstruction theory on $\widetilde{X}_\cl$, which recovers the semi-perfect obstruction theory on $\widetilde{X}_\cl$ constructed in \cite{Sav}, and also of an almost perfect obstruction theory which recovers the one constructed in \cite{KiemSavvas}.
		
		In particular, $E_\bullet$ together with $\tilde{X}$ recover the $0$-dimensional virtual fundamental cycle $[\widetilde{X}_\cl ]^\vir \in A_0 (\widetilde{X}_\cl ) $ and virtual structure sheaf $[\oO_{\widetilde{X}_\cl}^\vir] \in K_0(\widetilde{X}_\cl)$ of the intrinsic stabilizer reduction $\widetilde{X}_\cl$, constructed in \cite{Sav, KiemSavvas}.
	\end{theorem}

	\begin{proof}
		The statement on the perfectness of $E_\bullet$ follows immediately by Proposition~\ref{truncation is perfect prop} and the preceding corollary. The virtual dimension then follows from Lemma \ref{equivariant darboux lemma}.
		
		Now, by Proposition~\ref{truncation is perfect prop}, the complexes $E_\bullet|_{(U_\alpha)_\cl}$ define perfect obstruction theories $\phi_\alpha \colon E_\bullet|_{(U_\alpha)_\cl} \to \BL_{(U_\alpha)_\cl}^{\geq -1}$ on $(U_\alpha)_\cl$. We may assume that the $U_\alpha$ are affine schemes after \'{e}tale base change, since $U_\alpha$ are Deligne--Mumford.
		\smallskip
		
		Write $\mathbb{T}_{\tilde{X}_\cl}^{\leq 1} =( \BL_{\tilde{X}_\cl}^{\geq -1} )^{\vee}$. Then, by construction of the $\phi_\alpha$, we have commutative diagrams 
		\begin{align*}
			\xymatrix@C=1.2cm{
				\mathbb{T}_{\tilde{X}_\cl}^{\leq 1}|_{(U_\alpha)_\cl} \ar[r]^-{\phi_\alpha^\vee} \ar[dr]_-{\rho|_{(U_\alpha)_\cl}} & E^\bullet|_{(U_\alpha)_\cl} \ar[d]^-{\psi|_{(U_\alpha)_\cl}} \\
				& \mathbb{T}_{\tilde{X}}|_{(U_\alpha)_\cl},
			}
		\end{align*}
		where $\psi \colon E^\bullet \to \mathbb{T}_{\widetilde{X}} |_{\widetilde{X}_\cl}$ is the natural morphism, and $\rho$ is the dual of the composition $\BL_{\tilde{X}}|_{\tilde{X}_\cl} \to \BL_{\tilde{X}_\cl} \to \BL_{\tilde{X}_\cl}^{\geq -1}$. 
		
		In particular, restricting to any closed point $p \in (U_\alpha)_\cl$ and taking $h^1(-)$, we obtain commutative diagrams
		\begin{align} \label{loc 8.8}
			\xymatrix@C=1.2cm{
				h^1(\mathbb{T}_{\tilde{X}_\cl}^{\leq 1}|_p) \ar[r]^-{h^1(\phi_\alpha^\vee|_p)} \ar[dr]_-{h^1(\rho|_p)} & h^1(E^\bullet|_p) \ar[d]^-{h^1(\psi|_p)} \\
				& h^1(\mathbb{T}_{\tilde{X}}|_p),
			}
		\end{align}
		
		We have a global obstruction sheaf defined by $\Ob_{X_\cl} \coloneqq h^1(E^\bullet)$, which glues together the local obstruction sheaves $\Ob_\alpha = h^1(E^\bullet|_{U_\alpha})$ and the transition functions are the same as the ones used in \cite{Sav}, by construction.
		
		It is now a simple verification to check that the local obstruction theories satisfy the necessary compatibility condition to define a semi-perfect obstruction theory as in \cite{LiChang} and that we hence get the same virtual cycle: Let $(U_{\alpha \beta})_\cl = (U_\alpha)_\cl \times_{\tilde{X}_\cl} (U_\beta)_\cl$. In the terminology of \cite[Definition~4.2]{KLS}, consider
		\begin{align*}
			\xymatrix{
				\Delta \ar[r]^-g \ar[d]^-{\iota} & (U_{\alpha \beta})_\cl \\
				\bar{\Delta} \ar@{-->}[ur]_-{\bar{g}}
			}
		\end{align*}
		an infinitesimal lifting problem at a closed point $p \in (U_{\alpha \beta})_\cl$, where $\bar{\Delta}$ is an extension of $\Delta$ by an ideal $I$. By \cite[Lemma~2.6]{LiChang}, this induces a canonical obstruction element
		\begin{align*}
			\omega(g, \Delta, \bar{\Delta}) \in \mathrm{Ext}^1 (g^\ast \BL_{\tilde{X}_\cl}, I) = h^1(\mathbb{T}_{\tilde{X}_\cl}^{\leq 1}|_p) \otimes_{\CC} I. 
		\end{align*}
		We need to show that
		\begin{multline*}
			h^1(\phi_\alpha^\vee|_p) (\omega(g, \Delta, \bar{\Delta}) )  = h^1(\phi_\beta^\vee|_p) (\omega(g, \Delta, \bar{\Delta}))\\ \in \mathrm{Ext}^1 (g^\ast E_\bullet, I) = h^1(E^\bullet|_p) \otimes_{\CC} I.
		\end{multline*}
		
		But this follows immediately from diagram~\eqref{loc 8.8}, since $h^1(\psi|_p)$ is an injection by the definition of $E_\bullet$ and Proposition~\ref{truncation is perfect prop} and thus $h^1(\phi_\alpha^\vee|_p) = h^1(\phi_\beta^\vee|_p)$.
		
		The claim for the almost perfect obstruction theory is similar and follows the reasoning of \cite[Subsection~5.4]{KiemSavvas} identically. We leave the details to the reader.
	\end{proof}
	
	\section{Sheaf-theoretic invariants} \label{Sec: DT VW}
	
	In this section, we explore immediate applications of our results to the enumerative geometry of sheaves and complexes and Donaldson--Thomas theory in particular.
	
	Namely, we define virtual classes for moduli stacks of semi-stable sheaves and complexes on smooth, projective surfaces.
	
	Moreover, we obtain a fully derived upgrade of the construction of generalized Donaldson--Thomas invariants via Kirwan blow-ups developed in \cite{KLS, Sav}. These act as virtual counts of semi-stable sheaves and perfect complexes on smooth, projective Calabi--Yau threefolds.
	
	The robustness of derived geometry also allows us to apply the same method to define new generalized Vafa--Witten invariants, enumerating semi-stable Higgs pairs on projective surfaces.

	\subsection{Donaldson-type invariants of surfaces} Let $S$ be a smooth, projective surface and $\mM$ a derived moduli stack that parametrizes semi-stable sheaves or perfect complexes on $S$. Possible choices of stability include Gieseker or slope stability \cite{HuyLehn} for sheaves and Bridgeland stability \cite{BridgStab} for complexes. It is a standard fact that $\mM$ is then quasi-smooth.
	
	If the genus $p_g = H^0(K_S)$ of $S$ is positive or the rank of the objects in question is zero, it might be necessary to consider a modification of $\mM$, for example, a ``reduced'' projective linear version as in \cite{JoyceWC}, or fix the determinant, in order to remove certain trivial components from the cotangent complex of $\mM$. Without going into more detail, we assume that this is the case so that $\mM$ is a derived moduli stack of semi-stable sheaves or perfect complexes on a smooth, projective surface, which is of finite type with the good moduli space $M_\cl$ of $\mM_\cl$ being proper.
	
	By Theorem~\ref{thm:stab red of quasi-smooth}, the derived stabilizer reduction $\tilde{\mM}$ is a proper quasi-smooth derived Deligne--Mumford stack and we have virtual fundamental classes 
	$$[\tilde{\mM}_\cl]^\vir \in A_\ast (\tilde{\mM}_\cl),\ [\oO_{\tilde{\mM}_\cl}^\vir] \in K_0(\tilde{\mM}_\cl).$$
	
	\begin{remark}
	    Following up on Remark~\ref{Rem:Khan comparison}, it would be interesting to compare these classes with the virtual classes of $\mM$ itself, constructed by Khan \cite{KhanVirtualofStacks}, and also with the homological virtual classes constructed in \cite{JoyceWC}.
	\end{remark}
	
	We may thus define intersection-theoretic and $K$-theoretic Donaldson-type invariants as follows.
	
	\begin{definition} \label{def:Donaldson invariants}
	    For any $\gamma \in H^*(\tilde{\mM}_\cl, \QQ)$, the intersection-theoretic generalized Donaldson-type invariant associated to $\mM$ is defined by the formula
	    $$\mathrm{DT}(\mM, \gamma) = \int_{[\tilde{\mM}_\cl]^\vir} \gamma.$$
	    Similarly, for any $\beta \in K^0(\tilde{\mM}_\cl)$, the $\beta$-twisted $K$-theoretic generalized Donaldson-type invariant associated to $\mM$ is defined by the formula
	    $$\mathrm{DT}^{\mathrm{K-th},\beta}(\mM) = \chi \left( \tilde{\mM}_\cl, [\oO_{\tilde{\mM}_\cl}^\vir] \otimes_{\oO_{\tilde{\mM}_\cl}} \beta \right) .$$
	\end{definition}
	
	\subsection{Donaldson--Thomas invariants of Calabi--Yau threefolds via Kirwan blow-ups} 
	
	Let $\mM$ be a derived moduli stack parametrizing semi-stable sheaves or perfect complexes on a smooth, projective Calabi--Yau threefold $W$, as considered in \cite{Sav}. Possible choices of stability here are Gieseker stability \cite{HuyLehn} for sheaves and PT, polynomial or Bridgeland stability \cite{BridgStab, Bayer, PT1} for complexes.
	
	As in \cite{Sav}, $\mM$ is $(-1)$-shifted symplectic and its classical truncation $\mM_\cl$ admits a good moduli space $\mM_\cl \to M_\cl$ with $M_\cl$ proper.
	
	Let $\nN_\cl \coloneqq \mM_\cl \!\!\fatslash \Gm$ be the rigidification of $\mM_\cl$ by the scaling automorphisms of sheaves or complexes. In \cite{Sav}, it is shown that the intrinsic stabilizer reduction $\widetilde{\nN}_\cl$ of $\nN$ admits a natural semi-perfect obstruction theory of virtual dimension zero. The degree of the induced virtual fundamental cycle $[\widetilde{\nN}_\cl]^{\mathrm{vir}}$ is defined to be the generalized Donaldson--Thomas invariant via Kirwan blow-ups associated to $\mM$. By \cite{KiemSavvas}, the obstruction theory can be upgraded to an almost perfect obstruction theory with induced virtual structure sheaf $[\oO_{\widetilde{\nN}_\cl}^\vir] \in K_0(\widetilde{\nN}_\cl)$.
	
	Our results above allow us to obtain a canonical derived enhancement as follows.
	
	\begin{theorem}
		There exists a derived Artin stack $\nN$ with classical truncation $\nN_\cl$. The truncation $E^\bullet = \tau^{[0,1]} \mathbb{T}_{\widetilde{\nN}}|_{\widetilde{\nN}_\cl}$ of the tangent complex is perfect, with induced virtual fundamental cycle (using Theorem~\ref{thm trunc perf}) equal to $[\widetilde{\nN}_\cl]^{\mathrm{vir}} \in A_0(\widetilde{\nN}_\cl)$ and induced virtual structure sheaf equal to $[\oO_{\widetilde{\nN}_\cl}^\vir] \in K_0(\widetilde{\nN}_\cl)$.
	\end{theorem}
	
	\begin{proof}
		The existence of $\nN$ follows from \cite[Lem.~4.4.9]{HL}, which generalizes the rigidification construction to derived Artin stacks. 
		
		Now, observe that $\mM$ admits a strongly \'{e}tale cover by quotient stacks $[\Spec A / G]$ where $A$ satisfies property $(\dagger)$. Since $\nN = \mM \!\!\fatslash \Gm$, it follows that $\nN$ admits a strongly \'{e}tale cover by the quotient stacks $[\Spec A / (G / \Gm)]$.
		
		It is then clear that the arguments of Section~\ref{Sec:der stab red} extend verbatim to show that Kirwan blow-ups preserve the form of these \'{e}tale covers and, in addition, Deligne--Mumford stacks with such covers have perfect two-term truncation of their tangent complex in amplitude $[0,1]$ (with cohomological grading), as desired.
		
		The equality of virtual fundamental cycles and virtual structure sheaves follows identical arguments as for Theorem~\ref{thm trunc perf}.
	\end{proof}
	
	\begin{remark}
	    With minor modifications, the above also apply if one fixes the determinant of the sheaves or complexes under consideration.
	\end{remark}
	
	\subsection{Stacks with a torus action} Let $X \in \Stk^T$ be a derived Artin stack with a $T$-action, for a torus $T = \Gm \times \dots \times \Gm$.
	
	\begin{proposition} \label{prop 7.5}
		Let $X \in \Stk^T$ be a derived stack with a $T$-action. Then $X^{\maxlocus}$ admits a natural $T$-action and is a $T$-invariant closed substack of $X$.
	\end{proposition}
	
	\begin{proof}
		Observe that the action map $X \times T \to X$ is affine, since $T$ is affine. The claim then follows from functoriality of $(-)^{d}$ for representable and separated morphisms by Proposition~\ref{Prop:prop_xmax}(ii) and $(X \times T)^{\maxlocus} = X^{\maxlocus} \times T$.
	\end{proof}

	\subsection{Vafa--Witten invariants of surfaces} Let $(S,\oO_S(1))$ be a smooth, polarized projective surface and $\pi \colon W = \mathrm{Tot}(K_S) \to S$ the total space of its canonical bundle with polarization $\pi^\ast \oO_S(1)$. Let $T = \Gm$ be the one-dimensional torus.
	
	By the results of \cite{VafaWittenI, VafaWittenII}, there exists a $(-1)$-shifted symplectic derived Artin stack $\nN$, which parametrizes Gieseker semi-stable Higgs pairs $(E, \phi \colon E \to E \otimes K_S)$ on $S$ with fixed positive rank and Chern classes, fixed determinant $\det E$ and zero trace $\mathrm{tr}\ \phi = 0$. $\nN$ is a moduli stack of compactly supported Gieseker semi-stable sheaves on $W$ with respect to the polarization $\pi^\ast \oO_S(1)$. It admits a natural $T$-action induced by scaling the fibers of the projection map $W \to S$.
	
	By construction, the classical truncation $\nN_\cl$ is a GIT quotient stack $[Q^{\mathrm{ss}} / G]$, where $Q$ is an appropriate Quot scheme on $W$ and the $T$-action is induced by a $T$-action on $Q^{\mathrm{ss}}$, which commutes with the $G$-action on $Q^{\mathrm{ss}}$. Moreover, the $T$-fixed locus $\nN_\cl^T$ has a good moduli space $N_\cl^T$ which is a projective scheme.

	We now see that, since $\nN^{\max}$ is $T$-invariant by Proposition~\ref{prop 7.5}, $\widetilde{\nN}$ admits a $T$-action such that the projection $\widetilde{\nN} \to \nN$ is $T$-equivariant. Since $N_\cl^T$ is proper, it follows that $\widetilde{\nN}_\cl$ is a proper Deligne--Mumford stack. Moreover, Lemma~\ref{propagation of property B}, Corollary~\ref{cor 6.15} and Theorem~\ref{thm trunc perf} apply $T$-equivariantly and imply that $\tau^{[0,1]}\mathbb{T}_{\widetilde{\nN}}|_{\widetilde{\nN}_\cl}$ is a $T$-equivariant perfect complex on $\widetilde{\nN}_\cl$ and the same is true for the $T$-fixed component of the restriction to the fixed locus $\widetilde{\nN}_\cl^T$, i.e., $\tau^{[0,1]} \mathbb{T}_{\widetilde{\nN}}|_{{\widetilde{\nN}_\cl}^T}^{\fix}$, which we know is the same as $\tau^{[0,1]} \mathbb{T}_{\widetilde{\nN}^T}|_{{\widetilde{\nN}_\cl}^T}$. Write $N^\vir \coloneqq \tau^{[0,1]} \mathbb{T}_{\widetilde{\nN}}|_{{\widetilde{\nN}_\cl}^T}^{\mv}$ for the moving component.
	
	In particular, by \cite{KiemLoc, KiemSavvasLoc}, the $T$-fixed locus ${\widetilde{\nN}_\cl}^T $ admits a virtual fundamental cycle $[ \widetilde{\nN}_\cl^T ]^\mathrm{vir} $ and also a virtual structure sheaf $[\oO_{\widetilde{\nN}_\cl^T}^\vir] \in K_0^T(\widetilde{\nN}_\cl^T)$. Thus, mirroring the virtual torus localization formula \cite{GrabPand, yplee}, we can give the following definition.
	
	\begin{definition} \label{Def: gen VW inv}
		The numerical generalized Vafa--Witten invariant via Kirwan blow-ups associated to $\nN$ is defined by the formula
		\begin{align*}
			\mathrm{VWK}(\nN) = \int_{ \left[ \widetilde{\nN}_\cl^T \right]^\mathrm{vir} } \frac{1}{e ( N^\vir )}.
		\end{align*}
		
		For any $\beta \in K_T^0( \widetilde{\nN}_\cl^T ) \otimes_{\mathbb{Z}[t,t^{-1}]} \mathbb{Q}( t^{\frac{1}{2}} )$, the $\beta$-twisted $K$-theoretic generalized Vafa--Witten invariant via Kirwan blow-ups associated to $\nN$ is defined by the formula
		\begin{align*}
			\mathrm{VWK}^{\mathrm{K-th},\beta}(\nN) = \chi_t \left( \widetilde{\nN}_\cl^T, \frac{[\oO_{\widetilde{\nN}^T_\cl}^\vir]}{e(N^\vir)} \otimes_{\oO_{\widetilde{\nN}_\cl^T}} \beta \right) \in \mathbb{Q}(t^{\sfrac{1}{2}}).
		\end{align*}
		Here $t$ denotes the torus parameter and $\chi_t$ has the same meaning as in \cite[Subsection~2.4]{RefinedVafaWitten}.
	\end{definition}
	
	It is interesting to investigate whether this definition satisfies the wall-crossing formulas of \cite{VafaWittenII, JoyceWC}. Deformation invariance of $\mathrm{VWK}(\nN)$ follows from a $T$-equivariant version of the arguments in \cite{Sav}.
	
	\begin{remark}
		The reason for the localization to $\mathbb{Q}(t^{\sfrac{1}{2}})$ above is to preserve consistency with the notation used in \cite{RefinedVafaWitten}, as well as the classical case, where it might be necessary to consider half-integer weights for the torus action. This happens already with the square root of the virtual canonical bundle in the classical case.
	\end{remark}
	
	\begin{remark}
		Observe that $\nN$ does not have connected stabilizers, so we have to perform the derived stabilizer reduction procedure with non-connected stabilizers in this case.
	\end{remark}

	\bibliographystyle{dary}
	\bibliography{refs}	
\end{document}